\documentclass[reqno]{amsart}  
\usepackage{amsmath,amssymb,amsthm}
\theoremstyle{plain}
\newtheorem{theorem}{Theorem}[section]
\newtheorem{lemma}[theorem]{Lemma}
\newtheorem{remark}[theorem]{Remark}
\newtheorem{proposition}[theorem]{Proposition}

\newtheorem{corollary}[theorem]{Corollary}
\newtheorem{example}[theorem]{Example}

\numberwithin{equation}{section}
\allowdisplaybreaks[1]

\theoremstyle{definition}

\newtheorem{definition}[theorem]{Definition}

\theoremstyle{remark}

\newcommand{\bA}{{\mathbf A}}
\newcommand{\bB}{{\mathbf B}}
\newcommand{\bC}{{\mathbf C}}
\newcommand{\bD}{{\mathbf D
}}
\newcommand{\bU}{{\mathbf U}}
\newcommand{\To}{\Psi}
\newcommand{\Bo}{{\mathbf B}}
\newcommand{\Do}{{\mathbf D}}

\newcommand{\bk}{{\mathbf k}}

\newcommand{\cA}{{\mathcal A}}
\newcommand{\cE}{{\mathcal E}}
\newcommand{\cG}{{\mathcal G}}
\newcommand{\cH}{{\mathcal H}}
\newcommand{\cL}{{\mathcal L}}
\newcommand{\cM}{{\mathcal M}}
\newcommand{\cN}{{\mathcal N}}
\newcommand{\cO}{{\mathcal O}}

\newcommand{\cU}{{\mathcal U}}

\newcommand{\cX}{{\mathcal X}}
\newcommand{\cY}{{\mathcal Y}}
\newcommand{\bbD}{{\mathbb D}}

\newcommand{\C}{{\mathbb C}}

\newcommand{\bin}{{\mu}}

\newcommand{\bc}[1]{\left(\begin{matrix} #1
                \end{matrix}\right)}
\newcommand{\bcs}[1]{\left(\begin{smallmatrix} #1
                \end{smallmatrix}\right)}
\newcommand{\Ob}{\boldsymbol{{\mathfrak O}}}
\newcommand{\Gr}{{\boldsymbol{\mathfrak G}}}

\begin{document}

\title[Weighted Bergman spaces]{Weighted Bergman spaces:
shift-invariant subspaces and input/state/output linear systems}
\author[J. A. Ball]{Joseph A. Ball}
\address{Department of Mathematics,
Virginia Tech,
Blacksburg, VA 24061-0123, USA}
\email{joball@math.vt.edu}
\author[V. Bolotnikov]{Vladimir Bolotnikov}
\address{Department of Mathematics,
The College of William and Mary,
Williamsburg VA 23187-8795, USA}
\email{vladi@math.wm.edu}

\begin{abstract}
         It is well known that subspaces of the Hardy space over the unit
         disk which are invariant under the backward shift occur as the
         image of an observability operator associated with a
         discrete-time linear system with stable state-dynamics, as well as
         the functional-model space for a Hilbert space contraction
         operator, while forward shift-invariant subspaces have a 
	 representation in terms of an inner function.
	 We discuss several variants of these statements in 
	 the context of weighted  Bergman spaces on the unit disk.
\end{abstract}


\maketitle

\section{Introduction}  \label{S:Intro}
\setcounter{equation}{0}

Let $H^{2}$ be the standard Hardy space over the open unit disk 
$$
H^{2} = \{ f(z) = \sum_{n=0}^{\infty} f_{n}z^{n} \colon 
\sum_{n=0}^{\infty} |f_{n}|^{2} < \infty\}
$$ and we let $S \colon f(z) 
\mapsto z f(z)$ be the shift operator on $H^{2}$.  The classical 
theorem of Beurling \cite{Beurling} asserts that any $S$-invariant 
subspace $\cM$ of $H^{2}$ has the form $\cM = \theta \cdot H^{2}$ 
where $\theta$ is an inner function (analytic on the disk with 
boundary-value function on the unit circle having modulus 1 almost 
everywhere).  The result was extended to the vector-valued case by Lax 
\cite{Lax} (for the finite-dimensional case) and by Halmos \cite{Halmos} 
(for the general case). The main thrust of Beurling's paper was the 
development of a theory of inner-outer factorization for $H^{2}$ and 
$H^{\infty}$-functions; this theory was then used to arrive at the 
famous characterization of invariant subspaces $\cM = \theta \cdot 
H^{2}$. 

\smallskip

There are now a number of more operator-theoretic and/or 
system-theoretic proofs which better handle the vector-valued case.
For $\cY$ a Hilbert space, we write 
$H^{2}(\cY)$ for the Hardy space $H^{2}(\cY) = H^{2}\otimes \cY$ 
of $\cY$-valued functions and we let $\cM$ be a shift-invariant 
subspace of $H^{2}(\cY)$.  We single out four distinct approaches 
toward what we shall call simply the Beurling-Lax theorem for the vector-valued case.

\medskip
    
{\bf (1)}  Find a Hilbert space $\cX$ (playing the role of a {\em 
    state space} from the system theory point of view) and construct 
    operators $C \colon \cX \to \cY$ and $A \colon \cX \to \cX$ so 
    that $\left[ \begin{smallmatrix} A \\ C \end{smallmatrix} 
    \right]$ is isometric ($A^{*}A + C^{*}C = I_{\cX}$) and so that 
    $\cM^{\perp}$ has the representation
 \begin{equation}  \label{Mperprep}
 \cM^{\perp} = \{ C (I - zA)^{-1} x \colon x \in \cX\}.
 \end{equation}
 One such choice is $\cX = \cM^{\perp}$ with $C = E|_{\cM^{\perp}}$ 
 (where $E \colon f \mapsto f(0)$ is evaluation at zero and $A = 
 S^{*}|_{\cM^{\perp}}$). We note that $H^{2}(\cY)$ is a reproducing 
 kernel Hilbert space with reproducing kernel given by the Szeg\H{o} 
 kernel $\frac{I_{\cY}}{1 - z \overline{\zeta}}$.
  Then $\cM^{\perp} \subset H^{2}(\cY)$, as a subspace of 
  $H^{2}(\cY)$, is also a  reproducing kernel Hilbert space in its 
  own right.  One can identify the reproducing kernel for 
  $\cM^{\perp}$ explicitly as
  $$ \cM^{\perp} = \cH(K_{C,A})\quad\text{where}\quad K_{C,A} = C (I - 
  zA)^{-1} (I - \overline{\zeta} A^{*})^{-1} C^{*}.
  $$
  It then follows that $\cM$ has Beurling-Lax-Halmos representation 
  $\cM = \Theta \cdot H^{2}(\cU)$ where $\Theta$ is any solution of 
  the kernel factorization problem
  \begin{equation}   \label{kernelfact}
      \frac{ \Theta(z) \Theta(\zeta)^{*}}{1 - z \overline{\zeta}} = 
  \frac{I_{\cY}}{1 - z \overline{\zeta}} - K_{C,A}(z,\zeta)
  \end{equation}
  from which it follows that the reproducing kernel $k_{\cM}(z, 
  \zeta)$ for the subspace $\cM$ is
  \begin{equation}  \label{kerM}
      k_{\cM}(z, \zeta) = \frac{\Theta(z) \Theta(\zeta)^{*}}{ 1 - z 
      \overline{\zeta}}.
   \end{equation}
  One way to construct such a $\Theta$ is as 
  \begin{equation}   \label{realTheta}
  \Theta(z) = D + z C (I - zA)^{-1} B
  \end{equation}
  where $\left[ \begin{smallmatrix} B \\ D 
\end{smallmatrix} \right]$ solves the Cholesky factorization problem
$$
  \begin{bmatrix} B \\ D \end{bmatrix} \begin{bmatrix} B^{*} & D^{*} 
      \end{bmatrix} = \begin{bmatrix} I & 0 \\0 & I \end{bmatrix} - 
      \begin{bmatrix} A \\ C \end{bmatrix} \begin{bmatrix} A^{*} & 
	  C^{*} \end{bmatrix}.
$$
We refer to this approach as the {\em de Branges-Rovnyak--Potapov 
approach}, as the first part (the identification of $\cM^{\perp}$ as 
$\cH(K_{c,A}$ ) is prominent in the work of de Branges-Rovnyak 
\cite{dBR1, dBR2} while the second step (factorization of the kernel 
\eqref{kernelfact}) is prominent in the Potapov-school approach to 
interpolation (see \cite{Potapov}).  

\smallskip

{\bf (2)} Set $A = P_{\cM} S^{*}|_{\cM}$ and $C = P_{\cE}|_{\cM}$, $\cU = 
\cM \ominus S \cM$ and consider the operator 
$\cO_{C,A} \colon \cM \subset H^{2}(\cY) \to H^{2}(\cU)$ given by
$$
  \cO_{\C,A} \colon x \mapsto C (I - zA)^{-1} x.
$$
Then the inner representer $\Theta$ for $\cM$ can be obtained as
$$
  M_{\Theta} = \left( \cO_{C,A}\right)^{*} \colon H^{2}(\cU) \to 
  H^{2}(\cY).
$$
This approach works more generally for $S$-invariant subspaces of 
$H^{2}(\cY)$ contractively included in $H^{2}(\cY)$.
We refer to this approach as the {\em Rosenblum-Rovnyak approach}, as 
this is amounts to the proof of the Beurling-Lax Theorem in the book 
\cite{RRbook}.

\smallskip

{\bf (3)} Note that $\cM$ decomposes as $\cM = \bigoplus_{k \ge 0} 
\left(S^{k} \cM \ominus S^{k+1} \cM\right)$.  It turns out that the 
reproducing kernel for the subspace $S^{k} \cM \ominus S^{k+1} \cM$ 
has the form 
\begin{equation}  \label{KerSkm-Sk+1M}
 k_{S^{k} \cM \ominus S^{k+1} \cM}(z, \zeta) = z^{k} 
 \overline{\zeta}^{k} \Theta(z) \Theta(\zeta)^{*}
\end{equation}
with $\Theta(z)$ independent of $k$. Thus the operator of 
multiplication by $z^{k} \Theta(z)$ maps the coefficient space $\cU$ 
isometrically onto $S^{k} \cM \ominus S^{k+1} \cM$ and we recover 
the kernel function for $\cM$ to be
$$
k_{\cM}(z, \zeta) = \sum_{k=0}^{\infty} z^{k} \overline{\zeta}^{k} 
\Theta(z) \Theta(\zeta)^{*}
$$
in agreement with \eqref{kerM}.  It now follows that $\Theta$ is a 
Beurling-Lax representer for $\cM$.  We refer to this as the {\em iterated Halmos 
wandering-subspace approach} since the construction of the inner 
representer via looking at the wandering subspace $\cM \ominus S \cM$ 
is the main idea of the proof for the abstract shift-operator setting 
in \cite{Halmos}.

\smallskip

{\bf (4)} Construct $\Theta$ so that $M_{\Theta} \cU = \cM \ominus S 
\cM$, or $k_{\cM \ominus S \cM} = \Theta(z) \Theta(\zeta)^{*}$ (i.e., 
$\Theta$ is the same as in Approach 3).  
In this shift-operator context, then the multiplication 
operator $M_{\Theta}$ extends to be an isometry from $H^{2}(\cU)$ 
into $H^{2}(\cY)$ with $M_{\Theta} H^{2}(\cU) = \cM$ and once again 
$k_{\cM}(z, \zeta) = \Theta(z)\left( \frac{ I_{\cU} }{1 - z 
\overline{\zeta}} \right) \Theta(\zeta)^{*}$, so once again $\Theta$ 
is the Beurling-Lax representer for $\cM$. 
We refer to this as the {\em uniterated Halmos wandering subspace 
approach}.  

\smallskip

In all four approaches, one ultimately arrives at the realization 
formula \eqref{realTheta} for the inner function $\Theta$ where 
the system matrix (also called operator colligation) $\bU = \left[ 
\begin{smallmatrix} A & B \\ C & D \end{smallmatrix} \right]$ is 
    unitary.  This has the interpretation that the associated 
    input/state/output linear system
\begin{equation}   \label{sys}
\Sigma_{\bf U}: \quad     \left\{ \begin{array}{rcl} x(n+1) & = & A x(n) + B u(n) \\
y(n) & = & C x(n) + D u(n) \end{array} \right.
\end{equation}
is {\em conservative}: the energy stored by the state at time $n$ 
($\| x(n+1)\|^{2} - \| x(n) \|^{2}$) is exactly compensated by the 
net energy put into the system from the outside environment 
($\|u(n)\|^{2} - \| y(n)\|^{2}$), with a similar property for the 
adjoint system.  We note that application of the $Z$-transform
$$
  \{x(n)\}_{n \in {\mathbb Z}_{+}} \mapsto \widehat x(z): = 
  \sum_{n=0}^{\infty} x(n) z^{n}
$$
converts the system equations \eqref{sys} to
\begin{equation}   \label{freqsys}
    \begin{array}{rcl} z^{-1} \left( \widehat x(z) - x(0) \right) & = 
	& A \widehat x(z) + B \widehat u(z) \\
	\widehat y(z) & = & C \widehat x(z) + D \widehat u(z)
	\end{array}
\end{equation}
which can then be solved for $\widehat x(z)$ and $\widehat y(z)$:
$$  \begin{array}{rcl}
\widehat x(z) & =&  (I - zA)^{-1} x(0) + z (I - zA)^{-1} B \widehat u(z)  \\
\widehat y(z) & = &
 \left(\cO_{C,A}x(0)\right)(z) + \Theta_{\bU}(z) \cdot \widehat u(z)
 \end{array}
$$
where 
\begin{equation}  \label{obsop}
    \cO_{C,A} \colon x \mapsto C (I - zA)^{-1} x
\end{equation}
is the {\em observability operator} for the system $\Sigma_{\bU}$ 
(generating the $Z$-transform $\widehat y(z)$ of the output signal 
from the initial state $x(0) = x$ when the input signal is taken to 
be zero), and where
\begin{equation}   \label{transfunc}
    \Theta_{\bU}(z) =  D + z C (I - zA)^{-1} B
 \end{equation}
 is the {\em transfer function} of the system $\Sigma_{\bU}$ (having 
 the property that multiplication by $\Theta_{\bf U}$ on the 
 $Z$-transform $\widehat u(z)$ of the input signal generates the 
 $Z$-transform $\widehat y(z)$ of the output when the initial state 
 $x(0)$ is taken to be zero).  We note that the observability 
 operator $\cO_{C,A}$ appears in the representation \eqref{Mperprep} 
 for $\cM^{\perp}$ and that the formula \eqref{realTheta} exhibits 
 the Beurling-Lax representer for $\cM$ as the transfer function for the 
 conservative i/s/o linear system $\Sigma_{\bU}$.
  Many of these ideas connecting  operator-valued 
  $H^{\infty}$-functions with systems theory ideas can be found 
  already in the survey paper of Ball-Cohen \cite{BC}.

\smallskip

It was only much later that researchers began looking for analogues 
of Beurling-Lax representations for shift-invariant subspaces of weighted 
Bergman spaces (see \cite{HKZ, DS} and the references there for the 
scalar-valued case).  
For a Hilbert space $\cY$ and an integer $n\ge 1$, we denote by $\cA_n(\cY)$ the 
reproducing kernel Hilbert space
with reproducing kernel
\begin{equation}
\bk_n(z,\zeta)=\frac{1}{(1-z\overline{\zeta})^n}I_\cY
\label{0.2}
\end{equation}
which is clearly positive on $\bbD\times \bbD$. The space $\cA_n(\cY)$ can be alternatively 
characterized as the Hilbert space of $\cY$-valued functions $f$  analytic 
in the open unit disk $\bbD$ and with finite norm $\|f\|_{A_n(\cY)}$:
\begin{equation}
\cA_n(\cY)=\left\{f(z)={\displaystyle\sum_{j\ge 0}f_j z^j} \colon 
\|f\|^2_{\cA_n(\cY)}:={\displaystyle \sum_{j\ge 0}\bin_{n,j} \cdot 
\|f_j\|_{\cY}^2}<\infty\right\}
\label{0.1}
\end{equation}
where the weights $\mu_{n,k}$'s are defined by 
\begin{equation}
\bin_{n,j}:=\frac{1}{\binom{j+n-1}{j}} =
\frac{j!(n-1)!}{(j+n-1)!}.
\label{defmu}
\end{equation}
We let $S_{n}$ denote the Bergman shift operator of weight-index $n$ 
acting on $\cA_{n}(\cY)$ given by multiplication by the coordinate 
functions:  $S_{n} \colon f(z) \mapsto z f(z)$ for $f \in 
\cA_{n}(\cY)$.
The space $A_1(\cY)$ is therefore the 
standard Hardy space $H^2(\cY)$ of the unit disk with $S_{1} = S$ 
equal to the standard Hardy-space shift operator, and for $n\ge 2$ the space $\cA_n(\cY)$ is 
the standard weighted Bergman space. 

\smallskip

Even for the case $\cY = {\mathbb C}$, one can no longer use Blaschke 
products to collect the zeros of a Bergman-space function since the 
zero set of a Bergman space function need not satisfy the Blaschke 
growth condition.  A major advance came with the work of Hedenmalm 
\cite{hend1} who constructed Bergman-inner functions, i.e., a 
function $\theta$ so that $\theta {\mathbb C} = \cA_{2} \ominus S_{2} 
\cA_{2}$ as the solution of an extremal problem.  Such a $\theta$ has 
the {\em contractive divisor property} $\| \theta^{-1} f \|_{\cA_{2}} 
\le \| f \|_{\cA_{2}}$, thereby improving earlier results of Horowitz 
\cite{Hor}.  Since the work of Apostol-Bercovici-Foias-Pearcy 
\cite{ABFP} it is known that shift-invariant subspaces $\cM \subset 
\cA_{2}$ can have arbitrary index ${\rm ind} \, \cM:= \dim \left( \cM \ominus S_{2} 
\cM\right)$.  Nevertheless, the seminal work of 
Aleman-Richter-Sundberg \cite{ars} with later extensions by Shimorin 
\cite{sh1} showed that in all cases we recover $\cM$ as $\cM = 
\overline{\rm span}_{k \ge 0} S_{n}^{k} (\cM \ominus S_{n} \cM)$ (a 
partial analogue of the representation of $\cM$ in approach 4 above), 
at least for the cases $n=2,3$.  More recently, Olofsson \cite{oljfa, 
olieot, olaa} initiated the study of operator-valued Bergman-inner 
functions for standard weighted Bergman spaces as an object 
of independent interest. In particular he suggested a certain 
 time-invariant input/state/output linear system of higher order (past values of the 
 states and of the inputs enter the state-update equations rather 
 than just the current state and input values) as the time-domain 
 explanation for the input-output map corresponding to multiplication 
 by the Bergman inner function.
 
\smallskip

 The purpose of the present paper is to further enhance the theory of 
 Beurling-Lax representations and to sharpen the connections with 
 the theory of input/state/output  linear systems for the 
 Bergman-space setting.  In particular, we show that each of the four 
 approaches to Beurling-Lax representations sketched above (which 
 blur into each other in the classical case) actually lead to four 
 distinct kinds of theories for the Bergman-space setting.  Our 
 results for the first three approaches are new for the Bergman-space 
 setting, while the fourth approach is most closely connected with the notion of Bergman 
 inner function as appearing in the work of Hedenmalm, Olofsson, and 
 Shimorin. Specifically, for $\cM \subset \cA_{n}(\cY)$, the first approach leads to the 
 representation for $\cM$ as 
$$
\cM = M_{F} \left( \bigoplus_{j=1}^{n} 
 \cA_{j}(\cU_{j}\right),\quad\mbox{where}\quad F(z) = \begin{bmatrix} F_{1}(z) & \cdots & 
 F_{n}(z) \end{bmatrix}
$$ 
and where $F_{j}$ is a bounded multiplier from 
 $\cA_{j}(\cU_{j})$ into $\cA_{n}(\cY)$ for each $j=1, \dots, n$ with 
 the additional property that the multiplication operator 
 $\begin{bmatrix} M_{F_{1}} & \cdots & M_{F_{n}} \end{bmatrix}$ is a 
 partial isometry from $\bigoplus_{j=1}^{n} \cA_{j}(\cU_{j})$ into 
 $\cA_{n}(\cY)$ (see Theorem \ref{T:BL}).  The second approach leads 
 to the identification of certain conditions when one has a 
 representation $\cM = M_{F_{n}} \cA_{n}(\cU_{n})$ (i.e., one can 
 take $F_{1}, \dots, F_{n-1}$ all equal to zero in the previous 
 representation) and applies to the more general situation where 
 $\cM$ has its own norm and is contractively included in 
 $\cA_{n}(\cY)$.  The third approach leads to the construction of a 
 {\em Bergman inner family} $\Theta_{0}, \Theta_{1}, \Theta_{2}, 
 \dots$ of multipliers, where $ z^{ k} \Theta_{k}(z)$ maps the 
 coefficient space $\cU_{k}$ isometrically into $\cA_{n}(\cY)$ so 
 that   $\begin{bmatrix} M_{\Theta_{0}} & M_{\Theta_{1}} & M_{\Theta_{2}} & 
 \cdots \end{bmatrix}$ maps the time-varying Hardy space 
 $\bigoplus_{k=0}^{\infty} z^{k} \cU_{k}$ isometrically onto $\cM 
 \subset \cA_{n}(\cY)$.  This perhaps is the most compelling new 
 Beurling-Lax representation for the Bergman-space setting and has 
 the most striking connections with systems theory.  Namely, the 
 multiplication operator $\begin{bmatrix} M_{\Theta_{0}} & 
 M_{\Theta_{1}} & M_{\Theta_{2}} & 
 \cdots \end{bmatrix}$ can  be identified as the  $Z$-transform  of 
 the input-output map of a certain time-varying linear system (see 
formula  \eqref{m1} below) having certain additional metric properties 
 (see \eqref{isom}, \eqref{wghtcoisom} and Section \ref{S:syscon}).
 When one specializes to the classical case $n=1$, this time-varying 
 linear system collapses to the time-invariant linear system 
 \eqref{sys} with transfer function equal to the Beurling-Lax 
 representer $\Theta(z)$ in the classical case. 
 
\smallskip

 To make these system theory connections precise, we delve into the 
 structure of observability operators and observability gramians, 
 operator resolvents, and transfer functions more general than in the 
 classical case.  Namely, observability operators of the form
 \begin{equation}   \label{cOn}
 \cO_{n,C,A} \colon x \mapsto C (I - zA)^{-n} x = 
 \sum_{j=0}^{\infty} \bcs{ j+n-1 \\ j} (A^{j} x) z^{j}
 \end{equation}
 as well as a $k$-shifted version
 $$
  \Ob_{n,k,C,A} \colon x \mapsto \sum_{j=0}^{\infty} \bcs{j+k + n-1 
  \\ j+k } (A^{j} x) z^{j},
 $$
 as well as functions ${\boldsymbol \Theta}$ having a realization 
 of the form
 \begin{equation}   \label{boldTheta}
     {\boldsymbol \Theta}(z) = D + z \Ob_{n,k,C,A}(z) B
 \end{equation}
 come up.  Indeed such formulas appear already in the work of 
 Olofsson \cite{olaa}, at least for the cases $k=0,1$.  We show that 
 $\cO_{n,C,A}$ \eqref{cOn} arises as the observability operator for the 
 time-varying system \eqref{m1} and that operator-valued functions of
 the form ${\boldsymbol \Theta}$ \eqref{boldTheta} arise naturally in the construction 
 of kernel functions for subspaces of the form $S_{n}^{k} \cM \ominus 
 S_{n}^{k+1} \cM$ (where $\cM$ is an $S_{n}$-invariant subspace of 
 $\cA_{n}(\cY)$).
 
\smallskip

 As preliminaries to the exposition of these ideas, in Section 
 \ref{S:powerseries} we develop a 
 calculus of shifted power geometric series 
as preparation for their use in the  corresponding operator functional  calculus. 
 In Section \ref{STM} we present the form of the time-varying system 
 \eqref{m1} and develop its behavior under the $Z$-transform for the 
 general setting where no metric constraints are imposed.  Section 
 \ref{S:obs} extends standard stability notions 
to the Bergman-setting where the system has the 
form \eqref{m1}.  Section \ref{S:NC-Obs} identifies the 
reproducing-kernel structure on ranges of observability operators; a 
key result is Theorem \ref{T:frakOkernel} which identifies the 
reproducing kernel for a subspace of $\cA_{n}(\cY)$ of the form 
$S_{n}^{k} \operatorname{Ran} \Ob_{n,k,C,A}$.  The next Section 
\ref{S:metric} develops the metric properties for the system 
operators in the system \eqref{m1} which will be needed for the 
Bergman-space Beurling-Lax representations to come.  Finally Section 
\ref{S:BL} develops our Beurling-Lax representation theorems for the 
Bergman-space setting (all four approaches) while the concluding 
Section \ref{S:syscon} makes the connections with the linear system 
\eqref{m1} (with the additional metric constraints imposed) precise.
 
 \section{Power-series representations for generalized geometric series}
\label{S:powerseries}
 
 We start with recording an assortment of power series  expansions  which will   play a key role 
in the sequel. We let  $S_{1}^{*}$ be the standard 
backward shift operator acting on  formal power series according to
    \begin{equation}   \label{S1*}
	S_{1}^{*} \colon \sum_{j=0}^{\infty} a_{j} z^{j} \mapsto 
	\sum_{j=0}^{\infty} a_{j+1} z^{j} \text{ or, equivalently, }
	S_{1}^{*} \colon f(z) \mapsto \frac{ f(z) - f(0)}{z}.
\end{equation}
Let us introduce the notation
\begin{equation}   \label{apr1}
 R_{n}(z)= (1-z)^{-n}\quad\mbox{and}\quad  R_{n,k}(z)= \left(S_{1}^{*}\right)^{k} (1-z)^{-n},
\end{equation}
so that $R_{n,0}(z) = R_{n}(z)$, and  record the power series expansions
\begin{equation}
 R_{n}(z)=\sum_{j=0}^{\infty} \bcs{n+j-1\\ j}\cdot z^{j},\qquad
R_{n,k}(z)=\sum_{j=0}^{\infty} \bcs{n+j+k-1\\ j+k}\cdot z^{j}.
\label{07}
\end{equation}
The first  representation follows from successive  
  term-by-term differentiation of the geometric series $(1-z)^{-1} =
 \sum_{j\ge 0} z^{j}$ whereas the second follows from
the first and the definition \eqref{S1*}  of $S_{1}^{*}$. 
It is seen from \eqref{07}  that  for the special case $n=1$,
  \begin{equation}   \label{n=1}
      R_{1,k}(z) = R_{1}(z) \quad\text{for all}\quad k=0,1, \dots
   \end{equation}
The formula
 \begin{equation} 
R_{n,k}(z)=\sum_{\ell=1}^{n} \bcs{\ell+k-2 \\ \ell - 1}R_{n-\ell+1}(z)
\quad\mbox{for}\quad k\ge 1
\label{07-1}
\end{equation}
follows from \eqref{07} and the Chu-Vandermonde identity for binomial coefficients (see
e.g.~\cite[page 50]{Loehr})
    \begin{equation}   \label{ChuVan}  
\bcs{n+j+k-1\\ j+k} = \sum_{\ell = 1}^{n} \bcs{\ell + k-2\\ \ell-1}\cdot \bcs{n+j-\ell\\ j},
\end{equation}   
according to which indeed
\begin{align*}
R_{n,k}(z)=\sum_{j=0}^{\infty} \sum_{\ell=1}^{n} \bcs{\ell+k-2 \\ \ell - 1}
    \bcs{n+j-\ell \\ j} z^{j}
    & = \sum_{\ell=1}^{n} \bcs{\ell + k-2\\ \ell - 1 }\cdot \left(
    \sum_{j=0}^{\infty} \bcs{n+j-\ell\\ j} z^{j} \right) \\
    & = \sum_{\ell=1}^{n} \bcs{\ell + k - 2\\ \ell -1}\cdot R_{n-\ell+1}(z)
\end{align*}
where we made use of \eqref{07} again in the last step. An easy corollary of definitions \eqref{apr1} is 
the following.  
\begin{lemma}   \label{L:basic}
      The functions $R_{n,k}(z)$ given by \eqref{apr1} satisfy the
      recursion
  \begin{equation}   \label{difquotRnk}
    R_{n,k}(z) = \bcs{n+k-1 \\ k} + z R_{n,k+1}(z).
  \end{equation}
  \end{lemma}
  \begin{proof}  Since $R_{n,k+1}(z) = \left( S_{1}^{*}\right)^{k+1}
      (1 -z)^{-n} = S_{1}^{*} \left( (S_{1}^{*})^{k} (1-z)^{-n}
      \right) = S_{1}^{*} R_{n,k}(z)$, we have
      $$ R_{n,k+1}(z) = S_{1}^{*} R_{n,k}(z) = [ R_{n,k}(z) -
      R_{n,k}(0)]/z.
      $$
      From the expansion for $R_{n,k}(z)$ in \eqref{07} we see that 
      $R_{n,k}(0) = \bcs{n+k-1 \\ k}$ and
      the formula \eqref{difquotRnk} follows.   
        \end{proof}
	
The following analogue of the well-known formula  
$\sum_{j=0}^{N} z^{j} = \frac{1 - z^{N+1}}{1-z}$ for the sum of the 
truncated geometric series will be useful in the sequel.

\begin{proposition}   \label{P:trunBergseries}
    The truncation of the infinite series representation for $R_{n}(z)$ 
    in \eqref{07} has the explicit summation formula:
    \begin{equation}   \label{trunBergseries}
	\sum_{j=0}^N\bcs{n+j-1 \\ j}z^j = 
	\frac{1}{(1-z)^n}-\sum_{j=1}^{n}\bcs{N+n\\ 
	N+j}\frac{z^{N+j}}{(1-z)^{j}}.
\end{equation}
 \end{proposition}
 
 \begin{proof}  The following computation
     \begin{align*}
\sum_{j=0}^N\bcs{n+j-1 \\ j}z^j&=\frac{1}{(n-1)!}\cdot 
\frac{d^{n-1}}{dz^{n-1}}\left(\frac{1-z^{N+n}}{1-z}\right)\\
&=\frac{1-z^{N+n}}{(1-z)^n}-\frac{1}{(n-1)!}\cdot 
\sum_{j=0}^{n-2}\bcs{n-1 \\ j}\frac{j!(N+k)!}{(N+j+1)!}\frac{z^{N+j+1}}{(1-z)^{j+1}}\\
& \quad \text{ (by the Leibnitz rule for the derivative of  a 
product)} \\
&=\frac{1-z^{N+k}}{(1-z)^k}-\sum_{j=0}^{n-2}\bcs{N+n\\ N+j+1}\frac{z^{N+j+1}}{(1-z)^{j+1}}\\
&=\frac{1}{(1-z)^n}-\sum_{j=1}^{n}\bcs{N+n\\ N+j}\frac{z^{N+j}}{(1-z)^{j}}
\text{ (by the shift $j \mapsto j+1$)}  
\end{align*}
verifies the result.  \end{proof}

Proposition \ref{P:trunBergseries} enables us to express the 
shifted $n$-th-power geometric series $R_{n,k}(z)$ in terms of unshifted 
$n$-th-power geometric series $R_{n}(z)$ as follows.

\begin{proposition}  \label{P:RnkRn}
    The shifted $n$-resolvent $R_{n,k}(z)$ is recovered from the 
    unshifted $n$-resolvent $R_{n}(z)$ according to the formula
 \begin{equation}   \label{RnkRn}
 R_{n,k}(z) = R_{n}(z) \cdot \sum_{\kappa = 0}^{n-1} \left( 
  \sum_{j = 0}^{\kappa}(-1)^{j}  \bcs{k+n-1 \\ k} \bcs{n-1-\kappa + j \\ j} 
  \right) z^{\kappa}  \end{equation} 
\end{proposition}

\begin{proof}  By definitions \eqref{apr1}, 
\begin{align}
R_{n,k}(z)&=z^{-k}\cdot \left(R_n(z)-\sum_{j=0}^{k-1}\bcs{n+j-1\\ j} z^j\right)\notag\\
&=z^{-k}R_n(z)\cdot \left(1-(1-z)^n\cdot\sum_{j=0}^{k-1}\bcs{n+j-1\\ j} z^j\right).\label{sep2}
\end{align}
Letting $N=k-1$ in \eqref{trunBergseries} and combining this formula with \eqref{sep2}
gives
\begin{align*}
R_{n,k}(z)&=R_n(z)\cdot \sum_{j=0}^{n-1}\bcs{k+n-1\\ k} z^j(1-z)^{n-j-1}  \\
&= R_{n}(z) \cdot \sum_{j=0}^{n-1} \sum_{\ell=0}^{n-j-1} \bcs{k+n-1 
\\ k } \bcs{ n-j-1 \\ \ell } (-1)^{\ell} z^{j+\ell}  \\
& = R_{n}(z) \cdot \sum_{\kappa = 0}^{n-1} \left( \sum_{j=0}^{\kappa} 
\bcs{k+n-1 \\ k} \bcs{n-1-j \\ \kappa - j} (-1)^{\kappa -j} \right) 
z^{\kappa} \\
& = R_{n}(z) \cdot \sum_{\kappa = 0}^{n-1} \left( \sum_{j=0}^{\kappa} 
\bcs{k+n-1 \\ k } \bcs{ n-1-\kappa + j \\ j } (-1)^{j} \right) 
z^{\kappa}
\end{align*}
and \eqref{RnkRn} follows.
\end{proof}

	As a consequence of identities \eqref{07}, \eqref{07-1}, 
	\eqref{difquotRnk}, \eqref{RnkRn}, we have the following identities for any 
  Hilbert-space operator $A$ having spectral radius less than one:
  \begin{align}
R_{n}(zA)&=(I - zA)^{-n} = \sum_{j=0}^{\infty} 
      \bcs{n+j-1\\ j} A^{j} z^{j}, \label{operator-series} \\
       R_{n,k}(zA)&  = \sum_{j=0}^{\infty} \bcs{n+j+k-1\\ j+k} A^{j}  z^{j},\label{st41}\\
R_{n,k}(zA)&=\sum_{\ell = 1}^{n} \bcs{\ell+k-2\\ \ell - 1}R_{n-\ell+1}(zA)\quad\mbox{for}\quad k\ge 1,
\label{st42}\\
 R_{n,k}(zA) &= \bcs{n+k-1 \\ k} I_\cX+ zAR_{n,k+1}(zA), 
 \label{st4}   \\
 R_{n,k}(zA) & = R_{n}(zA) \cdot \sum_{\kappa = 0}^{n-1} \left( 
  \sum_{j = 0}^{\kappa} (-1)^{j} \bcs{k+n-1 \\ k} \bcs{n-1-\kappa + j \\ j} 
   \right) z^{\kappa} A^{\kappa}.
 \label{RnkRnzA}
  \end{align}
  
  \section{System theory motivation}   \label{STM}

We have seen in the Introduction (see formulas 
\eqref{sys}--\eqref{transfunc}) how a formula of the type
\begin{equation}   \label{time-inv-transfunc}
  \Theta_{\Sigma_{1}}(z) = D + z C(I - zA)^{-1} B = D + z C R_{1}(zA)  B.
 \end{equation}
 arises as the transfer function of a discrete-time time-invariant 
 input/state/output linear system of the form \eqref{sys}.
Furthermore, imposition of the condition that 
the colligation matrix $\left[ \begin{smallmatrix} A & B \\ C & 
D \end{smallmatrix} \right]$ be contractive  (respectively, unitary) together with an 
additional stability condition on the state-update operator $A$ leads to the transfer 
function $\Theta$  being in the Schur class, i.e., having contractive values on the unit 
disk (respectively, being inner, i.e.,  contractive on the unit disk with 
unitary nontangential boundary values on the unit circle almost 
everywhere).  

\smallskip

In the quest for a parallel system-theoretic interpretation for 
Bergman inner functions on the unit disk, Olofsson \cite{oljfa, olaa} 
has shown that realization formulas of the type $\Theta(z) = D + z 
CR_{n}(zA)B$ involving higher-order resolvents $R_{n}(zA)$ arise naturally  
and he also associated a certain higher-order 
time-invariant linear system having such a $\Theta(z)$ as transfer 
function.  We shall see below that functions of the form $\Theta(z) = D + z C R_{n,k}(zA)B$ 
involving shifted higher-order resolvents $R_{n,k}(zA)$ also arise naturally.  
To obtain a system-theoretic connection for functions involving such shifted 
higher-order  resolvents $R_{n,k}(zA)$, we propose to consider 
the following discrete-time time-varying linear system:
\begin{equation}  \label{m1}
 \Sigma_{n}\left( \left\{\left[ \begin{smallmatrix} A & B_{j} \\ C & D_{j} 
 \end{smallmatrix}\right]  \right\}_{j\in{\mathbb Z}_+} \right) \colon    \left\{ 
\begin{array}{rcl}
 x(j+1) & = & \frac{j+n}{j+1}\cdot Ax(j)+\binom{j+n}{j+1}\cdot B_ju(j), \\ [3mm]   
 y(j) & = & C x(j)+\binom{j+n-1}{j}\cdot D_ju(j)
 \end{array} \right.
 \end{equation}
where 
$$
A \in \cL(\cX), \; \; C \in \cL(\cX, \cY), \; \;
B_k \in \cL(\cU_k, \cX), \; \; D_k \in \cL(\cU_k, \cY)
$$
are given bounded linear operators acting between given Hilbert spaces 
$\cX$, $\cY$ and $\cU_k$ ($k\ge 0$).  We note that the case where $n=1$ and where the 
operators $B_{k} =  B$ and $D_{k} = D$ are taken independent of the time parameter 
 $k\in{\mathbb Z}_+$ reduces to the classical time-invariant case given 
 by \eqref{sys}. 

\smallskip

If we let the system evolve on the
nonnegative  integers $j \in {\mathbb Z}_{+}$, then the whole
trajectory $\{u(j), x(j), y(j)\}_{j \in {\mathbb Z}_{+}}$ is
determined from the input signal $\{u(j)\}_{j \in {\mathbb Z}_{+}}$
and the initial state $x(0)$ according to the formulas
\begin{align} 
x(j) &= \bcs{n+j-1\\ j}\cdot \left(A^{j} x(0)+\sum_{\ell=0}^{j-1}A^{j-\ell-1} B_\ell 
u(\ell)\right), 
\label{m2}\\
y(j) &= \bcs{n+j-1\\ j}\cdot \left(CA^{j} x(0)+\sum_{\ell=0}^{j-1}CA^{j-\ell-1} B_\ell u(\ell)
+D_ju(j)\right).\label{m3}
  \end{align}
Formula \eqref{m2} is established by simple induction arguments, while \eqref{m3} is 
obtained by substituting \eqref{m2} into the second equation in \eqref{m1}.

\smallskip

To write the $Z$-transformed version of the system-trajectory formula \eqref{m2},
we multiply both sides of \eqref{m2} by $z^{j}$ and sum over $j\ge 0$ to get, on account of 
\eqref{operator-series},
\begin{align}
\widehat x(z)& = \sum_{j=0}^{\infty} x(j) z^{j}\notag\\
&=\left(\sum_{j=0}^\infty \bcs{j+n-1\\ j}A^jz^j\right)x(0)+\sum_{k=1}^{\infty}
\left(\sum_{j=k}^\infty\bcs{j+n-1 \\ j} A^{j-k}z^j\right)B_{k-1}u(k-1)\notag\\
&=(I-zA)^{-n}x(0)+\sum_{k=1}^{\infty}z^k\left(\sum_{j=0}^\infty 
\bcs{j+k+n-1\\ j+k}A^jz^j\right)B_{k-1}u(k-1)\notag\\
&=(I-zA)^{-n}x(0)+\sum_{k=1}^{\infty}z^kR_{n,k}(zA)B_{k-1}u(k-1)\notag\\
&=(I-zA)^{-n}x(0)+\sum_{k=0}^{\infty}z^{k+1}R_{n,k+1}(zA)B_ku(k).\label{m4}
\end{align}
The same procedure applied to \eqref{m3} gives
\begin{align}
\widehat y(z)&=\sum_{k=0}^{\infty} y(k) z^{k}\notag \\
&=C(I-zA)^{-n}x(0)+\sum_{k=0}^{\infty}z^k \left(\bcs{k+n-1\\ k}D_k+ 
zCR_{n,k+1}(zA)B_k\right)u(k)\notag\\
&=\cO_{n,C,A}x(0)+\sum_{k=0}^{\infty}z^k \Theta_{n,k}(z)u(k), \label{y-hat}
\end{align}
where 
 \begin{equation}
\cO_{n,C,A} \colon \; x \mapsto
\sum_{j=0}^\infty\left(\bcs{n+j-1 \\ j}CA^jx\right) \, z^j=C(I-z A)^{-n} x
\label{0.9}
\end{equation}
is the $n$-observability operator and where 
 \begin{equation}
\Theta_{n,k}(z)=\bcs{k+n-1\\ k}D_k+
zCR_{n,k+1}(zA)B_k \qquad (k=0,1,\ldots) 
\label{m5}
\end{equation}
is the family of transfer functions.  We note that the transfer 
function $\Theta_{n,k}(z)$ encodes the result of a pulse input-vector 
$u$ being applied at time $j = k$:
$$
\widehat y(z) = \Theta_{n,k}(z) \cdot z^{k} u \quad\text{if}\quad x(0) = 0 
\quad\text{and}\quad u(j) = \delta_{j,k} u.
$$
In fact the functions $\Theta_{n,k}(z)$ could have been derived in 
this way and then one could arrive at input-output relation 
\eqref{y-hat} via superposition of all these time-$k$ impulse 
responses.   Note also that formula \eqref{m5} for the classical time-invariant case ($n=1$ and 
$B_k$, $D_k$  independent of $k$) reduces to the formula \eqref{time-inv-transfunc} for the 
classical transfer function $\Theta_{\Sigma_1}$ due 
to the identity \eqref{n=1}.

\smallskip

For the application to Bergman inner functions, it is natural to 
impose some additional metric constraints on the colligation matrices 
$\left[ \begin{smallmatrix} A & B_{k} \\ C & D_{k} \end{smallmatrix} 
\right]$; these are discussed in Section \ref{S:metric} below.
 
\section{Observability operators and 
gramians, Stein equalities and inequalities}  
\label{S:obs}

Formula \eqref{0.9} associates with any {\em output pair} $(C, A)$ 
(i.e., $C\in\cL(\cX,\cY)$ and $A\in\cL(\cX)$) the $n$-observability operator 
$\cO_{n,C,A}$. In case ${\mathcal O}_{n,C,A}$ is bounded as an operator from $\cX$ into $\cA_n(\cY)$, we 
say that the pair $(C,A)$ is {\em $n$-output-stable}.  
If $(C,A)$ is $n$-output stable, then the {\em $n$-observability 
gramian} 
\begin{equation}   \label{nobsgram}
{\mathcal G}_{n,C, A}:=(\cO_{n,C, A})^{*}\cO_{n,C,A}
\end{equation}
is bounded on $\cX$ and can be represented via the series
\begin{equation}
{\mathcal G}_{n,C, A}= \sum_{k=0}^\infty \bcs{k+n-1\\ k}
 A^{*k}C^*CA^k
\label{2.1}
\end{equation}
converging in the strong operator topology.  As suggested by the 
Agler hereditary functional calculus as formulated by 
Ambrozie-Engli\v{s}-M\"uller \cite{AEM}, we introduce the operator
\begin{equation} \label{defba}
B_{A} \colon X \mapsto A^{*} X A
\end{equation}
 mapping  $\cL(\cX)$ into itself, and then view $\cG_{n,C,A}$ (at 
 least formally) as being given by  
\begin{equation} \label{defGGamma}
\cG_{n,C,A} = (I - B_{A})^{-n}[C^{*}C].
\end{equation}
If $\| B_{A} \| < 1$, \eqref{defGGamma} is precise; in general one 
can make this precise by interpreting \eqref{defGGamma} in the form
\begin{equation}   \label{defGGamma'}
    \cG_{n,C,A} = \lim_{r \uparrow 1} (I - r B_{A})^{-n}[C^{*}C].
\end{equation}
Either of the formulas \eqref{defGGamma} and \eqref{defGGamma'}
suggests that we define
\begin{equation}   \label{G0}
    \cG_{0,C,A} = C^{*}C.
\end{equation}

We next introduce the operator map 
\begin{equation}   \label{2.3}
\Gamma_{n,A}=(I - B_{A})^{n} \colon \quad X\mapsto \sum_{k=0}^{n} (-1)^{k} \binom{n}{k} 
A^{*k} X A^{k}.
\end{equation}
There is then an assortment of identities as listed below.

\begin{lemma}\label{L:identities}
    \begin{enumerate}
\item For all $H,A \in \cL(\cX)$ and any integers $k\ge 1$ and $N\ge 0$,
\begin{align}
&  \Gamma_{k,A}[H] = \Gamma_{k-1,A}[H] - A^{*} \Gamma_{k-1,A}[H]A,\label{2.4}\\
&  H = \sum_{j=0}^N\bcs{k+j-1 \\ j}A^{*j}\Gamma_{k,A}[H]A^j+
\sum_{j=1}^{k}\bcs{N+k\\ N+j}A^{*N+j}\Gamma_{k-j,A}[H]A^{N+j}.        
\label{2.5}  
\end{align}
	\item If $(C,A)$ is $n$-output stable, then also
	\begin{align}
& \cG_{m, C,A} - A^{*} \cG_{m, C,A} A = \cG_{m-1,C,A}\quad\text{for}\quad m=1,\dots, n,
	\label{id0}  \\
&	\Gamma_{k,A}[\cG_{n,C,A}]  = \cG_{n-k,C,A} \quad\text{for}\quad k=0,\dots, n.
	\label{2.6}
  \end{align}
  Moreover in this case $(C,A)$ is also $k$-output stable for $k=1, 
  \dots, n-1$.
  \end{enumerate}
  \end{lemma}
  
  \begin{proof} 
Upon applying identity $(I - B_{A})^{k}=(I - B_{A})^{k-1}-B_A(I - B_{A})^{k-1}$ to an 
operator $H\in\cL(\cX)$ and making use of definition \eqref{2.3} we get \eqref{2.4}.
To verify \eqref{2.5}, we use the identity \eqref{trunBergseries} in the form
$$
\sum_{j=0}^N\bcs{k+j-1 \\ j}z^j=
\frac{1}{(1-z)^k}-\sum_{j=1}^{k}\bcs{N+k\\ 
N+j}\frac{z^{N+j}}{(1-z)^{j}}.
$$
Multiplying both parts in the latter identity by $(1-z)^k$ we get
\begin{equation}
1=\sum_{j=0}^N\bcs{k+j-1 \\ j}z^j(1-z)^k+\sum_{j=1}^{k}\bcs{N+k\\ N+j}z^{N+j}(1-z)^{N-j},
\label{weget}
\end{equation}
which in turn implies the operator identity
$$
I_{\cX}=\sum_{j=0}^N\bcs{k+j-1 \\ j}B_A^j(1-B_A)^k+\sum_{j=1}^{k}\bcs{N+k\\ 
N+j}B_A^{N+j}(1-B_A)^{N-j}.
$$
Upon applying this latter identity to an
operator $H\in\cL(\cX)$ and making use of definition \eqref{2.3} we get \eqref{2.5}.

\smallskip

We provide two proofs of \eqref{id0} and \eqref{2.6}, one of which is 
a straightforward easy-to-remember computation but which is rigorous 
only if the spectral radius $\rho(B_{A})$ of $B_{A}$ is less than 1,  together with a second more elaborate 
proof to handle the general case.
We have from \eqref{defba} and \eqref{defGGamma} 
    \begin{align*}
	\cG_{n,C,A} - A^{*} \cG_{m, C,A} A & = (I - 
	B_{A})[\cG_{m,C,A}] \\
	& = (I - B_{A}) \circ (I - B_{A})^{-m}[C^{*}C] \\
	& = (I - B_{A})^{-m+1}[C^{*} C ]= \cG_{m-1, C,A}
	\end{align*}
	and \eqref{id0} follows, at least for the case where $\| 
	B_{A} \| < 1$. To handle the general case (where we only 
	assume that $(C,A)$ is $n$-output stable),  one can plug in 
	the infinite series representation \eqref{2.1} for 
	$\cG_{m,C,A}$ and make use of the binomial coefficient 
	identity $ \binom{m}{k} = \binom{m-1}{k} + \binom{m-1}{k-1}$ to arrive at the result.
	A corollary of the computation is that the infinite series 
	defining $\cG_{m-1,C,A}$ is strongly convergent, i.e., 
	$(C,A)$ is also $(m-1)$-output stable whenever it is $m$-output 
	stable.
	
\smallskip 
      
To prove \eqref{2.6} under the assumption that  $\rho(B_{A}) < 1$, note that
\begin{align*}
\Gamma_{j,A}[\cG_{n,C,A}] & = (I - B_{A})^{j}[ (I - 
B_{A})^{-n}[C^{*}C]] \\
& = (I - B_{A})^{-n+j}[C^{*}C]= \cG_{n-j,C,A}.
\end{align*}
To handle the general case, one can do an inductive argument using identity 
\eqref{id0} to arrive at the result.  As a corollary we arrive at the 
final statement in Lemma \ref{L:identities}:  $(C,A)$ is $k$-output stable for $k=1, \dots, n-1$ 
whenever $(C,A)$ is $n$-output stable.
\end{proof}

\begin{definition}
The operator $A\in\cL(\cX)$ is called {\em $n$-contractive} if
$\Gamma_{n,A}[I]$ is positive semidefinite
and it is called {\em $n$-hypercontractive} if
$\Gamma_{k,A}[I]\ge 0$ for all $0\le k\le n$.
\label{D:0}
\end{definition}
It was shown in \cite{aglerhyper} (see also \cite{muller} as well as 
\cite{MV} for a multivariable version) that inequalities
$\Gamma_{1,A}[I]\ge 0$ and $\Gamma_{n,A}[I]\ge 0$ 
imply that $A$ is an $n$-hypercontraction. This result extends from 
$I$ to an arbitrary $H\ge 0$; the proof below is modelled from the one in \cite{muller}.
\begin{lemma}
Let  us assume that the operators $H, \, A\in\cL(\cX)$ are such that  
\begin{equation}
H\ge A^*HA\ge 0\quad\mbox{and}\quad\Gamma_{n,A}[H]\ge 0  
\label{sq1}
\end{equation}
for some integer $n\ge 3$. Then
\begin{equation}
\Gamma_{k,A}[H]\ge 0\quad\mbox{for all}\quad k=1,\ldots,n-1.
\label{sq1a}
\end{equation}
\label{squeeze}
\end{lemma}

\begin{proof}
Observe that the leftmost inequalities in \eqref{sq1} mean that $H$ and $\Gamma_{1,A}[H]$ are both 
positive semidefinite. Making use of definitions \eqref{defba} and  \eqref{2.3} we have
\begin{align}
\sum_{j=0}^NA^{*j}\Gamma_{n-1,A}[H]A^j&=\sum_{j=0}^NB_A^j(I-B_A)^{n-1}[H]\notag\\
&=(I-B_A^{N+1})(I-B_A)^{n-2}[H]\notag\\
&=\Gamma_{n-2,A}[H]-A^{*N+1}\Gamma_{n-2,A}[H]A^{N+1}.\label{sq2}
\end{align}
Iterating the first inequality in \eqref{sq1} gives $A^{*j}HA^j\le H$ for all $j\ge 0$
and therefore,
\begin{equation}
\left|\left\langle \Gamma_{n-2,A}[H]A^{k}x, \, A^{k}x\right\rangle\right|
\le \sum_{j=0}^{n-2}  \bcs{ n-2 \\ j} 
\langle Hx, \, x\rangle=2^{n-2}\langle Hx, \, x\rangle.
\label{sq3}
\end{equation}
Taking the inner product of both parts in \eqref{sq2} against $x\in\cX$ and then making use of
\eqref{sq3} gives
\begin{align}   
\left|\sum_{j=0}^N \left\langle \Gamma_{n-1,A}[H]A^jx, \, A^jx\right\rangle\right|&=
\left| \left\langle \Gamma_{n-2,A}[H]x, \, x\right\rangle-
 \left\langle \Gamma_{n-2,A}[H]A^{N+1}x, \, A^{N+1}x\right\rangle\right|\notag\\
&\le 2^{n-1}\left\langle Hx, \, x\right\rangle.
\label{sq2a}
\end{align}
On account of relation \eqref{2.4}, the second inequality in \eqref{sq1} implies
$$
\Gamma_{n-1,A}[H] \ge A^{*} \Gamma_{n-1,A}[H]A.
$$
Therefore $\left\langle \Gamma_{n-1,A}[H]A^jx, \, A^jx\right\rangle\ge
\left\langle \Gamma_{n-1,A}[H]A^{j+1}x, \, A^{j+1}x\right\rangle$ for all $j\ge 0$ and
all $x\in\cX$. Thus, on the left hand side of \eqref{sq2a} we have the partial sum
of a non-increasing sequence and, since the partial sums are uniformly bounded   
(by $2^{n-1}\left\langle Hx, \, x\right\rangle$), it follows that all the terms in the sequence
are nonnegative. In particular, $\left\langle \Gamma_{n-1,A}[H]x, \, x\right\rangle\ge
0$. Since the latter inequality holds for every $x\in\cX$, we conclude that
$\Gamma_{n-1,A}[H]\ge 0$. We then obtain recursively all the desired inequalities in  
\eqref{sq1a}.  
\end{proof}
\begin{definition}
The operator $A\in\cL(\cX)$ is called {\em strongly stable} if 
$A^k$ tends to zero as $k\to\infty$ in the strong operator topology, i.e., 
$\|A^{k}x\|\to 0$ for every $x\in\cX$. 
\label{D:1}
\end{definition}

The following result gives connections between $n$-output stability,
observability gramians and solutions of associated Stein equations
and inequalities. In case $n=1$ it is well-known. In what follows, we will
refer to the last relations in \eqref{2.8} and \eqref{2.9} as  
the {\em Stein inequality} and the {\em Stein equality}, respectively.

\begin{theorem}
\label{T:1.1}
Let $C\in\cL(\cX,\cY)$ and  $A \in\cL(\cX)$. Then:
\begin{enumerate}
\item The pair $(C,A)$ is $n$-output-stable if and only if there exists
an $H\in \cL(\cX)$ satisfying the system of inequalities
\begin{equation}
H\ge A^*HA\ge 0\quad\mbox{and}\quad\Gamma_{n,A}[H]\ge C^*C.
\label{2.8}
\end{equation}
\item If $(C,A)$ is $n$-output-stable, then the observability gramian
${\mathcal G}_{n, C, A}$ satisfies the system
\begin{equation}
H\ge A^*HA\ge 0\quad\mbox{and}\quad\Gamma_{n,A}[H]= C^*C
\label{2.9}
\end{equation}
and is the minimal positive semidefinite solution of the system \eqref{2.8}.

\item There is a unique positive-semidefinite solution $H$ of the system \eqref{2.9}
with $H = \cG_{n,C,A}$ if $A$ is {\em strongly stable}.  If $A$ is a contraction, then the 
solution of the system \eqref{2.9} is unique if and only if $A$ is strongly stable.
           \end{enumerate}
           \end{theorem}

\begin{proof} Suppose first that $(C,A)$ is $n$-output-stable.  
Then the infinite series in \eqref{2.1} 
converges in the strong operator topology to the operator $H= 
{\mathcal G}_{n,C,A}\ge 0$ and by \eqref{id0}, \eqref{2.6} and \eqref{G0},
$$
{\mathcal G}_{n,C,A}-A^*{\mathcal G}_{n,C,A}A=
{\mathcal G}_{n-1,C,A}\ge 0, \quad \Gamma_{n,A}[{\mathcal G}_{n,C,A}]={\mathcal G}_{0,C,A}=C^*C.
$$
Thus, $H={\mathcal G}_{n,C,A}$ satisfies relations \eqref{2.9} 
and hence also inequalities \eqref{2.8}.

\smallskip

Conversely, suppose that inequalities \eqref{2.8}
are satisfied for some $H\in\cL(\cX)$. Then for every integer $N\ge 0$ we have
\begin{align}
\sum_{j=0}^N \bcs{j+n-1\\ j}A^{*j}C^*CA^j&\le \sum_{j=0}^N \bcs{j+n-1\\ j}A^{*j}\Gamma_{n,A}[H]A^j
\notag\\
&=H-\sum_{j=1}^{n}\bcs{N+n\\ N+j}A^{*N+j}\Gamma_{n-j,A}[H]A^{N+j}\le H.\label{au1}
\end{align}
Indeed, the first inequality follows since $C^*C\le \Gamma_{n,A}[H]$, the second equality 
holds by \eqref{2.5} (with $k=n$) and the last inequality holds since $\Gamma_{k,A}[H]\ge
0$ for $k=0,\ldots,n-1$, by the assumptions \eqref{2.8} and Lemma \ref{squeeze}.

\smallskip

By letting $N\to\infty$ in \eqref{au1} we conclude that the left-hand side sum
converges (weakly and therefore, since all the terms are
positive semidefinite, strongly) to a bounded positive semidefinite
operator. By \eqref{2.1},
$$
\lim_{N\to\infty}\sum_{j=0}^N \bcs{j+n-1\\ j}A^{*j}C^*CA^j
=\sum_{j=0}^\infty \bcs{j+n-1\\ j}A^{*j}C^*CA^j={\mathcal G}_{n,C,A}
$$
and passing to the limit in \eqref{au1} as $N\to \infty$ gives
${\mathcal G}_{n,C, A}\le H$.  In particular  the
operator ${\mathcal G}_{n,C,A}$ is bounded (since $H$ is) and
therefore the pair $(C,A)$ is $n$-output-stable. This completes the proof of the
first part  of the theorem.

\smallskip

As observed above, the gramian ${\mathcal G}_{n,C, A}$ satisfies relations
\eqref{2.9} (and therefore, relations \eqref{2.8}). We also showed that any operator 
$H$ solving the inequalities \eqref{2.8} also satisfies the inequality $H\ge {\mathcal G}_{n,C,A}$
so that ${\mathcal G}_{n,C, A}$ is indeed the minimal solution to the system of inequalities 
\eqref{2.8}. This completes the proof of the second part of the theorem.

\smallskip

Now suppose that $A$ is strongly stable and that $H\in\cL(\cX)$ solves the system \eqref{2.9}.
We will show that then necessarily $H={\mathcal G}_{n,C,A}$. We first recall that if 
positive semidefinite operators $P, \, Q\in\cL(\cX)$ satisfy the Stein equation 
\begin{equation}
P-A^*PA=Q
\label{2.15}
\end{equation}
with a strongly stable $A\in\cL(\cX)$, then $P$ is uniquely recovered from \eqref{2.15}
via the strongly converging series 
\begin{equation}
P={\displaystyle\sum_{k=0}^\infty A^{*k}QA^k}.
\label{2.15a}
\end{equation}
Indeed, iterating \eqref{2.15} gives 
$$
P=A^{*N}PA^N+\sum_{j=0}^{N-1}A^{*j}QA^j\quad\mbox{for all}\quad N\ge 0,
$$
and since all the terms in the latter equality are positive semidefinite, the convergence
of the series on the right side of \eqref{2.15a} follows. The strong stability of $A$ 
guarantees that $A^{*N}PA^N\to 0$ as $N\to \infty$, which implies equality in 
\eqref{2.15a}.
Now we observe that by \eqref{2.4}, the Stein equation in \eqref{2.9} can be written as
$$
\Gamma_{n-1,A}[H]-A^*\Gamma_{n-1,A}[H]A=C^*C
$$
and the above uniqueness shows that
$$
\Gamma_{n-1,A}[H]=\sum_{k=0}^\infty A^{*k}C^*CA^k={\mathcal
G}_{1,C,A}.
$$
The latter equality can be in turn written in the form \eqref{2.15}
with $P=\Gamma_{n-2,A}[H]$
and $Q={\mathcal G}_{1,C,A}$, and by the same uniqueness argument we have
$$
\Gamma_{n-2,A}[H]=\sum_{k=0}^\infty A^{*k}{\mathcal
G}_{1,C,A} A^k={\mathcal G}_{2,C,A}.
$$
Continuing this procedure, we get $\Gamma_{n-j,A}[H]={\mathcal G}_{j,C,A}$ for $j=1,\ldots,n$.
For $j=n$ we have in particular, $H=\Gamma_{0,A}[H]={\mathcal G}_{n,C,A}$,
which gives the desired uniqueness.

\smallskip

To complete the proof of part (3) of the theorem it 
remains to  show that if $A$ is a contraction and the system \eqref{2.9} admits a unique 
solution (which necessarily is $H={\mathcal G}_{C,A,n}$), then the operator $A$ is strongly 
stable. We prove the contrapositive:  {\em if $A$ is not strongly stable, then the solution 
of \eqref{2.9} is not unique.}
         
\smallskip

Since $A$ is a contraction, the sequence of operators $\Delta_k=A^{*k}A^k$
is decreasing and bounded below and therefore has a strong limit
$\Delta\ge 0$ which clearly satisfies the relation $B_A[\Delta]=\Delta$ and which
is not zero, since $A$ is assumed not to be strongly stable. Then it follows that
$$
\Gamma_{k,A}[\Delta]=(I-B_A)^k[\Delta]=\sum_{j=0}^k(-1)^j\bcs{k \\ j}\Delta=0
$$
for all $k\in\{1,\ldots,n\}$.
Now we see that the operator $H={\mathcal G}_{n,C,A}+\Delta$
(as well as ${\mathcal G}_{n,C,A}$) satisfies the system \eqref{2.9}
which therefore has more than one positive-semidefinite solution.
\end{proof}

\begin{definition}
A pair $(C,A)$ with $C\in\cL(\cX,\cY)$ and $A\in\cL(\cX)$ is called 
$n$-{\em contractive} if inequalities \eqref{2.8} hold with $H = I_{{\mathcal X}}$,
that is, if $A$ is $n$-hypercontractive and in addition 
$$
\Gamma_{n,A}[I_{_\cX}]\ge C^*C, 
\quad \text{or equivalently} \quad  
C^*C+\sum_{j=1}^n \bcs{n \\ j}A^{*j}A^j\le I_{_\cX}.
$$
The pair $(C,A)$ will be called $n$-{\em isometric} if relations \eqref{2.9} hold 
with $H = I_{{\mathcal X}}$,
that is, if $A$ is $n$-hypercontractive and in addition
$$
\Gamma_{n,A}[I_{\cX}]=C^*C, \quad \text{or equivalently} \quad
C^*C+\sum_{j=1}^n \bcs{n \\ j}A^{*j}A^j= I_{_\cX}.
$$
\label{D:2}
\end{definition}
A pair $(C,A)$ is called {\em observable} if the operator $\cO_{1,C,A}$
(equivalently, ${\mathcal G}_{1,C, A}$) is
injective. This property means that a state space vector $x\in\cX$ is uniquely recovered
from the output string $\{y_k\}_{k\ge 0}$ generated by running the
system \eqref{m1} with  the initial condition
$x_0=x$ and the zero input string. A pair $(C,A)$ is called {\em exactly observable} if
$\cO_{1,C,A}$ (equivalently, ${\mathcal G}_{1, C, A}$) is bounded and bounded
from below. More generally, we say that the pair $(C,A)$ is {\em 
exactly $n$-observable} if the $n$-gramian $\cG_{n,C,A}$ is bounded 
and bounded below. While $n$-observability implies $k$-observability 
for $1 \le k < n$, the corresponding statement for exact 
$k$-observability fails (see Proposition \ref{P:exactobs} below).  

\smallskip

The following statements follow along the lines of Theorem \ref{T:1.1}.

\begin{proposition}  \label{P:2-1.1}
$(1) \;$ Suppose that $(C, A)$ is an $n$-contractive pair.  Then $(C, A)$ is
$n$-output-stable with ${\mathcal G}_{n,C, A} \le I_{_{\mathcal X}}$ 
and the observability gramian ${\mathcal G}_{n,C, A}$ is the unique 
solution of system \eqref{2.9} if and only if $A$ is strongly stable.

\smallskip

$(2) \;$ Suppose that $(C, A)$ is an $n$-isometric pair.  Then $H = I_{_\cX}$
is the unique solution of the system \eqref{2.9} if and only if $A$ is strongly stable.  In 
this case ${\mathcal O}_{n,C, A}$ is isometric and hence also $(C,A)$ is exactly observable.
\end{proposition}
Let us say that  {\em the pair $(C, A)$ is similar to the pair 
$(\widetilde C, \widetilde A)$} if there is an invertible operator $T$ on $\cX$ so that
$\widetilde C = C T^{-1}$ and $\widetilde A = T AT^{-1}$.
Then we have the following characterization of pairs
$(C,A)$ which are similar to an $n$-contractive or to an $n$-isometric pair.

\begin{proposition}  \label{P:similar}
$(1) \;$ The pair $(C,A)$ is similar to an $n$-contractive pair $(\widetilde C, \widetilde A)$ 
if and only if there exists a bounded, strictly positive-definite solution $H$ to the
 system of inequalities \eqref{2.8}.

\smallskip

$(2) \;$ The pair $(C, A)$ is similar to an $n$-isometric pair if and only if there exists a 
bounded, strictly positive-definite solution $H$ of the system  \eqref{2.9}.
\end{proposition}

\begin{proof}
Suppose that $H$ is a strictly positive-definite solution of \eqref{2.8}.   Factor 
$H$ as $H = T^{*}T$ with $T$ invertible and set
\begin{equation}  \label{tilde-pair}
\widetilde C = C T^{-1}, \qquad \widetilde A = T AT^{-1}.
\end{equation}
Multiplying each inequality in \eqref{2.8} on the left by $T^{*-1}$ and on the right
by $T^{-1}$ and replacing $A$ and $C$ respectively by $T^{-1}\widetilde{A}T$ and 
$\widetilde{C}T$ then leads us to 
$$
I_{\cX}\ge \widetilde{A}^*\widetilde{A} \quad \mbox{and} \quad
\Gamma_{n,\widetilde{A}}[I_{\cX}]\ge \widetilde{C}^*\widetilde{C} 
$$
which means that $(\widetilde C, \widetilde A)$ is an  $n$-contractive
pair which is similar to the original pair $(C, A)$.  Conversely, if
$(\widetilde C, \widetilde A)$ given by \eqref{tilde-pair} is
contractive, then $H = T^{*}T$ is bounded and strictly positive-definite and
satisfies the inequalities \eqref{2.8}. This verifies the first statement of
the Proposition.  The second statement follows in a similar way.
         \end{proof}
	 
	 As a consequence of the observations in Proposition \ref{P:similar},
Proposition \ref{P:2-1.1} can be formulated more generally as follows.

\begin{proposition}  \label{P:2-1.1'}
$(1) \;$ If the pair $(C, A)$ is such that inequalities \eqref{2.8} have a strictly 
positive-definite solution $H$, then $(C,A)$ is $n$-output-stable. Moreover, the 
observability gramian ${\mathcal G}_{n,C,A}$ is the unique positive
semidefinite solution of the system \eqref{2.9} if and
only if $A$ is strongly stable.

\smallskip

$(2) \;$  If the pair $(C,A)$ is such that the system \eqref{2.9} has a strictly 
positive-definite solution $H$, then $(C,A)$ is $n$-output-stable and the observability 
gramian ${\mathcal G}_{n,C,A}$ is the unique positive semidefinite solution of the 
system \eqref{2.9} if and only if $A$ is strongly stable.  In this
case $(C, A)$ is moreover exactly observable.
\end{proposition}

The last part of Proposition \ref{P:2-1.1'} has a converse.
\begin{proposition} \label{P:2-1.1'converse}
Suppose that the pair $(C,A)$ is $n$-output-stable and
exactly observable.  Then $A$ is strongly stable.
\end{proposition}

\begin{proof} Observe that if $H$ is any solution to the system \eqref{2.9}, then 
\eqref{au1} takes the form
\begin{equation}
\sum_{j=0}^N \bcs{j+n-1\\ j}A^{*j}C^*CA^j
=H-\sum_{j=1}^{n}\bcs{N+n\\ N+j}A^{*N+j}\Gamma_{n-j,A}[H]A^{N+j}
\label{au2}   
\end{equation}
and still holds for all $N\ge 0$. 
If $(C, A)$ is $n$-output-stable and exactly observable,
then the observability gramian ${\mathcal G}_{n,C,A}$
is a strictly positive-definite solution of the system \eqref{2.9}.
Plugging $H={\mathcal G}_{n,C,A}$ into \eqref{au2} gives
\begin{align}
{\mathcal G}_{n,C,A}=&\sum_{k=0}^N \bcs{k+n-1\\ k}A^{*k}C^*CA^k
+\sum_{j=1}^n \bcs{N+n \\
N+j}A^{* N+j}\Gamma_{n-j,A}[{\mathcal G}_{n,C,A}]A^{N+j}\notag\\
=&\sum_{k=0}^N \bcs{k+n-1\\ k}A^{*k}C^*CA^k
+\sum_{j=1}^n \bcs{N+n \\
N+j}A^{* N+j}{\mathcal G}_{j.C,A}A^{N+j}
\label{2.22}
\end{align}
where the last equality follows from \eqref{2.6}. From the infinite-series 
representation  \eqref{2.1} for ${\mathcal G}_{C,A,n}$, taking limits in
\eqref{2.22} gives
$$
\lim_{N \to \infty}\bcs{N+n \\ N+j}A^{* N+j}{\mathcal G}_{j,C,A}A^{N+j}=0
$$
for $j=1,\ldots,n$. In particular, we have
\begin{equation}
\lim_{N \to \infty}A^{* N+n}{\mathcal G}_{n,C,A}A^{N+n}=0.
\label{2.23}
\end{equation}
Since ${\mathcal G}_{n,C,A}$ is strictly positive definite, we conclude that 
there is an $\varepsilon > 0$ so that
$$
        \varepsilon  \| x \|^{2} \le \langle {\mathcal G}_{n,C,A} x,
        x \rangle\quad  \text{for all} \quad x \in {\mathcal X}.
$$
Upon combining \eqref{2.23} with the latter relations (with $x$ replaced by $A^Nx$) we get  
$$
\varepsilon \| A^Nx\|^2\le\langle {\mathcal G}_{n,C,A}A^{N} x, \, A^{N}x\rangle \to 0
$$
for all $x \in {\mathcal X}$, and we conclude that $A$ is strongly stable as asserted.
\end{proof}

\smallskip

We conclude this section with formal introduction of 
backward-shifted versions of the observability operator $\cO_{n,C,A}$ \eqref{0.9} and 
observability gramian $\cG_{n,C,A}$ \eqref{nobsgram}.  The latter operators 
can be expressed as 
$$
\cO_{n,C,A}: \; x\to C R_{n}(zA)x\quad\mbox{and}\quad \cG_{n,C,A} = R_{n}(B_{A})[C^{*}C]
$$
where $R_{n}(z)=(1-z)^{-n}$ and where $B_{A}$ is the operator on $\cL(\cX)$ given by 
\eqref{defba}. We now introduce the backward-shifted variants of these objects by simply 
replacing  the function $R_n$ in the two formulas above by its backward shifts $R_{n,k}$:
\begin{align}
&{\Ob}_{n,k,C,A}: \;  x \to C R_{n,k}(zA)x=
\sum_{j=0}^\infty\bcs{n+j+k-1\\ j+k}(CA^jx)z^j,
\label{defdelta}\\
&{\Gr}_{n,k,C,A}:= R_{n,k}(B_{A})[C^{*}C]=
\sum_{j=0}^\infty\bcs{n+j+k-1\\ j+k}A^{j*}C^*CA^j.
\label{defR}
\end{align}
Observe that the latter power series representations follow from that in 
\eqref{st41}. 
Letting $k=0$ in \eqref{defdelta}, \eqref{defR} we conclude
\begin{equation}
{\Ob}_{n,0,C,A}={\cO}_{n,C,A}\quad\mbox{and}\quad {\Gr}_{n,0,C,A}=\cG_{n,C,A}.
\label{st2}
\end{equation}
If $k\ge 1$, then we can use the formula \eqref{07-1} along with \eqref{defdelta}, \eqref{defR}
to get representations
$$
{\Ob}_{n,k,C,A}=\sum_{\ell= 1}^{n} \bcs{\ell+k-2\\ \ell - 1}\cO_{n-\ell+1,C,A},\qquad
{\Gr}_{n,k,C,A}=\sum_{\ell= 1}^{n} \bcs{\ell+k-2\\ \ell - 1}\cG_{n-\ell+1,C,A}
$$
which imply in particular that for an $n$-output stable pair $(C,A)$, 
the operator ${\Ob}_{n,k,C,A}: \, \cX\to \cA_n(\cY)$ is bounded for all $k\ge 0$.
\begin{proposition} \label{P:wghtSteinid} The weighted Stein identity
    \begin{equation}
A^*\Gr_{n,k+1,C,A}A+\bcs{n+k-1 \\ k}\cdot C^*C=\Gr_{n,k,C,A}
\label{7.3}
\end{equation}
holds for all integers $k\ge 0$.
\end{proposition}

\begin{proof}
The operatorial equality
$$
 R_{n,k}(B_A) - B_A\circ R_{n,k+1}(B_A) = R_{n,k}(0)\cdot 
 I_{\cL(\cX)} = \bcs{n+k-1 \\ k}\cdot 
I_{\cL(\cX)}
$$
is a consequence of the identity \eqref{difquotRnk}. Applying this equality to the operator 
$C^*C$ gives, on account of \eqref{defR},
\begin{align*}
\left( R_{n,k}(B_{A}) - B_{A}\circ R_{n,k+1}(B_{A}) \right)[C^{*}C]
&=\Gr_{n,k,C,A} - A^{*} \Gr_{n,k+1,C,A} A\\
&=\bcs{n+k-1 \\ k}\cdot C^*C
\end{align*}
and we arrive at \eqref{7.3} as wanted.  
\end{proof}

As a consequence of Proposition \ref{P:RnkRn}, it turns out that 
$\Gr_{n,k,C,A}$ can be expressed in terms of $\cG_{n,C,A}$ as 
follows.

\begin{proposition}   \label{P:gramRnkRn}
    The operator $\Gr_{n,k,C,A}$ given by \eqref{defR} admits the 
    representation
    \begin{equation}   \label{sep1}
	\Gr_{n,k,C,A} = \sum_{\kappa = 0}^{n-1} \left( 
	\sum_{j=0}^{\kappa} \bcs{ k+n-1 \\ k} \bcs{n-1-\kappa + j \\ 
	j} (-1)^{j} \right) A^{* \kappa} \cG_{n,C,A} A^{\kappa}.
	\end{equation}
 \end{proposition}
 
 \begin{proof}
     Note that the identity \eqref{RnkRn} leads to the operator 
     identity
     $$
     R_{n,k}(B_{A}) = 
      \sum_{\kappa = 0}^{n-1} \left( 
  \sum_{j = 0}^{\kappa} \bcs{k+n-1 \\ k} \bcs{n-1-\kappa + j \\ j} 
 (-1)^{j}  \right) (B_{A})^{\kappa} \circ R_{n}(B_{A}).
 $$
 Application of this operator to $C^{*}C$ now leads to the identity 
 \eqref{sep1} due to \eqref{defR}, \eqref{defba} and \eqref{defGGamma}.
  \end{proof}

\section{Observability-operator range spaces and reproducing
        kernel Hilbert spaces}\label{S:NC-Obs}

Let $S_n$ denote the shift operator on the space $\cA_n(\cY)$ defined as 
$S_n: \, f(z)\to zf(z)$. Reproducing kernel calculations show that 
its adjoint $S_n^*$ is given by 
\begin{equation}
S^*_n f=\sum_{j=0}^\infty\frac{j+1}{n+j}\cdot 
f_{j+1}z^j=\sum_{j=0}^\infty\frac{\mu_{n,j+1}}{\mu_{n,j}}\cdot 
f_{j+1}z^j\quad\mbox{if}\quad
f(z)=\sum_{j=0}^\infty f_jz^j.
\label{3.1}
\end{equation}
Iterating the latter formula gives 
\begin{equation}
S^{*k}_n f=\sum_{j=0}^\infty\frac{(k+j)!(n+j-1)!}{j!(n+k+j-1)!}\cdot f_{k+j}z^j
=\sum_{j=0}^\infty\frac{\mu_{n,j+k}}{\mu_{n,j}}\cdot f_{k+j}z^j.
\label{3.1a}
\end{equation}
The operator $S^*_n$ is a strongly stable $n$-hypercontraction. 
As was shown in \cite[Lemma 5.1]{olieot} for $f$ of the form \eqref{3.1},
\begin{equation}
\sum_{j=0}^n(-1)^j\bc{n \\ j}\|S_n^{*j}f\|^2_{\cA_n}=\|f_0\|^2.
\label{3.2}
\end{equation}
The latter equality means that the pair $(E,S_n^*)$ is $n$-isometric in the sense 
of Definition \ref{D:2}, where  $E: \, \cA_n(\cY)\to \cY$ is the evaluation operator
defined by $Ef=f(0)$. It follows from \eqref{3.1a} that 
\begin{equation}  \label{E*Snk}
ES^{*k}_n f=  \frac{j! (n-1)!}{(n+k-1)!}f_{k} =   \mu_{n,k} f_{k} 
\end{equation}
and therefore,
$$
{\mathcal O}_{n,E,S_n^*}f=
E(I-zS_n^*)^{-n}f=\sum_{j=0}^\infty  \bcs{j+n-1\\ j} \left(E
S_{n}^{*j}f\right) z^{j}= \sum_{j=0}^{\infty} f_{j} z^{j} = f
$$
so that the observability operator ${\mathcal O}_{n,E,S_n^*}$ equals the identity operator
on $\cA_n(\cY)$.

Associated with an $n$-output-stable pair $(C,A)$ is the range of the
observability operator
$$
\operatorname{Ran}{\mathcal O}_{n,C,A} =
\{ C (I - zA)^{-n}x \colon \; x \in {\mathcal X}\}.
$$
\begin{theorem} \label{T:1.2}
        Suppose that $(C,A)$ is an $n$-output-stable pair.  Then:
        \begin{enumerate}
	    
\item The intertwining relation 
\begin{equation}
S_n^*{\mathcal O}_{n,C, A}={\mathcal O}_{n,C, A}A
\label{3.3}
\end{equation}
holds and hence the linear manifold ${\mathcal M}=\operatorname{Ran} 
{\mathcal O}_{n,C,A}$  is $S_n^*$-invariant.

\item  Let $H\in\cL(\cX)$ satisfy inequalities \eqref{2.8}
and let ${\mathcal X}'$ be the completion of ${\mathcal X}$ with
inner product $\| [x]\|_{{\mathcal X}'}^{2} = \langle H x, x
\rangle_{{\mathcal X}}$ (where $[x]$ denotes the equivalence class
modulo $\operatorname{Ker} H$ generated by $x$).
Then $A$ and $C$ extend to define bounded operators
$A' \colon \cX' \to \cX'$ and $C' \colon {\mathcal X}' \to \cY$ and the observability operator 
${\mathcal O}_{n,C,A}$ extends to define a contraction operator
${\mathcal O}_{n,C', A'}: \, {\mathcal X}'\to \cA_n({\cY})$. Moreover,
${\mathcal O}_{n,C',A'}$ is an isometry if and only if $H$ satisfies relations \eqref{2.9}
and $A'$ is strongly stable,  i.e.,
$\langle H A^Nx, A^{N}x \rangle \to 0$ for all  $x \in\cX$ as $N\to\infty$.

\item If the linear manifold ${\mathcal M}:=
\operatorname{Ran} {\mathcal O}_{n,C,A}$ is given the lifted norm
$$
         \| {\mathcal O}_{n,C,A} x \|_{{\mathcal M}}^{2} =
        \inf_{y \in {\mathcal X} \colon {\mathcal O}_{n,C,A}y =
        {\mathcal O}_{n,C, A}x} \langle H y, y
         \rangle_{{\mathcal X}},
$$
then
\begin{enumerate}
\item ${\mathcal M}$ can be completed to ${\mathcal M}' =
\operatorname{Ran}{\mathcal O}_{n,C',A'}$ with contractive inclusion
in $\cA_n({\cY})$:
$$
\| f \|^{2}_{\cA_n} \le \| f \|^{2}_{{\mathcal M}'} \quad\text{for all}\quad f \in {\mathcal M}'.
$$
Furthermore, ${\mathcal M}'$ is isometrically equal to the reproducing kernel
Hilbert space with reproducing kernel $K_{n,C,A,H}$ given by
\begin{equation}  \label{3.4}
K_{n,C,A,H}(z, \zeta ) = C(I -z A)^{-n}H(I-\overline{\zeta} A^{*})^{-n} C^{*}.
\end{equation}

\item The operator $S_n^*: \, \cM\to \cM$ defined as in \eqref{3.1} is a contraction
and 
\begin{equation}
\sum_{j=0}^n(-1)^j\bc{n \\ j}\|S_n^{*j}f\|^2_{\cM}\ge \|f(0)\|^2_\cY
\quad\mbox{for all} \; \; f\in{\mathcal M}.\label{3.6}
\end{equation}
Moreover, if the Stein equality in \eqref{2.9} holds, then
\eqref{3.6} holds with equality.
\end{enumerate}

\item Conversely, if ${\mathcal M}$ is a Hilbert space contractively
included in $\cA_n({\cY})$ which is invariant under $S_n^{*}$ which in turn is a contraction on 
$\cM$ and for which \eqref{3.6} holds, then there is an $n$-contractive pair 
$(C,A)$ 
such that ${\mathcal M} = {\mathcal H}(K_{n,C,A,I}) = \operatorname{Ran}
{\mathcal O}_{n,C,A}$ isometrically. In case \eqref{3.6} holds with
equality, then $(C,A)$ can be taken to be $n$-isometric.
The canonical-model choice of such a pair $(C,A)$ is
$$
  (C,A) = (E|_{\cM}, S_{n}^{*}|_{\cM}) \text{ where } E \colon f \in 
  \cA_{n}(\cY) \mapsto f(0).
$$
\end{enumerate}
\end{theorem}

\begin{proof}[Proof of (1):]
Making use of power series expansion \eqref{07} and of 
\eqref{3.1} we get \eqref{3.3}:
\begin{align*}
S_n^*{\mathcal O}_{C, A,n}x=
S_n^*(C(I-zA)^{-n}x)&=\sum_{k=0}^\infty \frac{k+1}{n+k}\cdot\bc{n+k \\ k+1}(CA^{k+1}x)z^k\\
&=\sum_{k=0}^\infty \bc{n+k-1 \\ k}(CA^{k+1}x)z^k\\ &=C(I-zA)^{-n}Ax={\mathcal O}_{C, A,n}Ax.
\end{align*}
\noindent
{\em Proof of (2)}:
Inequalities \eqref{2.8} amount to the statement
that the pair $(C, A)$ is $n$-contractive  and well-defined on the
dense subset $[{\mathcal X}]$ of ${\mathcal X}'$ (where $[x]$ is the
equivalence class containing $x$) and hence extends to an
$n$-contractive pair $(C', A')$ on all of ${\mathcal X}'$
and moreover relations \eqref{au1} hold for all $N\ge 0$. 
By Theorem \ref{T:1.1}, $\cG_{n,C,A} x \le H$ and therefore
$$
\| \cO_{n,C,A}x\|_{\cA_{n}(\cY)} = \langle \cG_{n,C,A}x,\, 
x\rangle_{\cX} \le \langle H x, \, x \rangle_{\cX},
$$
so $\cO_{n,C,A}$ is contractive from $\cX$ to $\cA_n(\cY)$. It follows from \eqref{au1} that 
the limit
\begin{equation} \label{3.7}
\Delta_{H,A} = \lim_{N \to\infty} \sum_{j=1}^n \bcs{N+n \\
N+j}A^{* N+j}\Gamma_{n-j,A}[H]A^{N+j}
\end{equation}
exists in the strong sense.  We have from \eqref{au1}
\begin{align*}
0&\le \sum_{j=0}^N \bcs{j+n-1\\ j}A^{*j}\left(\Gamma_{n,A}[H]-C^*C\right)A^j\\
&\le H-\sum_{j=1}^n \bcs{N+n \\
N+j}A^{* N+j}\Gamma_{n-j,A}[H]A^{N+j}-\sum_{j=0}^N \bcs{j+n-1\\ j}A^{*j}C^*CA^j
\end{align*}
and letting $N\to \infty$ we get, on account of \eqref{3.7} and \eqref{2.1},
\begin{equation}
0\le \sum_{j=0}^\infty \bcs{j+n-1\\ j}A^{*j}\left(\Gamma_{n,A}[H]-C^*C\right)A^j
\le H-\Delta_{H,A}-\cG_{n,C,A}.
\label{au5}
\end{equation}
By definition, ${\mathcal O}_{n,C',A'} \colon\cX'\to \cA_n({\cY})$ being an isometry means
that ${\mathcal G}_{n,C, A} = H$ in which case \eqref{au5} forces $\Delta_{H,A} = 0$
and equalities throughout \eqref{au5}. Since all the terms on the right hand side of \eqref{3.7}
are positive semidefinite,  the condition $\Delta_{H,A}=0$ implies in particular that 
$$
0=\lim_{N \to\infty}A^{* N+n}\Gamma_{0,A}[H]A^{N+n}=\lim_{N \to\infty}
A^{* N+n}HA^{N+n}=\lim_{N \to\infty}A^{*N}HA^{N},
$$
so that $A'$ is strongly stable. All the terms in the series in \eqref{au5} are nonnegative and 
therefore, each term equals zero. The term corresponding to the index $j=0$ is 
$\Gamma_{n,A}[H]-C^*C$. Hence $\Gamma_{n,A}[H]=C^*C$ and thus the Stein equality in 
\eqref{2.9} holds.
Conversely, if $A'$ is strongly stable and relations \eqref{2.9} hold, we conclude as in the 
proof of Theorem \ref{T:1.1} that  ${\mathcal G}_{n,C', A'} =H$, i.e., that ${\mathcal 
O}_{n,C',A'}$ is  an isometry from ${\mathcal X}'$ into $\cA_n(\cY)$.

\medskip
\noindent
{\em Proof of (3a)}:  Statement (3a) follows from general
principles laid out in \cite{NFRKHS} (see also \cite{BC, oljfa} for 
applications very close to the context here).  For the sake of 
completeness and since the same construction arises again in the 
sequel, we sketch the argument here.  Note that for $x \in \cX$ and 
$y \in \cY$ we have
\begin{align*}
    \langle \left(\cO_{n,C,A}x \right)(\zeta), y 
    \rangle_{\cY} & = \langle C (I - \zeta A)^{-n}x, y \rangle_{\cY}   \\
 & = \langle x, H (I - \overline{\zeta}A^{*})^{-n} C^{*} y 
    \rangle_{\cX^\prime} \\
    & = \langle C (I - \cdot A)^{-n}x, C (I - \cdot A)^{-n} H (I 
    - \overline{\zeta} A^{*})^{-n} C^{*} y \rangle_{\cM}
 \end{align*}
  where the last step follows from the assumption that 
$\cO_{n,C,A}$ is isometric from $\cX^\prime$ into $\cM$.  From this 
 computation we see that 
 $$ \langle \left(\cO_{n,C,A}x\right)(\zeta), y \rangle_{\cY} 
 =\langle \cO_{n,C,A} x, K_{n,C,A,H}(\cdot, \zeta) y \rangle_{\cM}
 $$
 and it follows that $K_{n,C,A,H}$ is the reproducing kernel for 
 $\cM$ as asserted.

\medskip 
\noindent
{\em Proof of (3b)}:
For $f$ of the form $f(z) = C(I-zA)^{-n}x$, we have
$$
\| f \|^{2}_{{\mathcal H}(K_{n,C,A, H})} = \langle H x, x\rangle_{\cX} \quad\mbox{and}
\quad f(0)=Cx.
$$
Now it follows from \eqref{3.3} that 
$$
\| S_n^{*j} f \|^{2}_{{\mathcal H}(K_{n,C,A, H})} =
\langle H A^j x, A^j x \rangle_{\cX} \quad\text{for} \quad  j\ge 1.
$$
With these substitutions, we see that $S_n^*$ is a contraction on $\cM$ if and only if 
$\langle H x, x \rangle_{\cX}\ge \langle H A x, A x \rangle_{\cX}$ for all 
$x\in\cX$ or, in operator form, $H\ge A^*HA$, and that on another hand,  inequality \eqref{3.6} 
is equivalent to
$$
\sum_{j=0}^n(-1)^j\bc{k \\ j}\langle HA^jx, \, A^jx\rangle_{\cX}\ge \|Cx\|^2_\cY
\quad\mbox{for all} \; \; x\in\cX,
$$
or in operator form,
\begin{equation}
\sum_{j=0}^n(-1)^j\bc{k \\ j}A^{*j}HA^j\ge C^*C\label{3.10}
\end{equation}
with equality in \eqref{3.6} equivalent to equality in \eqref{3.10}. To complete verification of 
part (3b), it remains to remark that 
relation \eqref{3.10} coincide with that in \eqref{2.8} (or with that in \eqref{2.9}
in case equality holds in \eqref{3.10}).

\medskip
\noindent
{\em Proof of (4)}:
Suppose that ${\mathcal M}$ is a Hilbert space included in $\cA_n(\cY)$ which is invariant under 
$S_n^{*}$ which is contractive on $\cM$ and let us assume that the inequality 
\eqref{3.6} holds. Let $A=S_n^*\vert_{\cM}$ and let 
$C$ be defined by $Cf=f(0)$ for all $f\in\cM$. In other words, $C=E\vert_{\cM}$. Inequalities 
$S_nS_n^*\le I_{\cM}$ and \eqref{3.6} mean that the pair $(C,A)$ defined this way is $n$-contractive.
Moreover, if equality holds in \eqref{3.6} for every $f\in\cM$, then the pair $(C,A)$ is 
$n$-isometric. It is readily seen that ${\mathcal O}_{n.C, A}={\mathcal O}_{n,E, 
S_n^*}\vert_{\cM}=I_\cM$. Therefore, for each $f \in {\mathcal M}$ we have
$\|f\|_{\cH(K_{n,C,A,I})}=\|f\|_{{\mathcal M}}$ and thus ${\mathcal M}=\cH(K_{n,C,A,I})$ 
isometrically. It then follows from part (3a) of the theorem that in fact $\cM$ is contractively 
included in $\cA_n(\cY)$.
\end{proof}

The following corollary is a simple, useful special case of the 
general situation laid out in Theorem \ref{T:1.2}.

\begin{corollary}  \label{C:1.2}
    Suppose that the $n$-output stable pair $(C,A)$ is also exactly 
    $n$-observable.  Then there is an $S_{n}^{*}$-invariant subspace 
    $\cM \subset \cA_{n}(\cY)$ so that $(C,A)$ is similar to the 
    $n$-isometric pair $(\widetilde C, \widetilde A)$ given by
    $$ (\widetilde C, \widetilde A) = (E|_{\cM}, S_{n}^*|_{\cM}), 
    \quad\text{where}\quad E \colon f \in \cA_{n}(\cY) \mapsto f(0).
    $$
\end{corollary}

\begin{proof}  One can simply check directly that $\cO_{n,C,A} \colon 
    \cX \to \cA_{n}(\cY)$ considered as a linear transformation from 
    $\cX$ onto $\cM: = \operatorname{Ran} \cO_{n,C,A}$ provide the 
    required similarity transformation.
\end{proof}
    
As explained by part (4) of Theorem \ref{T:1.2}, for purposes of study of contractively-included, 
$S_n^*$-invariant subspaces of $\cA_n(\cY)$ which satisfy inequality \eqref{3.6}, 
without loss of generality we may suppose at the start that we are working with ${\mathcal X}'$ as 
the original state space ${\mathcal X}$ and with the solution $H$ of inequalities 
\eqref{2.8} to be normalized to $H = I_{{\mathcal X}}$. Then certain simplifications occur in 
parts (1)-(4) of Theorem \ref{T:1.2} as explained in the next result.

\begin{theorem}  \label{T:1.3}
Let $(C, A)$ be an $n$-contractive pair with $C\in\cL(\cX,\cY)$ and $A\in\cL(\cX)$. Then:
\begin{enumerate}
\item $(C,A)$ is $n$-output-stable and the intertwining relation \eqref{3.3} holds. Hence
         $\operatorname{Ran}  \cO_{n,C,A}$ is $S_n^*$-invariant.
\item The operator ${\mathcal O}_{n,C,A}: \cX\to \cA_n(\cY)$ is a contraction. 
         Moreover ${\mathcal O}_{n,C,A}$ is isometric if
         and only if $(C,A)$ is an $n$-isometric pair and $A$ is strongly stable.
\item If the linear manifold ${\mathcal M}:= \operatorname{Ran}
      {\mathcal O}_{n,C,A}$ is given the lifted norm
\begin{equation}  \label{3.12}
\|{\mathcal O}_{n,C,A}x\|_{{\mathcal M}}=\| Q x \|_{{\mathcal X}}
\end{equation}
where $Q$ is the orthogonal projection of ${\mathcal X}$ onto
$(\operatorname{Ker} {\mathcal O}_{n,C,A})^{\perp}$, then 
${\mathcal O}_{n,C,A}$ is 
a coisometry of ${\mathcal X}$ onto ${\mathcal M}$. Moreover, ${\mathcal M}$ is contained 
contractively in $\cA_n(\cY)$ and is isometrically equal to the  reproducing kernel Hilbert space 
${\mathcal H}(K_{n,C,A})$ with reproducing kernel $K_{n,C,A}(z,\zeta)$ given by
$$
K_{n,C,A}(z,\zeta) = C (I - z A)^{-n}(I-\overline{\zeta}A^{*})^{-n}C^{*}.
$$
\item If ${\mathcal M}=\operatorname{Ran}{\mathcal O}_{n,C,A}={\mathcal H}(K_{n,C,A})$ is given the lifted norm 
$\| \cdot \|_{{\mathcal H}(K_{n,C,A})}$ as in \eqref{3.12}, then $S_n^*\vert_{\cM}$ is a contraction and 
\begin{equation}
\sum_{j=0}^n(-1)^j\bcs{k \\ j}\cdot \|S_n^{*j}f\|^2_{{\mathcal H}(K_{n,C,A})}\ge \|f(0)\|^2_\cY
\quad\mbox{for all} \; \; f\in {\mathcal H}(K_{n,C,A}).\label{3.14}
\end{equation}
Moreover, \eqref{3.14} holds with equality if and only the orthogonal projection $Q$ of 
${\mathcal X}$ onto $(\operatorname{Ker} {\mathcal O}_{n,C,A})^{\perp}$ is subject to relations
\begin{equation}  
Q\ge A^*QA\quad \mbox{and}\quad \Gamma_{n,A}[Q]=C^*C.
\label{3.15}
\end{equation}
         In particular, if $(C,A)$ is observable, then
\eqref{3.14} holds with equality if and only if $(C,A)$ is an $n$-isometric pair.
         \end{enumerate}
         \end{theorem}
\begin{proof}  Statements (1)--(3) and all but the last part of
statement (4) are direct specializations to the case $H =I_{{\mathcal X}}$ of the 
corresponding results in Theorem \ref{T:1.2}. It remains only to analyze the conditions 
for equality in \eqref{3.14}.

\smallskip

From the intertwining relation \eqref{3.3}, we see that inequality $S_nS_n^*\le I_{\cM}$ and
the equality in \eqref{3.14} for a generic 
element $f = {\mathcal O}_{n,C,A}x \in\cM$ mean that
\begin{align*}
&\|{\mathcal O}_{n,C,A}x\|_{{\mathcal H}(K_{n,C,A})}\ge \|{\mathcal O}_{n,C,A}Ax\|_{{\mathcal H}(K_{n,C,A})},\\
&\sum_{j=0}^n(-1)^j \bcs{k\\ j} \cdot \|{\mathcal O}_{n,C,A}A^jx\|^2_{{\mathcal H}(K_{n,C,A})}=
\|Cx\|^2_{\cY}
\end{align*}
for all $x\in\cX$. By the definition \eqref{3.12} of the ${\mathcal H}(K_{n,C,A})$-norm, the latter relations 
can be written as
$$
\|Qx\|_{\cX}\ge \|QAx\|_{\cX}\quad\mbox{and}\quad
\sum_{j=0}^n(-1)^j \bcs{k\\ j}\cdot \|QA^jx\|^2_{\cX}=\|Cx\|^2_{\cY}
$$
which in turn is finally equivalent to \eqref{3.15}.    
\end{proof}

\begin{remark}[\textbf{Variations on a theme: shift-invariance of the range 
    of an observability operator}] {\em   A fundamental observation for 
    us is the fact that $\operatorname{Ran} \cO_{n,C,A}$ is invariant 
    under the $S_{n}^{*}$ due to the intertwining condition 
    \eqref{3.3}.  However, for the case of a backward-shifted observability operator 
$\Ob_{n,k,C,A}$, we see that $\operatorname{Ran} \Ob_{n,k,C,A}$ is not 
$S_{n}^{*}$-invariant; a partial substitute is the relation
$$
  S_{1}^{*} \Ob_{n,k,C,A}=\Ob_{n,k+1,C,A}A
$$
which follows directly from \eqref{st4}. We now list some additional 
curious connections between the ranges of $S_n^k\Ob_{n,k,C,A}$ and 
the operators $S_n$ and $S_n^*$; we shall not have need of any these in 
the sequel.

\begin{proposition}  \label{P:cutecon}
    If the pair $(C,A)$ is $n$-output stable, then the following 
    relations hold:
\begin{align}
& S_n^*S_n^{k}\Ob_{n,k,C,A}=S_n^{k-1}\Ob_{n,k-1,C,A}
\quad\quad\mbox{for all}\quad k\ge 1.\label{jul22}  \\
& S_n^{*m}S_n^{k}\Ob_{n,k,C,A}=\left\{\begin{array}{cll}
S_n^{k-m}\Ob_{n,k-m,C,A} & \mbox{if} & m<k,\\
\cO_{n,C,A}A^{m-k} & \mbox{if} & m\ge k.
\end{array}\right.
\label{jul23}  \\
& S_n(S_n^*S_n)^{-1}S_n^{k-1}\Ob_{n,k-1,C,A}=S_n^{k}\Ob_{n,k,C,A}
\quad\quad\mbox{for all}\quad k\ge 1.\label{jul22a}
\end{align}
\end{proposition}

\begin{proof} To prove \eqref{jul22}, we note that by \eqref{07},
$$
(S_n^{k}R_{n,k})(z)=\sum_{j=0}^\infty \bcs{n+j+k-1\\ j+k}\cdot z^{j+k}
$$
and therefore, we have from \eqref{3.1}
\begin{align*}
S_n^*S_n^{k} R_{n,k}(z) &=\sum_{j=0}^\infty \frac{k+j}{n+k+j-1}\cdot \bc{n+k+j-1\\ j+k}\cdot
z^{k+j-1}\\
&=\sum_{j=0}^\infty \bc{n+k+j-2\\ j+k-1}\cdot z^{k+j-1}=z^{k-1}R_{n,k-1}(z)=S_n^{k-1}
R_{n,k-1}(z)
\end{align*}  
which together with \eqref{defdelta} implies \eqref{jul22}. 

\smallskip

The identity \eqref{jul23} follows recursively from \eqref{jul22} and 
\eqref{3.3}.  Note that the special case $k=0, m=1$ gives us back  \eqref{3.3}.
To verify \eqref{jul22a}, we first note that 
$$
S_n^*S_n: \; \sum_{j=0}^\infty f_jz^j\mapsto \sum_{j=0}^\infty \frac{j+1}{j+n} \cdot f_jz^j
$$
from which it follows that  $S_n^*S_n$ is invertible on $\cA_n(\cY)$ and moreover,
$$
S_n(S_n^*S_n)^{-1}: \; \sum_{j=0}^\infty f_jz^j\mapsto \sum_{j=0}^\infty \frac{j+n}{j+1} 
\cdot f_jz^{j+1}.
$$
We now note that \eqref{jul22a} is an immediate consequence of 
\eqref{jul22}:
\begin{align*}
S_n(S_n^*S_n)^{-1}S_n^{k-1}\Ob_{n,k-1,C,A}&=S_n(S_n^*S_n)^{-1}S_n^*S_n^{k}\Ob_{n,k,C,A}\\
&=S_nS_n^{k-1}\Ob_{n,k,C,A}=S_n^{k}\Ob_{n,k,C,A}.
\end{align*}
\end{proof}

Finally we note that the operator $S_{n}(S_{n}^{*}S_{n})^{-1}$ was 
introduced and studied in the general context of a left-invertible 
Hilbert space operator $S_{n}$ by Shimorin \cite{sh1} where it 
is called the {\em Cauchy dual} of the left invertible operator 
$S_{n}$.  For the latest developments, see \cite{CC} and the 
references there.
} \end{remark}

Just as in the case of unshifted observability operators, it is 
possible to represent the range of a $k$-shifted observability 
operator as a reproducing kernel Hilbert space.

\begin{theorem}  \label{T:frakOkernel}
    The reproducing kernel for the space $S_{n}^{k} 
    \operatorname{Ran}  {\Ob}_{n,k,C,A}$ (with inner product 
    induced by $\cA_{n}(\cY)$) is given by
 \begin{equation}
\boldsymbol{\mathfrak{K}}_k(z,\zeta)=  z^k\overline\zeta^k
CR_{n,k}(zA)\Gr_{n,k,C,A}^{-1}R_{n,k}(\zeta A)^*C^*.
\label{deffrakk}  
\end{equation}
\end{theorem}
    
\begin{proof}
Observe that by \eqref{defdelta}, 
$$
S^k_n\Ob_{n,k,C,A} x=\sum_{j=k}^\infty\bcs{n+j-1 \\ j}(CA^{j-k}x)z^{j}.
$$
It then follows from \eqref{0.1} and \eqref{defR} that
\begin{align}
\|S^k_n\Ob_{n,k,C,A} x\|^2_{\cA_n(\cY)}=&\left\langle \sum_{j=k}^\infty 
\bcs{n+j-1\\ j}(A^*)^{j-k}C^*CA^{j-k}x, 
\, x\right\rangle_{\cX}\notag\\
=&\left\langle \sum_{j=0}^\infty \bcs{n+j+k-1\\ j+k}(A^*)^{j}C^*CA^{j}x, \, x
\right\rangle_{\cX}\notag\\
=&\left\langle \Gr_{n,k,C,A} x, \, x\right\rangle_{\cX}.\label{st10}
\end{align}
Therefore, by the general principle from \cite{NFRKHS} (see
the proof of part (3a) of Theorem \ref{T:1.2} above), it follows that
the reproducing kernel for $S^k_{n} \operatorname{Ran} {\Ob}_{n,k,C,A}$ is given by
\eqref{deffrakk} as asserted.
\end{proof}

With all this model machinery in hand, we are now able to give the 
additional properties concerning $n$-observability and exact 
$n$-observability.

\begin{proposition}  \label{P:exactobs} Let $(C,A)$ be an $n$-output 
    stable pair.
    
 \begin{enumerate}
\item If $(C,A)$ is $n$-observable, then $(C,A)$ is also 
$k$-observable for $1 \le k < n$.
	
\item It can happen that $(C,A)$ is exactly $n$-observable but there 
is a $k$ with $1 \le k < n$ such that $(C,A)$ is not exactly 
$k$-observable.

\item If $(C,A)$ is exactly $n$-observable (so $\Gr_{n,0,C,A} = 
\cG_{n,C,A}$ is strictly positive definite), then $\Gr_{n,k,C,A}$ is 
strictly positive definite for $k=1,2,3,\dots$.
\end{enumerate}
\end{proposition}

\begin{proof}[Proof of (1):] It is not difficult to check by making use of the formula \eqref{3.1} 
that the adjoint Bergman shift $S_{n}^{*}$ has no nontrivial 
isometric vectors for $n\ge 2$:  {\em if $n \ge 2$ and $f \in 
\cA_{n}(\cY)$ satisfies $\|S_{n}^{*} f \| = \| f \|$, then 
$f=0$.}  Making use of the intertwining relation \eqref{3.3} and the 
weighted Stein identity \eqref{id0}, one can show:  {\em for $n \ge 
2$, if $(C,A)$ is $n$-observable, then $(C,A)$ is also 
$(n-1)$-observable.}  Then a simple induction argument leads to 
statement (1) in the proposition.

\smallskip
\noindent 
\emph{Proof of (2):} We let $E \colon \cA_{n}(\cY) \to \cY$ be the 
evaluation-at-zero map $E \colon f \mapsto f(0)$.
The pair $(E,S_n^*)$ is $n$-observable since the kernel of $\cO_{n,E,S_n^*}$ is trivial and 
it is exactly $n$-observable since  $\cG_{n,E,S_n^*}=I_{\cA_n(\cY)}$ is strictly positive 
definite. However, we have from \eqref{id0}
$$
\cG_{n-1,E,S_n^*}=\cG_{n,E,S_n^*}- S_n\cG_{n,E,S_n^*}S_n^*=I-S_nS_n^*
$$
so that, according to \eqref{3.1}, $\cG_{n-1,E,S_n^*}: \; z^j\mapsto \frac{n-1}{n+j-1}z^j$.
Thus the $(n-1)$-gramian $\cG_{n-1,E,S_n^*}$ is not strictly positive definite so that the 
pair $(E,S_n^*)$ is not $(n-1)$-exactly observable.

\smallskip
\noindent
\emph{Proof of (3):}  If $(C,A)$ is exactly $n$-observable, then 
Corollary \ref{C:1.2} tells us that $(C,A)$ is similar to a model 
pair $(\widetilde C, \widetilde A): = (E|_{\cM},S_{n}^{*}|_{\cM})$ for an $S_{n}^{*}$-invariant 
subspace $\cM \subset \cA_{n}(\cY)$.  If $(C,A)$ and $(\widetilde C, 
\widetilde A)$ are similar output pairs, then it is easily verified 
that the associated gramians $\Gr_{n,k,C,A} = R_{n,k}(B_{A})[C^{*}C]$ 
and   $\Gr_{n,k,\widetilde C,\widetilde A} = R_{n,k}(B_{\widetilde 
A})[\widetilde C^{*} \widetilde C]$  are congruent.  Thus to show 
that $\Gr_{n,k,C,A}$ is strictly positive definite, it suffices to 
show that $\Gr_{n,k, \widetilde C, \widetilde A}$ is strictly positive 
definite. For the model pair $(\widetilde C, \widetilde A)$ one can 
check that
$$
  \Gr_{n,k,\widetilde C, \widetilde A} =
  P_{\cM} \Gr_{n,k,E, S_{n}^{*}}|_{\cM}.
$$
Thus strict positive-definiteness of $\Gr_{n,k,\widetilde C, 
\widetilde A}$ follows from strict positive-definiteness of 
$\Gr_{n,k,E, S_{n}^{*}}$.  
Observe that in view of \eqref{defdelta} and \eqref{E*Snk}, we have
\begin{equation}
\Ob_{n,k,E,S_n^*}f=\sum_{j=0}^\infty  \bcs{j+n+k-1\\ j+k} \left(E
S_{n}^{*j}f\right) z^{j}= \sum_{j=0}^{\infty} \frac{\mu_{n,j}}{\mu_{n,j+k}}f_{j} z^{j}.
\label{jul24}
\end{equation}
Therefore,
$$
S_n^k\Ob_{n,k,E,S_n^*}f=\sum_{j=0}^{\infty} \frac{\mu_{n,j}}{\mu_{n,j+k}}f_{j} z^{j+k}
$$
and by computation \eqref{st10}, we have
$$
\left\langle \Gr_{n,k,E,S_n^*}f, \, g\right\rangle
=\left\langle S_n^k\Ob_{n,k,E,S_n^*}f, \, S_n^k\Ob_{n,k,E,S_n^*}g\right\rangle
=\sum_{j=0}^\infty \frac{\mu_{n,j}^2}{\mu_{n,j+k}}\left\langle f_{j}, \, g_{j}\right\rangle
$$
from which we conclude that
\begin{equation}
 \Gr_{n,k,E,S_n^*}: \; \sum_{j=0}^\infty f_jz^j \to \sum_{j=0}^\infty
\frac{\mu_{n,j}}{\mu_{n,j+k}}f_{j} z^{j}.
\label{jul25}
\end{equation}
As $\mu_{n,j}$ is decreasing in $j$, we see that $\frac{ 
\mu_{n,j}}{\mu_{n,j+k}} \ge 1$ and hence $\Gr_{n,k,E,S_{n}^{*}} \ge 
I_{\cA_{n}(\cY)}$ is strictly positive definite.
\end{proof}

\section{Functions $\Theta_{n,k}$ and metric constraints } 
\label{S:metric}

In this section we take a closer look at the transfer functions $\Theta_{n,k}$ introduced by the 
realization formula \eqref{m5}. As was mentioned in Section 3, for an $n$-output stable pair $(C,A)$,
the associated backward-shifted observability operators $\Ob_{n,k,C,A}$ are bounded for all $k\ge 0$   
as operators from $\cX$ into $\cA_n(\cY)$. In this case, the multiplication operator $M_{\Theta_{n,k}}$
given (according to \eqref{m5}) by
\begin{equation}
M_{\Theta_{n,k}}=\bcs{k+n-1\\ k}D_k+ S_n \Ob_{n,k+1,C,A}B_k: \; 
\cU_k\to \cA_{n}(\cY)
\label{jul16}
\end{equation}
is also bounded. Therefore, the output function $\widehat{y}$ in \eqref{y-hat},
\begin{equation}
\widehat y(z)=\cO_{n,C,A}x+\sum_{k=0}^{N}z^k \Theta_{n,k}(z)u_k
\label{jul12}
\end{equation}
belongs to $\cA_n(\cY)$ for every choice of $x\in\cX$ and 
$u_k\in\cU_k$ for each $N=1,2, \dots$. We next impose some
additional metric relations on $\left[ \begin{smallmatrix} A & B_{k} 
\\ C & D_{k} \end{smallmatrix} \right]$, specifically one or more of 
the relations
\begin{align}
A^*\Gr_{n,k+1,C,A}B_k+\bcs{n+k-1 \\ k}\cdot C^*D_k&=0,\label{jul13}\\
B_k^*\Gr_{n,k+1,C,A}B_k+\bcs{n+k-1 \\ k}\cdot D_k^*D_k&\le I_{\cU_k},\label{jul14}\\
B_k^*\Gr_{n,k+1,C,A}B_k+\bcs{n+k-1 \\ k}\cdot D_k^*D_k&= I_{\cU_k},\label{jul14a}
\end{align}
and show how these lead to boundedness and orthogonality properties 
for the associated multiplication operator $M_{\Theta_{k}}$.
Due to equality \eqref{7.3}, it turns out that relations \eqref{jul13} and \eqref{jul14} are equivalent
to the matrix inequality 
\begin{equation}   \label{contr}
  \begin{bmatrix} A^{*} & C^{*} \\ B_{k}^{*} & D_{k}^{*}
  \end{bmatrix} \begin{bmatrix} \Gr_{n,k+1,C,A} & 0 \\ 0 & \bcs{n+k-1\\ k}\cdot  I_{\cY}
\end{bmatrix}   \begin{bmatrix} A & B_{k} \\ C & D_{k} \end{bmatrix}  
 \le \begin{bmatrix} \Gr_{n,k,C,A} & 0 \\ 0 & I_{\cU_{k}} \end{bmatrix},
    \end{equation}
while the equalities \eqref{jul13} and \eqref{jul14a} are equivalent to the matrix equality
\begin{equation}   \label{isom}
  \begin{bmatrix} A^{*} & C^{*} \\ B_{k}^{*} & D_{k}^{*}
  \end{bmatrix} \begin{bmatrix} \Gr_{n,k+1,C,A} & 0 \\ 0 & \bcs{n+k-1\\ k}\cdot  I_{\cY}
\end{bmatrix}   \begin{bmatrix} A & B_{k} \\ C & D_{k} \end{bmatrix}
 = \begin{bmatrix} \Gr_{n,k,C,A} & 0 \\ 0 & I_{\cU_{k}} \end{bmatrix}.
    \end{equation}
The two latter conditions are of metric nature; they express the 
contractivity or isometric property of the  colligation operator ${\bf U}_k =
\left[ \begin{smallmatrix} A & B_k \\ C & D_k \end{smallmatrix} \right]$ 
with respect to certain weights. Note that in the classical case $n=1$, 
the formula \eqref{jul18} below amounts to the following well-known fact:  
if the colligation operator ${\bf U} =
\left[ \begin{smallmatrix} A & B \\ C & D \end{smallmatrix} \right]$
is isometric and if we set $\Theta(z) = D + z C (I - zA)^{-1} B$,
then
\begin{equation}  \label{dBRdual}
   \frac{ I - \Theta(\zeta)^* \Theta(z)}{ 1 - z \bar{\zeta}} = B^* (I
   - \bar{\zeta}A^*)^{-1} (I - zA)^{-1}B.
\end{equation}

\begin{lemma}   \label{L:5.1}
Let $(C,A)$ be an $n$-output stable pair and let $\Theta_{n,k}$ be defined
as in \eqref{m5}  for some integer $k\ge 0$ and operators $B_k\in \cL(\cU_k, \cX)$ and $D_k\in \cL(\cU_k, 
\cY)$. 

\smallskip

$(1)$  If equality \eqref{jul13} holds, then 
\begin{itemize}
\item[(a)] $\cO_{n,C,A}x$ is orthogonal to $S_n^k\Theta_{n,k}u$ for all $x\in\cX$ and $u\in\cU_k$.
\item[(b)] $S_n^k\Theta_{n,k}u$ is orthogonal to  $S_n^m\Theta_{n,k}u^\prime$ for all $m>k$ and 
$u,u^\prime\in\cU_k$.
\end{itemize}

$(2)$ Moreover:
\begin{itemize}
    \item [(a)] If inequality \eqref{jul14} holds, then the 
    operator $S_{n}^{k} M_{\Theta_{n,k}}$ is a contraction from 
    $\cU_{k}$ into $\cA_{n}(\cY)$.
    
    \item[(b)] If both \eqref{jul13} and \eqref{jul14} hold, i.e., if 
    \eqref{contr} holds, then the operator 
    $S_{n}^{k}M_{\Theta_{n,k}}$ is a contraction from 
    $H^{2}(\cU_{k})$ into $\cA_{n}(\cY)$.
    \end{itemize}
    
\smallskip

$(3)$ Similarly:
\begin{itemize}
    \item[(a)]  If equality \eqref{jul14a} holds, then the operator 
    $S_{n}^{k} M_{\Theta_{n,k}}$ is an isometry from $\cU_{k}$ into 
    $\cA_{n}(\cY)$.
    
    \item[(b)]  If \eqref{jul13} and \eqref{jul14a} hold, i.e., if 
    \eqref{isom} holds, then 
    \begin{equation}   \label{SThetank-isom}
	\| S_{n}^{k} \Theta_{n,k}f \|^{2}_{\cA_{n}(\cY)} =
    \| f \|^{2}_{H^{2}(\cU_{k})} - \sum_{j=1}^{\infty } \| (I - 
    S_{n}^{*}S_{n})^{1/2} S_{n}^{k} \Theta_{n,k} S_{1}^{*j} f \|^{2}
    \end{equation}
for every $f\in H^2(\cU_k)$.

    \item[(c)] If \eqref{isom} holds, 
then 
\begin{align}
\frac{I_{\cU_{k}}}{\mu_{n,k}} - \Theta_{n,k}(z)^*
   \Theta_{n,k}(\zeta) = & \mu_{n,k}B_k^*R_{n,k}(z A)^*\Gr_{n,k+1,C,A}R_{n,k}(\zeta A)B_k  \label{jul18}\\
&- z \overline{\zeta} \mu_{n,k}  B_k^*R_{n,k+1}(z A)^*\Gr_{n,k,C,A}R_{n,k+1}(\zeta A)B_k.\notag
\end{align}
\end{itemize}
\end{lemma}

\begin{proof}[Proof of (1):] We first observe the power series expansion
\begin{equation}
\Theta_{n,k}(z)=\mu_{n,k}^{-1}D_k+\sum_{j=0}^\infty\mu_{n,j+k+1}^{-1}\cdot
CA^jB_kz^{j+1}.
\label{jul15}
\end{equation}
which is an immediate consequence of formulas \eqref{m5} and \eqref{defdelta}. We then make use of 
expansions \eqref{0.9}, \eqref{jul15} and the definition of the inner product in $\cA_n(\cY)$
to get 
\begin{align*}
&\left\langle S_n^{k}\Theta_{n,k} u, \, \cO_{n,C,A}x\right\rangle_{\cA_n(\cY)}=
\mu_{n,k}\cdot\left\langle \mu_{n,k}^{-1}D_ku, \, \mu_{n,k}^{-1}CA^{k}x\right\rangle_{\cY}  \\
&\qquad+\sum_{j=0}^\infty\mu_{n,j+k+1}\cdot \left\langle
\mu_{n,j+k+1}^{-1}C A^{j} B_k u, \, \mu_{n,j+k+1}^{-1}CA^{j+k+1} x\right\rangle_{\cY}  \\
&=\left\langle \left(\bcs{n+k-1 \\ k} C^*D_k+ A^*\left(
\sum_{j=0}^\infty \bcs{n+j+k \\ j+k+1} A^{*j}C^*CA^{j}\right)B_k\right)u,\,
A^kx\right\rangle_{\cX}  \\
&=\left\langle \left(\bcs{n+k-1 \\ k}C^*D_k+
A^*\Gr_{n,k+1,C,A}B_k\right)u,\, A^kx\right\rangle_{\cX} \text{ (by 
\eqref{defR})} \\
& =0  \text{ (by \eqref{jul13}).}
\end{align*}
which proves part (a).  Verification of part $(b)$ is quite similar: for $m>k$ we have
\begin{align*}
&\left\langle S_n^{m}\Theta_{n,k} u^\prime,\, S_n^{k}\Theta_{n,k}u\right\rangle_{\cA_n(\cY)}=
\mu_{n,m}\cdot\left\langle  \mu_{n,k}^{-1}D_ku^\prime, \; 
\mu_{n,m}^{-1}CA^{m-k-1}B_k u\right\rangle_{\cY}  \\
&\qquad+\sum_{j=0}^\infty\mu_{n,j+m+1}\cdot \left\langle
\mu_{n,j+k+1}^{-1}C A^{j} B_k u^\prime, \, 
\mu_{n,j+m+1}^{-1}CA^{j+m-k-1}B_k u\right\rangle_{\cY}  \\
&=\left\langle \left(\bcs{n+k-1 \\ k} C^*D_k+A^*\left(
\sum_{j=0}^\infty \bcs{n+j+k \\ j+k+1}A^{*j}C^*CA^{j}\right)B_k
\right)u^\prime,\, A^{m-k-1}B_ku\right\rangle_{\cX}  \\
&=\left\langle \left(\bcs{n+k-1 \\ k} C^*D_k+A^*\Gr_{n,k+1,C,A}B_k\right)
u^\prime, \,  A^{m-k-1}B_ku\right\rangle_{\cX} =0.
\end{align*}  
\end{proof}

\begin{proof}[Proof of (2):] According to \eqref{jul14},
\begin{equation}
\bcs{k+n-1\\ k}\cdot \| D_ku\|^2_{\cY}+\left\langle\Gr_{n,k+1,C,A}B_ku, \; B_ku\right\rangle_{\cX}\le 
\|u\|^2_{\cU_k}
\label{jul20}
\end{equation}
for all $u\in\cU_k$. 
We now may make use of \eqref{st10} to get
\begin{align}
\| S_n^{k}\Theta_{n,k} u\|^2_{\cA_n(\cY)}&=
\|(\bcs{k+n-1\\ k}S_n^k D_k+ S^{k+1}_n \Ob_{n,k+1,C,A}B_k)u\|^2_{\cA_n(\cY)}\notag\\
&=\|\bcs{k+n-1\\ k}S_n^k D_ku\|^2_{\cA_n(\cY)}+\|S^{k+1}_n \Ob_{n,k+1,C,A}B_ku\|^2_{\cA_n(\cY)}\notag\\
&=\bcs{k+n-1\\ k}\| D_ku\|^2_{\cY}+\left\langle\Gr_{n,k+1,C,A}B_ku, \; B_ku\right\rangle_{\cX}
\le \|u\|^2_{\cU_k}.\label{jul19}
\end{align}
Thus, $S_n^k M_{\Theta_{n,k}}$  is a contraction from $\cU_k$ to 
$\cA_n(\cY)$ and part (a) follows.
To prove part (b), following ideas from \cite{olaa}, under the 
assumption that both \eqref{jul13} and \eqref{jul14} hold, 
  we shall show that for any
$\cU_k$-valued polynomial  $f(z)={\displaystyle \sum_{j=0}^m f_jz^j}$,
\begin{equation}
\|S_n^{k}\Theta_{n,k} f\|^2_{\cA_n(\cY)}\le
\|f\|_{H^2(\cU_k)}^2=\sum_{j=0}^m\|f_j\|_{\cU_k}^2.
\label{8.3}
\end{equation}
Let $S_1^*$ be the operator of backward shift on $H^2(\cU_k)$ so that for the
polynomial $f$ as above, $(S_1^*f)(z)={\displaystyle \sum_{j=0}^{m-1} f_{j+1}z^j}$. By
statements (1b) and (3a) of the lemma, we have
\begin{align}
\left\|S_n^{k}\Theta_{n,k} f \right\|^2=&\left\|\sum_{j=0}^mS_n^{k+j}\Theta_{n,k}
f_j\right\|^2\notag\\
=&\left\|S_n^{k}\Theta_{n,k}f_0 \right\|^2+ 
\left\|\sum_{j=1}^mS_n^{k+j}\Theta_{n,k}f_j \right\|^2   \text{ (by (1b))} \notag\\
\le &\left\|f_0 \right\|^2+ 
\left\|S_n^{k+1}\sum_{j=0}^{m-1}S_n^{j}\Theta_{n,k}f_{j+1}\right\|^2\notag\\
=&\|f_0\|^2+\|S_n^{k+1}\Theta_{n,k} S_1^*f\|^2\notag\\
=&\left\|f_0 \right\|^2+ 
\left\|S_n^{k}\Theta_{n,k} S_1^*f \right\|^2- \left\|(I-S_n^*S_n)^{\frac{1}{2}}S_n^{k}
\Theta_{n,k}S_1^*f\right\|^2.\label{8.4}
\end{align}
Replacing $f$ by $S_1^{*j}f$ in \eqref{8.4} gives
\begin{align} 
\|S_n^{k}\Theta_{n,k} S_1^{*j}f\|^2\le \|f_j\|^2&+\|S_{n}^k\Theta_{n,k}
(S_1^{*})^{j+1}f\|^2\notag\\
&-\|(I-S_n^*S_n)^{\frac{1}{2}}S_n^{k}\Theta_{n,k} (S_1^*)^{j+1}f\|^2
 \label{8.3'}
\end{align}
for $j=1,\ldots,m$.  Iteration of the inequality \eqref{8.4} using 
\eqref{8.3'} then gives
\begin{align}
\|S_n^{k}\Theta_{n,k} f\|^2_{\cA_n(\cY)}
&\le \sum_{j=0}^m\|f_j\|_{\cU_k}^2-\sum_{j=1}^m\|(I-S_n^*S_n)^{\frac{1}{2}}S_n^{k}\Theta_{n,k}
S_1^{*j}f\|^2 \label{8.3iterate}\\
&\le \sum_{j=0}^{m} \| f_{j}\|^{2}_{\cU_{k}}.\notag
\end{align}
Letting $m \to \infty$ in \eqref{8.3iterate} now implies the validity 
of \eqref{8.3} for every $f\in H^2(\cU_k)$ and the proof of part 
(b) of (2) is now complete. 
\end{proof}

\begin{proof}[Proof of (3):]  In case \eqref{jul14a} holds,  then 
    \eqref{jul20} holds with equality and part (a) of (3) follows.  If also \eqref{jul18} holds, 
    then \eqref{8.4} holds with equality as well in \eqref{8.3}, \eqref{8.3'}, and 
    \eqref{8.3iterate}.  Part (b) of (3) now follows by letting $m 
    \to \infty$ in \eqref{8.3iterate}.
    
\smallskip

    It remains to verify the formula \eqref{jul18} under assumption 
    \eqref{isom}.  The identity \eqref{isom} is equivalent to the 
    collection of identities
    \begin{align} 
       &   A^{*} \Gr_{n,k+1,C,A} A + \frac{C^{*}C}{\mu_{n,k}} = 
	\Gr_{n,k,C,A}, \quad 
	 A^{*} \Gr_{n,k+1,C,A} B_{k} + \frac{C^{*}D_{k}}{\mu_{n,k}} 
	= 0,  \notag  \\  
     &  B_{k}^{*} \Gr_{n,k+1,C,A} B_{k} + 
	\frac{D_{k}^{*}D_{k}}{\mu_{n,k}} = I_{\cU_{k}}.
	\label{relations'}
\end{align}
We use these relations to compute
\begin{align*}
   & \frac{I_{\cU_{k}}}{\mu_{n,k}} - \Theta_{n,k}(z)^{*} 
    \Theta_{n,k}(\zeta)  =  \\
   &  
  \frac{I_{\cU_{k}}}{\mu_{n,k}} -
    \left[ \frac{D_{k}^{*}}{\mu_{n,k}} + \overline{z} B_{k}^{*} 
    R_{n,k+1}(zA)^{*} C^{*} \right]
    \left[ \frac{D_{k}}{\mu_{n,k}} + \zeta C R_{n,k+1}(\zeta A) B_{k} 
    \right]  \\
    &   = 
    \frac{I_{\cU_{k}}}{\mu_{n,k}} - \frac{D_{k}^{*} 
    D_{k}}{\mu^{2}_{n,k}} - \overline{z} B_{k}^{*} R_{n,k+1}(zA)^{*} 
    \frac{C^{*}D_{k}}{\mu_{n,k}} -\zeta \frac{D_{k}^{*}C}{\mu_{n,k}}
    R_{n,k+1}(\zeta A) B_{k} \\
    & \quad  - \overline{z} \zeta B_{k}^{*} R_{n,k+1}(zA)^{*} 
    C^{*}C R_{n,k+1}(\zeta A) B_{k}  \\
    &  = \frac{1}{\mu_{n,k}} B_{k}^{*} \Gr_{n,k+1,C,A} B_{k} + 
    \overline{z} B_{k}^{*} R_{n,k+1}(zA)^{*} \Gr_{n,k+1,C,A} B_{k} \\
    & \quad \
     + \zeta B_{k}^{*} \Gr_{n,k+1,C,A} A R_{n,k+1}(\zeta A) B_{k}  \\
     & \quad  
    - \mu_{n,k} \overline{z} \zeta B_{k}^{*} R_{n,k+1}(zA)^{*}
    \left(\Gr_{n,k,c,A} - A^{*} \Gr_{n,k+1,C,A}A \right) R_{n,k+1}(\zeta A) B_{k}
    \end{align*}
    where we made use of \eqref{relations'} in the last step.
  By making use next of relation \eqref{st4} we can continue the 
  computation as
  \begin{align*}
   & \frac{I_{\cU_{k}}}{\mu_{n,k}} - \Theta_{n,k}(z)^{*} 
    \Theta_{n,k}(\zeta)  \\  
   &=  \frac{1}{\mu_{n,k}} B_{k}^{*} \Gr_{n,k+1,C,A} B_{k} + B_{k}^{*}
   \left( R_{n,k}(zA)^{*} - \frac{I}{\mu_{n,k}} \right) 
   \Gr_{n,k+1,C,A} B_{k}  \\
   & \quad + B_{k}^{*} \Gr_{n,k+1,C,A} \left( 
   R_{n,k}(\zeta A) - \frac{I}{\mu_{n,k}} \right) B_{k} \\
   & \quad + \mu_{n,k} B_{k}^{*} \left( R_{n,k}(zA)^{*} - 
   \frac{I}{\mu_{n,k}} \right) \Gr_{n,k+1,C,A} \left( R_{n,k}(\zeta 
   A) - \frac{I}{\mu_{n,k}} \right) B_{k}  \\
   & \quad - \mu_{n,k} \overline{z} \zeta B_{k}^{*} R_{n,k+1}(zA)^{*} 
   \cG_{n,k,C,A} R_{n,k+1}(\zeta A) B_{k} \\
   & = \mu_{n,k} B_{k}^{*} R_{n,k}(zA)^{*} \Gr_{n,k+1,C,A} R_{n,k}(\zeta 
   A) B_{k} \\
   &\quad- \mu_{n,k} \overline{z} \zeta B_{k}^{*} 
   R_{n,k+1}(zA)^{*} \cG_{n,k,C,A} R_{n,k+1}(\zeta A) B_{k}
   \end{align*}
   verifying formula \eqref{jul18}.
\end{proof}

The following result is an immediate consequence of Lemma \ref{L:5.1}.

\begin{corollary}   \label{C:5.3}
Let us assume that the pair $(C,A)$ is $n$-output stable and that relations \eqref{jul13}, \eqref{jul14} hold 
for all $k\ge 0$. Then the representation \eqref{jul12} of the 
function $\widehat{y}$ is orthogonal in the metric 
of $\cA_n(\cY)$ and 
\begin{equation}
\|\widehat y\|^2_{\cA_n(\cY)}=\|\cO_{n,C,A}x\|^2+\sum_{k=0}^\infty 
\|\Theta_{n,k}u_k\|^2\le 
\|\cG_{n,C,A}^{\frac{1}{2}}x\|^2_{\cX}+\sum_{k=0}^\infty \|u_k\|^2_{\cU_k}.
\label{jul17}
\end{equation}
If relations \eqref{jul14} hold with equalities for all $k\ge 0$, then equality holds in \eqref{jul17}.
\end{corollary}
Observe that in case the pair $(C,A)$ is exactly $n$-observable (so that $\Gr_{n,k,C,A}$ is strictly positive 
definite for all $k\ge 0$ by Proposition \ref{P:exactobs}), the inequality \eqref{contr} can 
equivalently be expressed as $\| \Xi \| \le 1$ where $\Xi$ is the 
operator given by
\begin{equation}   \label{pr7}
 \Xi: = \begin{bmatrix} \Gr_{n,k+1,C,A}^{1/2} & 0 \\ 0 
 & \frac{I_{\cY}}{\mu_{n,k}} \end{bmatrix}
 \begin{bmatrix} A & B_{k} \\ C & D_{k} \end{bmatrix}
     \begin{bmatrix} \Gr_{n,k,C,A}^{-1/2} & 0 \\ 0 & 
	 I_{\cU_{k}} \end{bmatrix} \colon \begin{bmatrix} \cX \\ 
	 \cU_{k} \end{bmatrix} \to \begin{bmatrix} \cX \\ \cY 
     \end{bmatrix}.
 \end{equation}
Another equivalent condition is that $\| \Xi^{*} \| \le 1$ which in 
turn can be expressed as 
\begin{equation}   \label{contra}
    \begin{bmatrix} A & B_{k} \\ C & D_{k} \end{bmatrix}
        \begin{bmatrix} \Gr_{n,k,C,A}^{-1} & 0 \\ 0 & I_{\cU_k} \end{bmatrix}
        \begin{bmatrix} A^{*} & C^{*} \\ B_{k}^{*} & D_{k}^{*}
        \end{bmatrix}\le  \begin{bmatrix} \Gr_{n,k+1,C,A}^{-1} & 0 \\ 0 &
        \mu_{n,k} I_{\cY} \end{bmatrix}.
\end{equation}
 Note that equality \eqref{isom} means that the operator $\Xi$ is isometric.
Of  particular interest  is the case where $\Xi$ is coisometric, i.e., where 
the colligation operator ${\bf U}_k= \left[ \begin{smallmatrix} A & B_k \\ C & D_k \end{smallmatrix} \right]$
is coisometric with respect to the weights indicated below:
\begin{equation}   \label{wghtcoisom}
    \begin{bmatrix} A & B_{k} \\ C & D_{k} \end{bmatrix}
        \begin{bmatrix} \Gr_{n,k,C,A}^{-1} & 0 \\ 0 & I_{\cU_k} \end{bmatrix}
        \begin{bmatrix} A^{*} & C^{*} \\ B_{k}^{*} & D_{k}^{*}
        \end{bmatrix} = \begin{bmatrix} \Gr_{n,k+1,C,A}^{-1} & 0 \\ 0 &
        \mu_{n,k} I_{\cY} \end{bmatrix}.
\end{equation}

We can now derive the shifted weighted Bergman-space analogue of the 
familiar identity for the de Branges-Rovnyak kernel which is a dual 
version  
of \eqref{dBRdual}:  if $\bU =
\left[ \begin{smallmatrix} A & B \\ C & D \end{smallmatrix} \right]$
is coisometric and if we set $\Theta(z) = D + z C (I - zA)^{-1} B$,
then 
$$
   \frac{ I - \Theta(z) \Theta(\zeta)^{*}}{ 1 - z \bar{\zeta}} = C (I
   - zA)^{-1} (I - \bar{\zeta}A^{*})^{-1} C^{*}.
$$

\begin{lemma}   \label{L:frakcoisom}
Let $(C,A)$ be an exactly $n$-observable $n$-output stable pair and let $\Theta_{n,k}$ be defined
as in \eqref{m5}  for some operators $B_k\in \cL(\cU_k, \cX)$ and $D_k\in \cL(\cU_k)$ subject to 
equality \eqref{wghtcoisom}. Then
\begin{align}
\frac{I_{\cY}}{\mu_{n,k}} - \Theta_{n,k}(z)
   \Theta_{n,k}(\zeta)^{*} = &
CR_{n,k}(zA)\Gr_{n,k,C,A}^{-1}R_{n,k}(\zeta A)^*C^*,\notag \\
&- z \overline{\zeta} \cdot CR_{n,k+1}(zA)\Gr_{n,k+1,C,A}^{-1}R_{n,k+1}(\zeta A)^*C^*.
    \label{frakkerid}
\end{align}
\end{lemma}

\begin{proof}The proof parallels the verification of the identity 
    \eqref{jul18} done above.  The weighted-coisometry condition \eqref{wghtcoisom}
    gives us the set of equations
 \begin{align}
     A \Gr_{n,k,C,A}^{-1} A^{*} + B_{k} B_{k}^{*} = \Gr_{n,k+1,C,A}^{-1}, & \quad C
     \Gr_{n,k,C,A}^{-1} A^{*} + D_{k} B_{k}^{*} = 0  \notag \\
     C \Gr_{n,k,C,A}^{-1} C^{*} + D_{k} D_{k}^{*} = \mu_{n,k} I_{\cY}.
     \label{relations1}
 \end{align}
We then  compute:
 \begin{align*}
    & \frac{I_{\cY}}{\mu_{n,k}} - \Theta_{n,k}(z)
     \Theta_{n,k}(\zeta)^{*} \\
     & =     \frac{I_\cY}{\mu_{n,k}} - \left[ \frac{D_{k}}{\mu_{n,k}} + z
     CR_{n,k+1}(zA) B_{k} \right] \left[ \frac{D_{k}^{*}}{\mu_{n,k}} +
     \bar{\zeta} B_{k}^{*} R_{n,k+1}(\zeta A)^*C^*\right] \\
     & = \frac{I_{\cY}}{\mu_{n,k}} - \frac{D_{k}
     D_{k}^{*}}{\mu_{n,k}^{2}} - z CR_{n,k+1}(zA)
     \frac{B_{k}D_{k}^{*}}{\mu_{n,k}} - \bar{\zeta}   
     \frac{D_{k}B_{k}^{*}}{\mu_{n,k}} R_{n,k+1}(\zeta A)^*C^*  \\
     & \quad -  z \bar{\zeta} \cdot CR_{n,k+1}(zA)B_{k}B_{k}^{*}
     R_{n,k+1}(\zeta A)^*C^* \text{ (by \eqref{relations1})}      \\
     & =  \frac{C}{\mu_{n,k}}
     \Gr_{n,k,C,A}^{-1} \frac{C^{*}}{\mu_{n,k}} + z CR_{n,k+1}(zA) A
     \Gr_{n,k,C,A}^{-1} \frac{C^{*}}{\mu_{n,k}} \\
 & \quad + \bar{\zeta}
     \frac{C}{\mu_{n,k}} \Gr_{n,k,C,A}^{-1} A^{*} R_{n,k+1}(\zeta A)^*C^* \\
     & \quad        - z \bar{\zeta} \cdot CR_{n,k+1}(zA)\left[ \Gr_{n,k+1,C,A}^{-1} - A
\Gr_{n,k,C,A}^{-1}  A^{*}\right] R_{n,k+1}(\zeta A)^*C^*\\
&=C\left(\frac{I_{\cX}}{\mu_{n,k}}+zR_{n,k+1}(zA)A\right)\Gr_{n,k,C,A}^{-1}
\left(\frac{I_{\cX}}{\mu_{n,k}}+\overline{\zeta}A^*R_{n,k+1}(\zeta A)^*\right)C^* \\
&\quad- z \bar{\zeta} \cdot CR_{n,k+1}(zA)\Gr_{n,k+1,C,A}^{-1} R_{n,k+1}(\zeta A)^*C^*\\
&=CR_{n,k}(zA)\Gr_{n,k,C,A}^{-1}R_{n,k}(\zeta A)^*C^*\notag \\
&\quad - z \overline{\zeta} \cdot CR_{n,k+1}(zA)\Gr_{n,k+1,C,A}^{-1}R_{n,k+1}(\zeta A)^*C^*
\text{ (by \eqref{st4}).}
     \end{align*}
\end{proof}

\begin{remark}   \label{R:defect}
    {\em More generally, if $\Theta_{n,k}(z)$ is given by \eqref{m5}
    and if we do not assume the weighted coisometry condition
    \eqref{wghtcoisom}, then the decomposition \eqref{frakkerid}    
    holds in the more general form
    \begin{align*}
     &  \frac{I_{\cY}}{\mu_{n,k}} - \Theta_{n,k}(z)
        \Theta_{n,k}(\zeta)^{*} =
CR_{n,k}(zA)\Gr_{n,k,C,A}^{-1}R_{n,k}(\zeta A)^*C^* \\
     & \quad - z  \overline{\zeta} \cdot
CR_{n,k+1}(zA)\Gr_{n,k+1,C,A}^{-1}R_{n,k+1}(\zeta A)^*C^*
 + \Xi_{k}(z, \zeta)
    \end{align*}
    where the defect kernel $\Xi_{k}(z, \zeta)$ is given by
  \begin{align*}
       \Xi_{k}(z, \zeta) = &\begin{bmatrix} z CR_{n,k}(zA) &
      \frac{I_{\cY}}{\mu_{n,k}} \end{bmatrix}
\left( \begin{bmatrix} \Gr_{n,k+1}^{-1} & 0 \\ 0 &
      \mu_{n,k} I_{\cY} \end{bmatrix}\right. \\
      & \left. 
    - \begin{bmatrix} A & B_{k} \\ C & D_{k} \end{bmatrix}
  \begin{bmatrix} \Gr_{n,k}^{-1} & 0 \\ 0 &  I_{\cY} \end{bmatrix}
      \begin{bmatrix} A^{*} & C^{*} \\ B_{k}^{*} & D_{k}^{*}   
      \end{bmatrix} \right) \cdot  \begin{bmatrix} \overline{\zeta}
   R_{n,k}(\zeta A)^*C^* \\ \frac{I_{\cY}}{\mu_{n,k}}
  \end{bmatrix}. 
  \end{align*}
     }\end{remark}
Since equality \eqref{wghtcoisom} implies inequality \eqref{contr}, it follows that under assumption of Lemma 
\ref{L:frakcoisom}, all the conclusions of parts (1) and (2) in Lemma \ref{L:5.1} are true. To have 
all conclusions true, we need the operator \eqref{pr7} to be unitary. 
\begin{lemma}   \label{L:5.6}
    Suppose that we are given an integer $k \ge 0$ and an exactly $n$-observable $n$-output-stable pair
    $(C,A)$ with $A\in\cL(\cX)$ and $C\in\cL(\cX,\cY)$ . Then there exist operators  $B_k\in \cL(\cU_k, \cX)$ 
and $D_k\in \cL(\cU_k, \cY)$ such that equalities \eqref{wghtcoisom} and \eqref{contr} hold.
 Explicitly, such  $B_{k}$ and $C_{k}$ are essentially uniquely
 determined by solving the Cholesky factorization problem:
  \begin{equation}  \label{pr6}
\begin{bmatrix}B_k \\ D_k\end{bmatrix}\begin{bmatrix}B_k^* & D_k^*\end{bmatrix}=
\begin{bmatrix}\Gr_{n,k+1,C,A}^{-1} & 0 \\ 0 & \mu_{n,k} I_{\cY}\end{bmatrix}-
\begin{bmatrix}A \\ C\end{bmatrix}\Gr_{n,k,C,A}^{-1}\begin{bmatrix}A^* & C^*\end{bmatrix}
\end{equation}
subject to the additional constraint that the coefficient space
$\cU_{k}$ be chosen so that $\left[ \begin{smallmatrix} B_{k} \\
D_{k} \end{smallmatrix} \right] \colon \cU_{k} \to \cX \oplus \cY$ is
injective.
\end{lemma}
\begin{proof}
By Proposition \ref{P:wghtSteinid}, the weighted Stein
identity \eqref{7.3} holds for each $k \ge 1$.
Since $(C,A)$ is exactly observable, the gramian $\Gr_{n,k,C,A}$ is strictly
positive definite and then it follows from \eqref{7.3} that the operator
$$
\begin{bmatrix}\Gr_{n,k+1,C,A}^{\frac{1}{2}}A\Gr_{n,k,C,A}^{-\frac{1}{2}} \\
\mu_{n,k}^{-\frac{1}{2}}C\Gr_{n,k,C,A}^{-\frac{1}{2}}\end{bmatrix} \colon  \cX\to \cX\oplus \cY
$$
is an isometry.  By extending this operator to a coisometric  operator \eqref{pr7}
we arrive at $B_k$ and $D_k$ solving \eqref{pr6}.  Further, extension
of this operator to a unitary amounts to the additional restriction  
that $\left[ \begin{smallmatrix} B_{k} \\ D_{k} \end{smallmatrix}    
\right]$ be injective.  
\end{proof}

\section{Beurling-Lax theorem for $\cA_{n}$}  \label{S:BL}

We now present our four approaches to Beurling-Lax representations 
for the Bergman space setting; these were sketched for the classical 
case in the Introduction above.
\smallskip

We start with a general observation. If the subspace $\cM \subset \cA_n(\cY)$ is $S_n$-invariant, 
then $\cM^{\perp}$ is $S_n^*$-invariant; hence we may apply Theorem \ref{T:1.2} part
    (3) (with $\cM^{\perp}$ in place of $\cM$ viewed as sitting
    isometrically inside $\cA_{n}(\cY)$) to conclude that there is an
    $n$-isometric pair $(C,A)$ so that $\cM^{\perp} =
    \operatorname{Ran} \cO_{n, C,A}$; in fact, as indicated in
    Theorem \ref{T:1.2}, we may take $(C,A)$ to be the
    model output pair $(C,A) = (E|_{\cM^{\perp}},
    S_{n}^{*}|_{\cM^{\perp}})$, and $\cO_{n,C,A}$ amounts to
    the inclusion map of $\cM^{\perp}$ into $\cA_{n}(\cY)$. Since $\|{\cO}_{n,C,A}x\|_{\cA_n(\cY)}^2=
\langle \cG_{n,C,A}x, \, x\rangle_{\cX}$ for every $x\in\cX$, it follows that 
$\cM^{\perp}$ is a reproducing kernel Hilbert space  with reproducing kernel 
$$
k_{\cM^\perp}(z, \zeta) = C(I - zA)^{-n} \cG_{n,C,A}^{-1} (I -
    \overline{\zeta} A^{*})^{-n} C^{*}.
$$
It then follows that $\cM = \left( \cM^{\perp} \right)^{\perp}$ has reproducing kernel
\begin{equation}   
\label{kM}
k_{\cM}(z, \zeta) = \frac{I_{\cY}}{(1 - z \overline{\zeta})^{n}} -
C (I - zA)^{-n} \cG^{-1}_{n,C,A} (I - \overline{\zeta}A^{*})^{-n}C^{*}.
    \end{equation}

\subsection{The first approach: partially isometric multipliers} 
Suppose that we are given a shift-invariant subspace $\cM$ contained 
in the Hardy space $H^{2}(\cY)$.  As we just observed in the previous 
paragraph, then $\cM^{\perp}$ is 
backward-shift-invariant and hence can be represented as the range 
$\operatorname{Ran} \cO_{C,A}$ of an observable output pair $(C,A)$; 
we can even take $(C,A)$ to be the model output pair:  $C \colon f \mapsto f(0)$ and $A$ equal to 
the backward 
shift restricted to $\cX=\cM^{\perp}$. 
From the identity
\begin{equation}
{\mathcal G}_{1,C,A}-A^*{\mathcal G}_{1,C,A}A=C^*C\label{6.7}
\end{equation}
(where $\cG_{1,C,A}$ is the identity operator in the model case),
we see that the column matrix $\left[ \begin{smallmatrix}
\cG_{1,C,A}^{1/2} A \cG^{-1/2}_{1,C,A} \\ C \cG_{1,
C,A}^{-1/2} \end{smallmatrix} \right]$ is isometric.  Hence there
exist  operators $\Bo_{1} \colon \widetilde \cU_{1} \to \cX$ ($=
\cM^{\perp}$) and $\Do_{1} \colon \widetilde \cU_{1} \to \cY$ (where
$\widetilde \cU_{1}$ is an appropriate coefficient Hilbert space) so   
that the operator matrix
\begin{equation}   \label{unitary1}
\begin{bmatrix}{\mathcal G}_{1,C,A}^{\frac{1}{2}}A{\mathcal G}_{1,C,A}^{-\frac{1}{2}} &
{\mathcal  G}_{1,C,A}^{\frac{1}{2}}\Bo_1 \\ C{\mathcal G}_{1,C,A}^{-\frac{1}{2}} & 
\Do_1\end{bmatrix}:
\; \begin{bmatrix}\cX \\ \cU\end{bmatrix}\to \begin{bmatrix}\cX \\ \cY\end{bmatrix}
\end{equation}
is unitary. We then set
\begin{equation}
\To_1(z)=\Do_1+zC(I-zA)^{-1}\Bo_1.   
\label{6.8}
\end{equation}
Making use of the unitary property of the matrix \eqref{unitary1}, we
deduce that $\To_{0}$ is inner in the sense that its boundary
values exist almost everywhere on ${\mathbb T}$ and are coisometric,
and that the kernel identity
\begin{equation}
\frac{I_\cY-\To_1(z)\To_1(\zeta)^*}{1-z\overline{\zeta}}=
C(I-zA)^{-1}{\mathcal G}_{1,C,A}^{-1}(I-\overline{\zeta}A^*)^{-1}C^*   
\label{6.9}
\end{equation}
holds. For the case where $n=1$, the latter identity implies the factorization 
of the kernel \eqref{kM}
\begin{equation}
k_{\cM}(z,\zeta)=\frac{\To_1(z)\To_1(\zeta)^*}{1-z\overline{\zeta}}.
\label{kM1}
\end{equation}
Standard reproducing-kernel arguments (details appear in the proof of 
Theorem \ref{T:BL} below for the more general representation 
\eqref{sumker}) then imply that $\cM = 
\To_{1} \cdot \cA_{1}(\cU) = \To_{1} \cdot H^{2}(\cU)$, i.e., 
$\To_{1}$ serves as a Beurling-Lax representer for $\cM$.
One can view this approach as a state-space adaptation of the de 
Branges-Rovnyak reproducing kernel approach (see \cite{dBR1, dBR2}); it has proved to work 
well in a number of multivariable settings (see e.g.~\cite{BBF1, 
BBF2a}).  The next theorem presents a Bergman-space analogue 
of the representation \eqref{kM1}.
\begin{theorem} \label{T:BL}
    Let $\cM$ be a closed $S_{n}$-invariant subspace of 
    $\cA_{n}(\cY)$. Then there is a partially isometric multiplier 
    $$ 
M_{F} = \begin{bmatrix} M_{F_{1}} & \cdots & M_{F_{n}} \end{bmatrix}
    \colon \; {\displaystyle\bigoplus_{j=1}^{n} \cA_{j}(\cU_{j})} \to 
    \cA_{n}(\cY)
$$
    (so $F(z)$ has the form $F(z) = \begin{bmatrix} F_{1}(z) & \cdots 
    & F_{n}(z) \end{bmatrix}$ with each $F_{j}$ holomorphic on 
    ${\mathbb D}$ with values in $\cL(\cU_{j}, \cY)$ for appropriate 
    coefficient Hilbert spaces $\cU_{j}$, ($j = 1, \dots, n$), so that 
    \begin{equation}   \label{BLrep}
	\cM = M_{F} \left( \bigoplus_{j=1}^{n} \cA_{j}(\cU_{j}) 
	\right).
  \end{equation}
    \end{theorem}
    
\begin{proof} To prove the theorem, it suffices to produce 
    coefficient Hilbert spaces $\cU_{1}, \dots, \cU_{n}$ and 
    holomorphic operator-valued functions $F_{j} \colon {\mathbb D} 
    \to \cL(\cU_{j}, \cY)$ so that
\begin{equation}   \label{sumker}
    k_{\cM}(z, \zeta) = \sum_{j=1}^{n} \frac{F_{j}(z) 
    F_{j}(\zeta)^{*}}{(1 - z \overline{\zeta})^{j}},
 \end{equation}
where $k_\cM$ is given in \eqref{kM}.
Indeed, assume that \eqref{sumker} holds. Recall that $\cA_j$ is the reproducing kernel Hilbert space with 
reproducing kernel ${\bf k}_j(z,\zeta)=(1-z\overline{\zeta})^{-j}$ and define the operator 
 $X
\colon \, \cM = \cH(k_{\cM}) \to 
\bigoplus_{\ell=1}^{n} \cA_{\ell}(\cU_{\ell})$ by
 $$
  X \colon \; k_{\cM}(\cdot, \zeta) y \mapsto \begin{bmatrix} 
  {\bf k}_{1}(\cdot, \zeta) F_{1}(\zeta)^{*}y \\ \vdots \\ {\bf k}_{n}(\cdot, 
  \zeta) F_{n}(\zeta)^{*}y \end{bmatrix}\qquad (y\in\cY, \; \zeta\in{\mathbb D}).
 $$
 Then we have for $y_i\in\cY$ and $\zeta_i\in{\mathbb D}$,
 \begin{align*}
     \left\| X \left( \sum_{j=1}^{N} k_{\cM}(\cdot, \zeta_{j}) y_{j}) 
     \right) \right\|^{2} & = 
     \left\langle \sum_{j=1}^{N} \begin{bmatrix} {\bf k}_{1}(\cdot, 
     \zeta_{j}) F_{1}(\zeta_{j})^{*} y_{j} \\ \vdots \\ {\bf k}_{n}(\cdot, 
     \zeta_{j}) F_{n}(\zeta)^{*} y_{j} \end{bmatrix} , \sum_{i=1}^{N} 
     \begin{bmatrix} {\bf k}_{1}(\cdot, \zeta_{i}) F_{1}(\zeta_{i})^{*} 
	 y_{i} \\ \vdots \\ {\bf k}_{n}(\cdot, \zeta_{i}) 
	 F_{n}(\zeta_{i})^{*} y_{i} \end{bmatrix} \right\rangle \\
	 & = \sum_{i,j=1}^{N} \left\langle \sum_{\ell=1}^{N} 
	 F_{\ell}(\zeta_{i}) {\bf k}_{\ell}(\zeta_{i}, \zeta_{j}) 
	 F_{\ell}(\zeta_{j})^{*} y_{j}, y_{i} \right\rangle_{\cY} \\
	 & = \sum_{i,j=1}^{N} \langle k_{\cM}(\zeta_{i}, \zeta_{j}) 
	 y_{j}, \, y_{i} \rangle_{\cY}= \left\| \sum_{j=1}^{N} k_{\cM}(\cdot, \zeta_{j})y_{j} 
	 \right\|^{2}_{\cM}.
\end{align*}
We conclude that $X$ extends uniquely from the span of the kernel 
functions $k_{\cM}(\cdot, \zeta) y$ ($\zeta \in {\mathbb D}$, $y \in 
\cY$) to define an isometry from $\cH(k_{\cM}) = \cM$ into the direct 
sum space $\bigoplus_{\ell = 1}^{n} \cA_{\ell}(\cU_{\ell})$.  Moreover, we 
compute, for $h_{\ell} \in \cA_{\ell}(\cU_{\ell})$ ($\ell = 1, \dots, n$),
\begin{align*}
    \left\langle \left( X^{*} \left[ \begin{smallmatrix} h_{1} \\ 
    \vdots \\ h_{n} \end{smallmatrix} \right] \right)(\zeta), \, y 
    \right\rangle_{\cY} & =
    \left\langle X^{*} \left[ \begin{smallmatrix} h_{1} \\ \vdots \\ 
    h_{n} \end{smallmatrix} \right], \, k_{\cM}(\cdot, \zeta) y 
    \right\rangle_{\cM}  \\
    & = \left\langle \left[\begin{smallmatrix} h_{1} \\ \vdots \\ h_{n} 
\end{smallmatrix} \right], \, X k_{\cM}(\cdot, \zeta) y 
\right\rangle_{\bigoplus_{\ell=1}^{n}\cA_{\ell}(\cU_{\ell})}  \\
& = \left\langle \left[ \begin{smallmatrix} h_{1} \\ \vdots \\ h_{n} 
\end{smallmatrix} \right], \left[ \begin{smallmatrix} {\bf k}_{1}(\cdot, 
\zeta) F_{1}(\zeta)^{*} y \\ \vdots \\ {\bf k}_{n}(\cdot, \zeta) 
F_{n}(\zeta)^{*} y \end{smallmatrix} \right] \right\rangle_{\bigoplus_{\ell=1}^{n} 
\cA_{\ell}(\cU_{\ell})}  \\
& = \langle F_{1}(\zeta) h_{1}(\zeta) + \cdots + F_{n}(\zeta) 
h_{n}(\zeta), \, y \rangle_{\cY}
\end{align*}
from which we conclude that 
$$
X^{*} = M_{F} \colon \left[ \begin{smallmatrix} h_{1}(z) \\ \vdots \\ 
h_{n}(z) \end{smallmatrix} \right]  \mapsto F_{1}(z) h_{1}(z) + 
\cdots + F_{n}(z) h_{n}(z)
$$
is the multiplication operator $M_{F} = \begin{bmatrix} M_{F_{1}} & 
\cdots & M_{F_{n}} \end{bmatrix}$ induced by the matrix function $F(z) 
= \begin{bmatrix} F_{1}(z) & \cdots & F_{n}(z) \end{bmatrix}$ and  
Theorem \ref{T:BL} then follows.

\smallskip

Thus the proof of Theorem \ref{T:BL} is reduced to the construction of 
holomorphic operator-valued functions $F_{1}(z), \dots, F_{n}(z)$ so 
that \eqref{sumker} holds.  The construction proceeds via an iterated 
unitary completion process as follows.  

\smallskip	

For $j=2, \dots, n$ we do a construction similar to that in 
\eqref{6.7}--\eqref{6.8} but based  on the  identity (see \eqref{id0})
\begin{equation}
{\mathcal G}_{j,C,A}-A^*{\mathcal G}_{j,C,A}A={\mathcal G}_{j-1,C,A} \quad  (j=2,\ldots,n).
\label{6.10}
\end{equation}
We find operators $\Bo_{j} \colon \, \cU_{n+1-j} \to \cX$ and 
$\Do_{j} \colon \, \cU_{n+1-j} \to \cY$ so that 
\begin{equation}
\begin{bmatrix}A & \Bo_j \\ {\mathcal G}_{j-1,C,A}^{\frac{1}{2}} &
\Do_j\end{bmatrix}\begin{bmatrix}{\mathcal G}_{j,C,A}^{-1} & 0 \\ 0 &
I_{\cU_{n+1-j}}\end{bmatrix}\begin{bmatrix}A^* & {\mathcal G}_{j-1,C,A}^{\frac{1}{2}} \\ 
\Bo_j^* & \Do_j^*\end{bmatrix}=
\begin{bmatrix}{\mathcal G}_{j,C,A}^{-1} & 0 \\ 0 & I_{\cY}\end{bmatrix}
\label{6.11}
\end{equation}
and
\begin{equation}
\begin{bmatrix}A^* & {\mathcal G}_{j-1,C,A}^{\frac{1}{2}} \\ \Bo_j^* &
\Do_j^*\end{bmatrix}\begin{bmatrix}{\mathcal G}_{j,C,A}& 0 \\ 0 &
I_{\cY}\end{bmatrix}\begin{bmatrix}A & \Bo_j \\ {\mathcal G}_{j-1,C,A}^{\frac{1}{2}} &
\Do_j\end{bmatrix}=\begin{bmatrix}{\mathcal G}_{j,C,A} & 0 \\ 0 & 
I_{\cU_{n+1-j}}\end{bmatrix}.
\label{6.12}
\end{equation}
In fact, the latter equalities determine $\Bo_j$ and $\Do_j$ uniquely up 
to a common unitary factor on  the right:
\begin{equation}
\Bo_j=({\mathcal G}_{j,C,A}^{-1}-A{\mathcal G}_{j,C,A}^{-1}A^*)^{\frac{1}{2}},\quad
\Do_j=-{\mathcal G}_{j-1,C,A}^{-\frac{1}{2}}A^*{\mathcal G}_{j,C,A}\Bo_j.
\label{6.13}
\end{equation}
We now define the functions
\begin{equation}
\To_j(z)=\Do_j+z{\mathcal G}_{j-1,C,A}^{\frac{1}{2}}(I-zA)^{-1}\Bo_j
\label{6.14}
\end{equation}
for $j=2,\ldots,n$, which are inner and satisfy the identities
\begin{equation}
\frac{I_\cY-\To_j(z)\To_j(\zeta)^*}{1-z\overline{\zeta}}=
{\mathcal G}_{j-1,C,A}^{\frac{1}{2}}(I-zA)^{-1}{\mathcal G}_{j,C,A}^{-1}
(I-\overline{\zeta}A^*)^{-1}{\mathcal G}_{j-1,C,A}^{\frac{1}{2}}.
\label{6.15}
\end{equation}

We finally define operator-valued functions $F_{\ell}(z) \colon \, \cU_{\ell} \to \cY$ via
\begin{align} 
  &  F_{\ell}(z) = C (I -zA)^{-(n-\ell)}\cG^{-\frac{1}{2}}_{n-\ell,C,A} 
    \To_{n+1-\ell}(z) \quad\text{for} \quad\ell = 1, 2, \dots, n-1,  
    \notag \\
    & F_{n}(z) = \To_{1}(z)=\Do_1+zC(I-zA)^{-1}\Bo_1
    \label{defFell} 
 \end{align}
and claim that this choice of $F_{1}, \dots, F_{n}$ satisfies 
 \eqref{sumker} (with coefficient spaces $\cU_{j} = \cX$ for $j=1, 
 \dots, n-1$). Indeed, substituting equality \eqref{6.15} (for $j=n$) into the 
formula \eqref{kM} for $k_{\cM}$ and then making use of formula \eqref{defFell} for $F_1$,
we have
 \begin{align*}
k_{\cM}(z, \zeta) & = \frac{I_{\cY}}{(1 - z \overline{\zeta})^{n}} - C(I - 
     zA)^{-(n-1)} \cG^{-\frac{1}{2}}_{n-1,C,A} \left[ \frac{ I - \To_{n}(z) 
     \To_{n}(\zeta)^{*}}{1 - z \overline{\zeta}} \right] \\
    & \qquad\qquad\qquad\qquad  \cdot \cG^{-\frac{1}{2}}_{n-1,C,A} (I - \overline{\zeta} 
     A^{*})^{-(n-1)} C^{*}  \\
     & =  \frac{I_{\cY}}{(1 - z \overline{\zeta})^{n}} + 
     \frac{F_{1}(z) F_{1}(\zeta)^{*}}{1 - z \overline{\zeta}}\\
& \qquad- C(I - zA)^{-(n-1)} \frac{\cG^{-1}_{n-1,C,A}}{1 - z \overline{\zeta}} 
     (I - \overline{\zeta} A^{*})^{-(n-1)} C^{*}.
 \end{align*}
 One can now use an inductive argument to show that in general
 \begin{align}  
k_{\cM}(z, \zeta) =& \frac{I_{\cY}}{(1 - z \overline{\zeta})^{n}}
    + \sum_{\ell=1}^{k} \frac{F_{\ell}(z) F_{\ell}(\zeta)^{*}}{(1 -
     z \overline{\zeta})^{\ell}} \notag \\
&     - C (I - zA)^{-(n - k)} 
     \frac{\cG^{-1}_{n-k, C,A}}{(1 - z \overline{\zeta})^{k}} 
     (I - \overline{\zeta} A^{*})^{-(n-k)} C^{*}
      \label{id-ell}
 \end{align}
 for $k = 1,2, \dots, n-1$.  For the final step, we note that the 
 last term in \eqref{id-ell} with $k=n-1$ is given by
 \begin{align*}
   C(I - zA)^{-1} \frac{\cG^{-1}_{1,C,A}}{(1 - z 
    \overline{\zeta})^{n-1}} (I - \overline{\zeta} A^{*})^{-1} C^{*} 
   &= \frac{1}{(1 - z \overline{\zeta})^{n-1}} \left[ \frac{I_{\cY} - 
   \To_{1}(z) \To_{1}(\zeta)^{*}}{1 - z \overline{\zeta}}
   \right] \\
   & = \frac{I_{\cY}}{(1 - z \overline{\zeta})^{n}} - 
   \frac{F_{n}(z) F_{n}(\zeta)^{*}}{(1 - z \overline{\zeta})^{n}}.
  \end{align*}
  Plugging this expression into the $k = n-1$ case of \eqref{id-ell} leaves us with 
  \eqref{sumker} as wanted.
 \end{proof}
 
 There is an alternate formula for $\To_{j}$ (and hence also for 
 $F_{\ell}$) which may prove useful in applications.
 
 \begin{proposition} \label{L:6.4}
     The function $\To_{j}$ in \eqref{6.14} for $j=2, \dots, n$ can 
     alternatively by given by
  	\begin{equation}  \label{6.16}
\To_j(z)={\mathcal G}_{j-1}^{\frac{1}{2}}(I-zA)^{-1}{\mathcal G}_{j}^{-1}(zI-A^*)\Bo_j^{-1}.
\end{equation}
Hence $F_{\ell}(z)$ can alternatively be given for  $\ell = 1, 2, \dots, n-1$ by
\begin{align}
    F_{\ell}(z) & = C (I - zA)^{-(n-\ell)} 
    \cG^{-\frac{1}{2}}_{n-\ell,C,A} \To_{n+1-\ell}(z)  \notag \\
    & = C (I - zA)^{-(n+1-\ell)} \cG^{-1}_{n+1-\ell}(zI - A^{*}) 
    \Bo_{n+1-\ell}^{-1}.
  \label{Fell-new}
 \end{align}
    \end{proposition}
    
\begin{proof}
In the proof, we  shorten notation ${\mathcal G}_{C,A,j}$ to  ${\mathcal G}_j$.
We substitute the second equality from \eqref{6.13} into \eqref{6.14} to get 
\begin{align}
\To_j(z)=&-{\mathcal G}_{j-1}^{-\frac{1}{2}}A^*{\mathcal G}_{j}B_j+z{\mathcal 
G}_{j-1}^{\frac{1}{2}}(I-zA)^{-1}\Bo_j\notag\\
=&{\mathcal G}_{j-1}^{\frac{1}{2}}(I-zA)^{-1}\left[
-(I-zA){\mathcal G}_{j-1}^{-1}A^*{\mathcal G}_{j}+zI\right]\Bo_j\notag\\
=&{\mathcal G}_{j-1}^{\frac{1}{2}}(I-zA)^{-1}\left[z(I+A{\mathcal G}_{j-1}^{-1}A^*{\mathcal 
G}_{j})-{\mathcal G}_{j-1}^{-1}A^*{\mathcal G}_{j}\right]\Bo_j.\label{6.17}
\end{align}
We next observe the equalities
$$
{\mathcal G}_{j-1}^{-1}A^*{\mathcal G}_{j}\left(
I+A{\mathcal G}_{j-1}^{-1}A^*{\mathcal G}_{j}\right)^{-1}
={\mathcal G}_{j}^{-1}A^*{\mathcal G}_{j},
$$
and 
$$
{\mathcal G}_{j}+{\mathcal G}_{j}A {\mathcal G}_{j-1}^{-1}A^*{\mathcal G}_{j}
=({\mathcal G}_{C,A,j}^{-1}-A{\mathcal G}_{C,A,j}^{-1}A^*)^{-1},
$$
which both follow from \eqref{6.10}. We then use these equalities along with 
\eqref{6.17} and the first formula in \eqref{6.14} to get \eqref{6.16}:
\begin{align}
\To_j(z)=&{\mathcal G}_{j-1}^{\frac{1}{2}}(I-zA)^{-1}\left[zI-{\mathcal G}_{j}^{-1}A^*
{\mathcal G}_{j}\right]
(I+A{\mathcal G}_{j-1}^{-1}A^*{\mathcal G}_{j})\Bo_j\notag\\
=&{\mathcal G}_{j-1}^{\frac{1}{2}}(I-zA)^{-1}{\mathcal G}_{j}^{-1}(zI-A^*)({\mathcal
G}_{j}+{\mathcal G}_{j}A {\mathcal G}_{j-1}^{-1}A^*{\mathcal G}_{j})\Bo_j\notag\\
=&{\mathcal G}_{j-1}^{\frac{1}{2}}(I-zA)^{-1}{\mathcal G}_{j}^{-1}(zI-A^*)({\mathcal 
G}_{j}^{-1}-A{\mathcal G}_{j}^{-1}A^*)^{-1}\Bo_j\notag\\
=&{\mathcal G}_{j-1}^{\frac{1}{2}}(I-zA)^{-1}{\mathcal G}_{j}^{-1}(zI-A^*)\Bo_j^{-1}.\notag
\end{align}
\end{proof}

Given a shift invariant subspace $\cM$, let us refer to the function 
$F = \begin{bmatrix} F_{1} & \cdots & F_{n}\end{bmatrix}$ constructed 
as in Theorem \ref{T:BL} as the {\em partially isometric Bergman inner 
function} associated with $\cM$.

\subsection{The second approach: shift-invariant subspaces
contractively included in $\cA_n(\cY)$}  \label{S:c-i-sub}

Let us say that the Hilbert space  $\cM$ is contractively included in
the Hilbert space $\cH$ if $\cM \subset \cH$ as sets and moreover
$ \|m\|_{\cH} \le \| m\|_{\cM}$ for all $m \in \cM$.
We also say that an $\cL(\cU,\cY)$-valued function $\Theta$ is a {\em contractive multiplier}
if the operator $M_{\Theta}: \, f(z) \mapsto \Theta(z) \cdot f(z)$ of multiplication
by $\Theta$ defines a contractive operator from  $\cA_n(\cU)$  to $\cA_n(\cY)$.

\begin{theorem} \label{T:4.1} 
A Hilbert space $\cM$ is such that
\begin{enumerate}
\item $\cM$ is contractively included in $\cA_n(\cY)$, i.e. $\cM 
\subset \cA_{n}(\cY)$ with $\| f\|_{\cA_{n}(\cY)} \le \| f \|_{\cM}$ 
for all $f \in \cM$,
\item $\cM$ is $S_n$-invariant: $S_n\cM\subset\cM$,
\item the operator $A=(S_n\vert_{\cM})^*$ is a strongly stable $n$-hypercontraction,
\end{enumerate}
if and only if there is a Hilbert space $\cU$ and a
contractive multiplier $\Theta$ so that
\begin{equation}  \label{4.0} 
\cM = \Theta\cdot \cA_n({\cU})
\end{equation}
with lifted norm
\begin{equation}  \label{4.1}
\| \Theta\cdot f \|_{\cM} = \| Q f \|_{\cA_n({\cU})}
\end{equation}
where $Q$ is the orthogonal projection onto $(\operatorname{Ker} \, M_{\Theta})^\perp$.
        \end{theorem}

\begin{proof}  We first verify sufficiency. Suppose that $\cM$ has the form \eqref{4.0} for a
contractive multiplier $\Theta$ with $\cM$-norm given by \eqref{4.1}.  Since
$\|M_{\Theta}\| \le 1$ it follows that
$$
\|\Theta f\|_{\cA_n(\cY)} =  \|M_\Theta Qf\|_{\cA_n(\cY)}\le \|Qf\|_{\cA_n(\cU)}=\|\Theta f\|_{\cM}
$$
i.e., (1) holds. From the intertwining equality $S_nM_{\Theta} =
M_{\Theta}S_n$, property (2) follows. The latter  intertwining equality
also implies $M_\Theta
S_n\vert_{\operatorname{Ker}M_\Theta}=0$ which can be written equivalently in terms of the
orthogonal projection $Q$ onto $(\operatorname{Ker}
M_{\Theta})^{\perp}\subset \cA_n({\cU})$ as $QS_n(I-Q)=0$. Thus, we have
\begin{equation}  \label{4.2}
Q S_n = Q S_nQ\quad \text{and}\quad S_n^{*} Q= QS_n^{*}Q.
\end{equation}
Furthermore, for every $f, \, g\in \cA_n(\cU)$, we have
\begin{align*}
\langle \Theta g, \, A\Theta f\rangle_{\cM}=&\langle S_n\Theta g, \, \Theta f\rangle_{\cM}
=\langle \Theta S_n g, \, \Theta f\rangle_{\cM}=\langle Q S_n g, \, f\rangle_{\cA_n(\cU)}\\
=&\langle Q S_n Qg, \, f\rangle_{\cA_n(\cU)}=\langle Q g, \, S_n^*Qf\rangle_{\cA_n(\cU)}=
\langle \Theta g, \, \Theta S_n^* Qf\rangle_{\cM},
\end{align*}
which implies that $A: \, \Theta f\to \Theta S_n^* Qf$. Iterating the latter formula gives
\begin{equation}
A^j: \, \Theta f\to \Theta S_n^{*j} Qf\quad\mbox{for}\quad j\ge 0.
\label{4.4}
\end{equation}
Since $S_n^*$ is a strongly stable $n$-hypercontraction on $\cA_n(\cY)$, we conclude from
\eqref{4.2},
\eqref{4.4}, and \eqref{4.1}  that
\begin{align*}
\lim_{N\to \infty}\|A^j\Theta f\|_{\cM}=&\lim_{N\to \infty}\|\Theta S_n^{*j} Qf\|_{\cM}\\
=&\lim_{N\to \infty}\|Q S_n^{*j} Qf\|_{\cA_n(\cU)}=\lim_{N\to \infty}\|S_n^{*j} Qf\|_{\cA_n(\cU)}=0
\end{align*}
and also
$$
\sum_{j=0}^k(-1)^j\bc{k \\ j}\|A^j\Theta f\|^2_{\cM}= 
\sum_{j=0}^k(-1)^j\bc{k \\ j}\|S_n^{*j}Qf\|^2_{\cA_n(\cU)}\ge 0
$$
for $k=1,\ldots,n$ which shows that $A$ is strongly stable $n$-hypercontraction on $\cM$
and therefore completes the proof of sufficiency.

\smallskip

Suppose now that the Hilbert space $\cM$ satisfies conditions (1), (2), (3) in the
statement of the theorem. Using hypothesis (2) we can define the operator $A:=(S_n\vert_{\cM})^*$
on $\cM$ and since it is $n$-hypercontractive by hypothesis (3), the operator $\Gamma_{A,n}(I)$
is positive semidefinite. Choose the coefficient Hilbert space $\cU$ so that
$$
\operatorname{dim} \cU = \operatorname{rank} \Gamma_{A,n}(I)
$$
and then choose the operator $C \colon \cM \to \cU$ so that
$$
C^{*}C = \Gamma_{n,A}(I)=\sum_{j=0}^n (-1)^j\bc{n \\ j}A^{*j}A^j\ge 0.
$$
Then $(C, A)$ is an $n$-isometric pair and, since $A$ is strongly stable by hypothesis (3).
it follows from  part (2) of Proposition \ref{P:2-1.1} that the observability operator
$$
\cO_{n,C, A} \colon f \mapsto C(I - zA)^{-n} f
$$
is an isometry from $\cM$ into $\cA_n(\cY)$. By part (1) of Theorem \ref{T:1.2}, we have
the intertwining equality $S_n^*\cO_{n,C,A} =\cO_{n,C,A}A$. Taking adjoints then gives
\begin{equation}  
\label{4.5}
(\cO_{n,C,A})^*S_n=A^*(\cO_{n,C,A})^*.
\end{equation}
The inclusion map $\iota \colon \cM \to \cA_n(\cY)$ is a contraction by hypothesis (1).
Moreover, $\iota A^*=S_n \iota: \, \cM\to \cA_n(\cY)$. Therefore the operator
$$
R = \iota \circ (\cO_{C, A,n})^{*} \colon \cA_n(\cU)\to \cA_n(\cY)
$$
is a contraction and
$$
RS_n=\iota (\cO_{n,C, A})^{*}S_n=\iota A^*(\cO_{n,C, A})^{*}=
S_n \iota (\cO_{n,C, A})^{*}=S_nR.
$$
Therefore (see \cite{olieot}) $R$ is a multiplication operator, i.e., there is a contractive
multiplier $\Theta$ so that $R = M_{\Theta}$. Since
$\cO_{n,C, A} \colon \cM \to \cA_n(\cY)$ is an isometry, it follows that
$\operatorname{Ran} (\cO_{n,C,A})^{*} = \cM$ and also that $\cM = \Theta \cdot
\cA_n(\cU)$ with $\cM$-norm given by \eqref{4.1}.  This completes the proof.
\end{proof}

\begin{remark} \label{R:c-i-ss}
    {\em  It is of interest to consider Theorem \ref{T:4.1}
    for the case where condition (1) is strengthened to
    \begin{itemize}
        \item[(1$^{\prime}$)] {\em $\cM$ is isometrically contained in
        $\cA_{n}(\cY)$}.
\end{itemize}
For the special case where $n=1$ (so $\cA_{1}(\cY)$ is the Hardy
space $H^{2}(\cY)$),  it is not difficult to see
that condition (3) is a consequence of (1$^{\prime}$) and (2).  Then
the representer $\Theta$ is a coisometric multiplier and
the proof of Theorem \ref{T:4.1} reduces essentially to that of
Rosenblum-Rovnyak (see \cite{RRbook}) for the case $n=1$.
The same approach has been adapted to the multivariable setting of the
Drury-Arveson space in Arveson \cite{Arv1998} and McCullough-Trent \cite{McCT2000} as well
as the freely noncommutative lift of the Drury-Arveson spaces to the
Fock space \cite[Theorem 2.14]{BBF1}.  For the case $n>1$,
there is a new phenomenon: it is not the case that (1$^{\prime}$) and (2) imply (3):
indeed, it is known that any Hilbert space operator can be recovered
up to unitary equivalence as the compression of the Bergman shift
$S_{2}$ to the orthogonal difference $\cM \ominus \cN$ of two nested  
$S_{2}$-invariant subspaces $\cN \subset \cM$; in particular, the
$2$-hypercontractivity property of $S_{2}^{*}$ is not preserved
when one considers $(S_{2}|_{\cM})^{*}$ for an $S_{2}$-invariant
subspace $\cM$. A second manifestation of  the inapplicability of
Theorem \ref{T:4.1} to general isometrically-included  
$S_{n}$-invariant subspaces of $\cA_{n}(\cY)$ is given by Theorem
\ref{T:BL}: $S_{n}$-invariant subspaces of $\cA_{n}(\cY)$ are modeled
by a coisometric multiplier acting on a direct-sum Bergman space $\bigoplus_{j=1}^{n}
\cA_{j}(\cU_{j})$ rather than just on a single Bergman space
$\cA_{n}(\cU)$.
}\end{remark}

\subsection{The third approach:  Bergman-inner families} 

Given an $S_{n}$-invariant subspace $\cM$ of $\cA_{n}(\cY)$, from the 
fact that
$$
\bigcap_{k ge o} S_{n}^{k} \cM \subset \bigcap_{k \ge 0} S_{n}^{k} 
\cA(\cY) = \{0\}
$$
one can see that we always have the orthogonal-sum decomposition for 
$\cM$:
\begin{equation}  \label{cmdec}
 \cM = \bigoplus_{k=0}^{\infty} \left( S^{k} \cM \ominus S^{k+1} \cM 
 \right).
\end{equation}
The space $S^k_n\cM$ has reproducing kernel $(z,\zeta)\to
z^k\overline{\zeta}^k k_\cM(z,\zeta)$, but since the operator $S^k_n: \, \cA_n(\cY)\to \cA_n(\cY)$
is only a contraction  and not an isometry if $n>1$, the metric induced by the above reproducing kernel
is different from the metric of $\cA_n(\cY)$. To get the reproducing 
kernel for $S^k_n\cM$ consistent with the 
metric of $\cA_n(\cY)$,  we need to consider backward-shifted 
version $ \Ob_{n,k,C,A}$ of the observability operator $\cO_{n,C,A}$ and the 
backward-shifted version $\Gr_{n,k,C,A}$ of the observability gramian $\cG_{n,C,A}$ 
introduced  in Section \ref{S:obs} above. It turns out that the space  $(S^k_{n} \cM)^{\perp}$ has a nice
characterization in terms of the operator ${\Ob}_{n,k,C,A}$.
\begin{proposition}   \label{P:SMperp}
    The space $(S^k_{n} \cM)^{\perp}$ is characterized as
 \begin{equation}   \label{SMperp}
     \left( S^k_{n} \cM\right)^{\perp} = \left(\bigoplus_{j=0}^{k-1}S_n^j\cY\right)\bigoplus
    S_n^k \operatorname{Ran} {\Ob}_{n,k,C,A}
 \end{equation}
 where we identify the first term with the subspace of polynomials of degree at most $k-1$  
in $\cA_{n}(\cY)$.
 \end{proposition}

\begin{proof}
We wish to characterize all functions $f(z)={\displaystyle
\sum_{j=0}^{\infty} f_{j} z^{j}}$ which are orthogonal to $S_{n}^k \cM$ 
in $\cA_{n}(\cY)$. We may write 
$f(z)  = p(z) + z^k \widetilde f(z)$ where $p(z)={\displaystyle\sum_{j=0}^{k-1} f_{j} z^{j}}$
and $\widetilde f \in \cA_{n}(\cY)$.
Clearly polynomials of degree at most $k-1$ are orthogonal to
$S^k_{n} \cM$, so it suffices to characterize which functions of the   
form $z^k \widetilde f(z)$ are orthogonal to $S^k_{n} \cM$. To this end, observe that 
$S_n^k\widetilde{f}$ is orthogonal to $S^k_{n} \cM$ if and only if the function 
$(S^k_{n})^*S^k_{n}\widetilde{f}$ belongs to $\cM^\perp= \operatorname{Ran}\cO_{n,C,A}$.
It follows from the formulas \eqref{3.1a} and \eqref{defmu} that 
$$
(S^k_{n})^*S^k_{n}: \; \sum_{j=0}^\infty \widetilde{f}_jz^j\to \sum_{j=0}^\infty \frac{\mu_{n,j+k}}{\mu_{n,j}} 
\, \widetilde{f}_jz^j.
$$
We thus conclude that $S_n^k\widetilde{f}$ is orthogonal to $S^k_{n} \cM$ if and only if there exists a vector 
$x\in\cX$ such that 
$$
\sum_{j=0}^\infty \frac{\mu_{n,j+k}}{\mu_{n,j}}
\, \widetilde{f}_jz^j=C(I-zA)^{-n}x=\sum_{j=0}^\infty \left(\frac{1}{\mu_{n,j}}\cdot CA^jx\right)z^j.
$$ 
Equating the corresponding Taylor coefficients gives
$$
\widetilde{f}_j=\frac{1}{\mu_{n,j+k}}\cdot CA^jx=\bcs{n+j+k-1\\ j+k}\cdot CA^jx\quad\mbox{for all}\quad j\ge 0,
$$
and therefore,
$$
\widetilde{f}(z)=\sum_{j=0}^\infty \widetilde{f}_jz^j=\sum_{j=0}^\infty \left(\bcs{n+j+k-1\\ j+k}\cdot 
CA^jx\right)z^j={\Ob}_{n,k,C,A}x,
$$
by \eqref{defdelta}. Thus, $\widetilde{f}\in\operatorname{Ran} {\Ob}_{n,k,C,A}$.
As the analysis is necessary and sufficient, the result follows.
    \end{proof}

    With characterization \eqref{SMperp} in hand, it is straightforward to derive the   
    kernel function for the space $S_{n}^k \cM$ with respect to the   
    metric inherited from $\cA_{n}(\cY)$.

   \begin{proposition}  \label{P:6.2}  Let $\cM$ be a closed shift-invariant subspace of
        $\cA_n(\cY)$ with reproducing kernel
$k_{\cM}$ given by \eqref{kM}. Then the reproducing kernel functions for the closures of  $S^k_n\cM$
and of $S^{k}_n\cM\ominus S^{k+1}_n\cM$ in the metric of $\cA_n(\cY)$ are given by
\begin{equation}
k_{S^k_n\cM}(z,\zeta) = z^k\overline\zeta^k
\left(\sum_{\ell=1}^{n}
\frac{\bcs{\ell+k-2 \\ \ell-1}\cdot I_{\cY}}{(1-z\overline\zeta)^{n-\ell+1}}
-CR_{n,k}(zA)\Gr_{n,k,C,A}^{-1}R_{n,k}(\zeta A)^*C^*\right),
\label{kscm}
\end{equation}
\begin{align}
k_{S^{k}_n\cM\ominus S^{k+1}_n\cM}(z,\zeta)=& z^{k}\overline\zeta^{k}
\left(\bcs{n+k-1 \\ k}\cdot I_\cY
-CR_{n,k}(zA)\Gr_{n,k,C,A}^{-1}R_{n,k}(\zeta A)^*C^* 
\right. \notag \\
& \left.  +z\overline\zeta CR_{n,k+1}(zA)\Gr_{n,k+1,C,A}^{-1}R_{n,k+1}(\zeta A)^*C^*
\right).
\label{kdif}
\end{align}
\end{proposition}

\begin{proof} By Theorem \ref{T:frakOkernel}, the kernel $\boldsymbol{\mathfrak{K}}$
    for the space $S^k_{n} \operatorname{Ran} {\Ob}_{n,k,C,A}$ (with inner
    product induced by $\cA_{n}(\cY)$) is given by \eqref{deffrakk}.
From the formula \eqref{SMperp} for $(S_{n} \cM)^{\perp}$ we deduce that
\begin{align}
    S^k_{n} \cM & = \left(\bigoplus_{j=0}^{k-1}S_n^j\cY\right)^{\perp} \bigcap 
\left(S^k_{n} \operatorname{Ran}{\Ob}_{n,k,C,A} \right)^{\perp} \notag  \\
    & = S^k_{n}\cA_{n}(\cY) \ominus ( S^k_{n} 
    \operatorname{Ran}{\Ob}_{n,k,C,A}).
    \label{SM}
\end{align}
Since the reproducing kernel for the subspace $S_n^j\cY$ of $\cA_n(\cY)$ is 
$z^j\overline{\zeta}^j\bcs{n+j-1\\ j}\cdot I_{\cY}$, we deduce that 
$S^k_{n} \cA_{n}(\cY)=\cA_n(\cY)\ominus {\displaystyle\left(\bigoplus_{j=0}^{k-1}S_n^j\cY\right)}$
has reproducing kernel
\begin{align*}
    k_{S^k_{n} \cA_{n}(\cY)}(z, \zeta) & =\frac{I_{\cY}}{(1 - z \overline{\zeta})^{n}} 
- \sum_{j=0}^{k-1}\bcs{n+j-1\\ j}z^j\overline{\zeta}^j I_{\cY}\notag\\
& = \sum_{j=k}^{\infty} \bcs{n+j-1\\ j} z^{j} \bar{\zeta}^{j} 
I_{\cY} = z^{k} \bar{\zeta}^{k} \left( \sum_{j=0}^{\infty} 
\bcs{n+j+k-1\\ j+k} z^{j} \bar{\zeta}^{j} I_{\cY} \right).
\end{align*}
If we now make use of the Chu-Vandermonde identity \eqref{ChuVan}, we can continue the 
computation as follows:
\begin{align}
    k_{S_{n}^{k} \cA_{n}(\cY)}(z, \zeta) & = z^{k} \bar{\zeta}^{k} \left( 
    \sum_{j=0}^{\infty} \sum_{\ell=1}^{n} \bcs{\ell +k-2 \\ \ell-1} 
    \cdot  \bcs{n+j-\ell\\ j} z^{j} \bar{\zeta}^{j} \right) I_{\cY}  \notag \\
    & = z^{k} \bar{\zeta}^{k} \sum_{\ell=1}^{n} \bcs{\ell + k 
    -2\\ \ell-1} \left( \sum_{j=0}^{\infty} \bcs{n+j-\ell \\ j} 
    z^{j} \bar{\zeta}^{j} \right) I_{\cY}\notag \\ 
   & = z^{k} \bar{\zeta}^{k} \sum_{\ell=1}^{n} \frac{\bcs{\ell + k 
    -2\\ \ell -1}I_{\cY}}{(1 - z \bar{\zeta})^{n-\ell +1}}.
    \label{san}
\end{align}
From the representation for $S_{n}^{k} \cM$ in \eqref{SM},  we see that
\begin{equation}
    k_{S^k_{n} \cM} (z, \zeta) = k_{S^k_{n} \cA_{n}(\cY)}(z, \zeta) -
    \boldsymbol{\mathfrak K}_k(z, \zeta)
\label{jul1}   
\end{equation}
where $\boldsymbol{\mathfrak K}_{k}$ is the reproducing kernel for 
$S_{n}^{k} \operatorname{Ran} \Ob_{n,k,C,A}$.
Combining this with \eqref{san} and the formula 
\eqref{deffrakk} for $\boldsymbol{\mathfrak K}_{k}$ leads to the 
formula \eqref{kscm} for $k_{S_{n}^{k} \cM}$.  

\smallskip 

It remains to verify the formula \eqref{kdif}.  Toward this end, note that in view of 
\eqref{jul1} and \eqref{san}, 
\begin{align*}
k_{S^{k}_n\cM\ominus S^{k+1}_n\cM}(z,\zeta)   &  =  k_{S^{k}_{n} \cM} 
(z, \zeta)-k_{S^{k+1}_{n} \cM} (z,  \zeta)\\
&=k_{S^{k}_{n} \cA_{n}(\cY)}(z, \zeta)-k_{S^{k+1}_{n} \cA_{n}(\cY)}(z, \zeta)
-\boldsymbol{\mathfrak K}_{k}(z, \zeta)+\boldsymbol{\mathfrak 
K}_{k+1}(z, \zeta)  \\
& = k_{S_{n}^{k} \cY}(z, \zeta) -\boldsymbol{\mathfrak K}_{k}(z, \zeta)+\boldsymbol{\mathfrak 
K}_{k+1}(z, \zeta)   \\
& = z^{k}\overline\zeta^{k}\bcs{n+k-1 \\ k}\cdot I_\cY 
-\boldsymbol{\mathfrak K}_{k}(z, \zeta)+\boldsymbol{\mathfrak 
K}_{k+1}(z, \zeta). 
\end{align*}
Plugging \eqref{deffrakk} into this last expression leads to 
\eqref{kdif} as wanted. 
\end{proof}
\begin{lemma}  \label{L:6.8}
Given an integer $k\ge 1$ and an exactly observable $n$-output stable pair
$(C,A)$, construct operators $B_k \in \cL(\cU_k, \cX)$ and $D_k \in \cL(\cU_k, \cY)$
as in Lemma \ref{L:5.6} and let $\Theta_{n,k}$ be the associated
function given by \eqref{m5}. Then the kernel \eqref{kdif} can be factored as
\begin{equation}
k_{S^{k}_n\cM\ominus S^{k+1}_n\cM}(z,\zeta)=z^{k}\overline{\zeta}^{k}\Theta_{n,k}(z)\Theta_{n,k}(\zeta)^*.
\label{id6}
\end{equation}
\end{lemma}
\begin{proof}  By Lemma \ref{L:frakcoisom}, identity \eqref{frakkerid} holds. Multiplying both parts 
of \eqref{frakkerid} by $z^k\bar\zeta^k$ and combining the obtained equality with 
\eqref{kdif} we easily get \eqref{id6}. 
\end{proof}

Let us say that a family of operator-valued functions $\{ 
\Theta_{n,k} \colon {\mathbb D} \to \cL(\cU_{k}, 
\cY)\}_{k=0}^{\infty}$ is a {\em Bergman-inner family} if, for each 
$k=0,1,2,\dots$, we have:
\begin{enumerate}
    \item $M_{\Theta_{n,k}} \colon S_{1}^{k} \cU_{k} \to \cA_{n}(\cY)$ 
    is  isometric,
    \item $M_{\Theta_{n,k}} \left( S_{1}^{k} \cU_{k} \right)$ is 
    orthogonal (in $\cA_{n}(\cY)$) to $\overline{\operatorname{span}} 
    \{ M_{\Theta_{n,\ell}} S_{1}^{\ell} \cU_{\ell} \colon \ell > k 
    \}$, and
    \item $S_{n}^{k+1} M_{\Theta_{n,k}} \cU_{k} \subset 
    \overline{\operatorname{span}} \{ S_{n}^{\ell} 
    M_{\Theta_{n,\ell}} \cU_{\ell} \colon \ell > k \}$.
  \end{enumerate}
If $\{ \Theta_{n,k}\}_{k \ge 0}$ is a Bergman-inner family and we set 
\begin{equation}   \label{tvmultop}
\cM = \bigoplus_{k=0}^{\infty} \Theta_{n,k} S_{1}^{k} \cU_{k} \subset 
\cA_{n}(\cY),
\end{equation}
it then follows that the multiplication operator
$$
M_{\Theta} = \begin{bmatrix} M_{\Theta_{0}} & M_{\Theta_{1}} & 
M_{\Theta_{2}} & \cdots \end{bmatrix}  \colon 
\bigoplus_{k=0}^{\infty} S_{1}^{k} \cU_{k} \to \cA_{n}(\cY)
$$
maps the {\em time-varying Hardy space}  
$H^{2}(\{\cU_{k}\}_{k \ge 0}): = {\displaystyle\bigoplus_{k=0}^{\infty} S_1^{k} 
\cU_{k}}$ (where ${\bf u}={\displaystyle\bigoplus_{k=0}^{\infty} z^{k} u_{k}} \in 
H^{2}(\{\cU_{k}\})$ is assigned the Hardy-space norm $\| {\bf u}\|^{2} = {\displaystyle\sum_{k=0}^{\infty} 
\| u_{k} \|^{2}}$) unitarily onto the  $S_{n}$-invariant subspace $\cM \subset \cA_{n}(\cY)$.
Putting all the pieces together, we arrive at the following converse 
of all these observations which amounts to our third analogue of the Beurling-Lax 
theorem for the Bergman space setting.

\begin{theorem}  \label{T:BL3}
    Let $\cM$ be a closed $S_{n}$-invariant subspace of 
    $\cA_{n}(\cY)$.  Then there is a Bergman inner family 
    $\{\Theta_{n,k}\}_{k \ge 0}$ so that $M_{\Theta} H^{2}(\{\cU\}_{k 
    \ge 0})$ (with $M_{\Theta}$ as in \eqref{tvmultop}).
    
    Furthermore, given the $S_{n}$-invariant subspace $\cM \subset 
    \cA_{n}(\cY)$, the associated representing Bergman-inner family 
    $\{\Theta_{n,k}\}_{k\ge 0}$ can be constructed according to the 
    following algorithm:
    \begin{enumerate}
	\item Set $\cX = \cM^{\perp}$ and define $A \in \cL(\cX)$ and 
	$C \in \cL(\cX, \cY)$ by
$$
  A = S_{n}^{*}|_{\cM^{\perp}},  \quad Cf = f(0) \quad\text{for}\quad f \in 
  \cM^{\perp}.
$$

\item Construct $\left[ \begin{smallmatrix} B_{k} \\ D_{k} 
\end{smallmatrix} \right]$ by solving the Cholesky factorization 
problem \eqref{pr6} in Lemma \ref{L:5.6}.

\item Set $\Theta_{n,k}(z) = \frac{D_{k}}{\mu_{n,k}} + z C 
R_{n,k+1}(zA) B_{k}$.
\end{enumerate}
    
  Alternatively, Bergman-inner families $\{\Theta_{n,k}\}_{k \ge 0}$ 
  can be constructed from an arbitrary strongly stable 
  $n$-hypercontraction as follows.  Let $A \in \cL(\cX)$ be any strongly 
  stable $n$-hypercontraction on a Hilbert space $\cX$ and choose $C 
  \in \cL(\cX, \cY)$ with dense range so that
  $$
  C^{*}C = (I - B_{A})^{n}[I_{\cX}]
  $$
  (One such choice is $C = D_{n,A} \colon \cX \to \cY$ where $D_{n,A} 
  = \left( \sum_{k=0}^{n} (-1)^{k} \binom{n}{k} A^{*k} A^{k} 
  \right)^{1/2}$ and $\cY = {\mathcal D}_{n,A} : = \overline{\operatorname{Ran}} 
  D_{n,A}$.)   Construct $B_{k}$, $D_{k}$ as in Step 2 above and 
  then $\Theta_{n,k}$ as in Step 3 above.  Then $\{\Theta_{n,k}\}_{k 
  \ge 0}$ is a Bergman-inner family and any Bergman-inner family 
  arises in this way. 
    \end{theorem}
    
    \begin{example}  \label{E:BL3}  {\em By way of an example, suppose 
	that $\cM = \{0\} \subset \cA_{n}(\cY)$, so $\cX: = 
	\cM^{\perp} = \cA_{n}(\cY)$.  Then
$$
A = S_{n}^{*}, \quad C = E \colon f \in \cA_{n}(\cY) \mapsto f(0) \in 
\cY.
$$
In the proof of part (3) of Proposition \ref{P:exactobs} (see 
\eqref{jul25}), we saw that
\begin{equation}  \label{jul25'}
     \Gr_{n,k,E,S_n^*}: \; \sum_{j=0}^\infty f_jz^j \to \sum_{j=0}^\infty
\frac{\mu_{n,j}}{\mu_{n,j+k}}f_{j} z^{j}.
\end{equation}
and hence
\begin{equation}
 \Gr_{n,k,E,S_n^*}^{-1}: \; \sum_{j=0}^\infty f_jz^j \to \sum_{j=0}^\infty
\frac{\mu_{n,j+k}}{\mu_{n,j}}f_{j} z^{j}.
\label{jul26}
\end{equation}
One can then work out that
$$
\Gr^{-1}_{n,k+1,E,S_{n}^{*}} - S_{n}^{*}  \Gr^{-1}_{n,,E,S_{n}^{*}} 
S_{n} = 0, \quad  \mu_{n,k}I_{\cY} - E \Gr^{-1}_{n,k,E,S_{n}^{*}} 
E^{*} = 0.
$$
Hence, solving for $\left[ \begin{smallmatrix} B_{k} \\ D_{k} 
\end{smallmatrix} \right]$ via the Cholesky factorization problem 
\eqref{pr6} (with $(C,A) = (E, S_{n}^{*})$) leads to $\left[ 
\begin{smallmatrix} B_{k} \\ D_{k} \end{smallmatrix} \right] = \left[ 
    \begin{smallmatrix} 0 \\ 0 \end{smallmatrix} \right]$ for all 
	$k$, as is to be expected from the conclusion of Theorem 
	\ref{T:BL3} with $\cM = \{0\}$.
}\end{example}

\subsection{The fourth approach: wandering-subspace Bergman-inner functions.} 

Suppose that $\cM \subset \cA_{n}(\cY)$ is $S_{n}$-invariant.  The 
subspace 
$$
\cE: = \cM \ominus S\cM
$$
has the property that $\cE \subset 
\cM$ and $\cE \perp S_{n} \cM$.  If it is the case that $\cE$ 
generates $\cM$ in the sense that $\cM = 
\overline{\operatorname{span}} \{ S_{n}^{k} \cE \colon k=0,1,2, \dots 
\}$, then one says that $\cM$ has the {\em wandering subspace 
property} (with wandering subspace equal to $\cE$). 
For the classical Hardy-space case $n=1$, one has that 
$\cE=\Theta_0\cdot \cU_{0}$ for an 
appropriate coefficient Hilbert space $\cU_{0}$ and the inner function 
$\Theta_0$  constructed in \eqref{6.8}.  Furthermore, $\cM$ admits
an orthogonal decomposition 
\begin{equation}
\cM=\bigoplus_{j=0}^\infty S_1^j\cE
\label{st13}
\end{equation}
and this decomposition coincides with that in \eqref{cmdec}.

\smallskip

If $n>1$, the subspace $S_n^j\cE$ is orthogonal to $\cE$ for all $j\ge 1$, but it is not orthogonal 
to $S_n^m\cE$ if $1\le m\neq j$. Thus, the best one can expect is the equality 
\begin{equation}
\cM=\bigvee_{j=0}^\infty S_n^j\cE
\label{st14}
\end{equation}
which indeed holds if $n=2,3$; see \cite{ars}, \cite{sh1}, \cite{sh2} 
(see also \cite{Sutton} for a different approach). Letting $k=1$ in formula
\eqref{kdif} and recalling that ${\Gr}_{n,0,C,A}=\cG_{n,C,A}$, 
${\Ob}_{n,0,C,A}={\cO}_{n,C,A}$ (see 
\eqref{st2}) and $\mu_{n,0}=1$, we conclude that the reproducing kernel for the subspace $\cE$ equals 
\begin{align}
k_\cE(z,\zeta)=& \,
I_{\cY}-C(I-zA)^{-n}\cG_{n,C,A}^{-1}(I_{\cX}-\overline{\zeta}A^*)^{-n}C^*\notag\\
&+z\overline{\zeta}C\left(\sum_{j=1}^n(I_{\cX}-zA)^{-j}\right)\Gr_{n,1,C,A}^{-1}\left(\sum_{j=1}^n
(I_{\cX}-\overline{\zeta}A^*)^{-j}\right)C^*.
\label{kce}
\end{align}

Following \cite{oljfa, olaa} we say that the function $\Theta$ is a 
{\em wandering-subspace Bergman-inner function}  
whenever 
 \begin{enumerate}
	\item $M_{\Theta} \colon \cU_{0} \to A_{n}(\cY)$ is 
	isometric, and
	\item $\Theta \cdot \cU_{0}$ is orthogonal to $S_{n}^{\ell} 
	\Theta \cdot \cU_{0}$ for $\ell \ge 1$.
\end{enumerate}
In case $M_{\Theta} \cU_{0}  = \cE$ is the wandering subspace for the 
$S_{n}$-invariant subspace $\cM$, then we have the Beurling-Lax-type 
representation of $\cM$ as the closure of $\Theta \cdot {\mathbb C}[z] 
\otimes \cU_{0}$, where ${\mathbb C}[z]$ denotes the algebra of 
polynomials with coefficients in ${\mathbb C}$.  Construction of a 
Beurling inner function for the wandering subspace $\cE$ of $\cM$ 
amounts to focusing on the first element $\Theta_{0}$ in the 
Bergman-inner family associated with $\cM$.
Specifying Lemma \ref{L:6.8} for the case $k=0$ then leads to the 
following.

\begin{theorem}  \label{P:6.4u}
    Given a strongly stable $n$-output-pair $(C,A)$,  then there exist operators
    $B \in \cL(\cU, \cX)$ and $D \in \cL(\cU, \cY)$ which solve the
    Cholesky factorization problem
 \begin{equation}  \label{pr6u}
\begin{bmatrix}B \\ D\end{bmatrix}\begin{bmatrix}B^* & D^*\end{bmatrix}=
\begin{bmatrix}\Gr_{n,1,C,A}^{-1} & 0 \\ 0 & I_{\cY}\end{bmatrix}-
\begin{bmatrix}A \\ C\end{bmatrix}\cG_{n,C,A}^{-1}\begin{bmatrix}A^* & C^*\end{bmatrix}.
\end{equation}
Moreover, if $\Theta$ is defined by
\begin{equation}
\Theta(z)=D+zC\sum_{j=1}^n(I-zA)^{-j}B
\label{7.6u}
\end{equation}
where $\left[ \begin{smallmatrix} B \\ D \end{smallmatrix} \right]$
solves \eqref{pr6} (with $k=1$), then:
\begin{enumerate}
\item The factorization $k_\cE(z,\zeta)=\Theta(z)\Theta(\zeta)^*$ holds and therefore,
the multiplication operator $M_{\Theta}$ maps $\cU$ onto $\cE$ unitarily.

\item The subspace $\cE$ is orthogonal to $S_n^{k}\Theta_k\cdot\cU$ for every $k\ge 1$.

\item $M_{\Theta}$ is a contractive multiplier from the Hardy space 
$H^{2}(\cU_{0})$ into the Bergman space $\cA_{n}(\cY)$.  Hence, if 
$\cM$ has the wandering subspace property, we conclude that $\cM$ has 
the Beurling-Lax-type representation
$$
 \cM = \text{$\cA_{n}(\cY)$-closure of } M_{\Theta} H^{2}(\cU).
 $$
\end{enumerate}
\end{theorem}

We note that the last statement in the theorem is a consequence of 
part (2b) of Lemma \ref{L:5.1}. We also note that 
realization formula \eqref{7.6u} appears in \cite{oljfa, olaa}) in slightly different 
terms.

\subsection{The case of zero-based shift-invariant subspaces with one 
zero}  \label{S:alpha}

In this concluding subsection we illustrate all four approaches by applying them 
to the simplest shift-invariant (zero-based) subspace 
\begin{equation}
\cM=\left\{f\in \cA_n: \; f(\alpha)=0\right\}
\label{ad1}
\end{equation}
of  $\cA_n$ where $\alpha$ is a fixed point in ${\mathbb D}$. If $n=1$, then 
$$
\cM=b_\alpha(z) \cA_1,\quad\mbox{where}\quad b_\alpha(z)=\frac{z-\alpha}{1-z\overline{\alpha}}.
$$
If $n>1$ and $\alpha=0$, the latter representation still holds. Let us assume that $\alpha\neq 0$.
The one-dimensional space $\cM^\perp$ is spanned by the function ${\bf 
k}_n(z,\alpha)=(1-z\bar\alpha)^{-n}$ and  thus 
$\cM^\perp={\rm Ran} \, \cO_{n,C.A}$ with $A=\overline{\alpha}$ and any $C\neq 0$. We choose
$C$ to make the gramian $\cG_{n,C,A}$ equal one. Thus, 
\begin{equation}
A=\overline{\alpha}, \qquad C=(1-|\alpha|^2)^{\frac{n}{2}}.
\label{ad2}
\end{equation}
Then a simple calculation reveals that
\begin{equation}
\cG_{j,C,A}=(1-|\alpha|^2)^{n-j}\quad\mbox{for}\quad j=0,\ldots,n
\label{ad3}.
\end{equation}
An induction argument making use of the weighted Stein identity \eqref{7.3} combined with 
the initial condition $\Gr_{n,0,C,A} = \Gr_{n,C,A} = 1$ then leads us 
to the formula for the shifted gramians:
\begin{equation}
\Gr_{n,k,C,A}=\frac{1}{|\alpha|^{2k}}\cdot\left(1-(1-|\alpha|^2)^{n}\cdot
\sum_{j=0}^{k-1}\bcs{n+j-1\\ j}|\alpha|^{2j}
\right)\quad\mbox{for}\quad k\ge 1.
\label{ad4}   
\end{equation}
We may also relate the numbers \eqref{ad4} to the function $R_{n,k}$ defined in \eqref{07}:
\begin{equation}
\Gr_{n,k,C,A}=(1-|\alpha|^2)^nR_{n,k}(|\alpha|^2).
\label{ad4a}
\end{equation}
Formula \eqref{kM} takes the form 
$$
k_\cM(z,\zeta)
=\frac{1}{(1-z\overline{\zeta})^n}-\frac{(1-|\alpha|^2)^n}{(1-z\overline{\alpha})^n(1-\alpha
\overline{\zeta})^n}.
$$

\smallskip 

\textbf{The first approach:} Substituting \eqref{ad2} and \eqref{ad3} into \eqref{6.13} gives
$$
B_j=(1-|\alpha|^2)^{\frac{j+1-n}{2}},\quad D_j=-\alpha\quad\mbox{for}\quad j=1,\ldots,n.
$$
We now get from \eqref{defFell}
$$
F_n(z)=-\alpha+z\frac{1-|\alpha|^2}{1-z\overline{\alpha}}=b_\alpha(z),
$$
and we get from \eqref{Fell-new}
$$
F_\ell(z)=\frac{(z-\alpha)(1-|\alpha|^2)^{\frac{n-\ell}{2}}}{(1-z\overline{\alpha})^{n+1-\ell}}
=b_\alpha(z)\cdot \left(\frac{\sqrt{1-|\alpha|^2}}{1-z\overline{\alpha}}\right)^{n-\ell}
\quad\mbox{for}\quad \ell=1,\dots,n-1.
$$
The identity \eqref{sumker} now amounts to
\begin{align*}
k_\cM(z,\zeta)
&=\frac{1}{(1-z\overline{\zeta})^n}-\frac{(1-|\alpha|^2)^n}{(1-z\overline{\alpha})^n(1-\alpha
\overline{\zeta})^n}\\&=\sum_{j=0}^{n-1}\frac{(1-|\alpha|^2)^j \,
b_\alpha(z)\overline{b_\alpha(\zeta)}}{(1-z\overline{\alpha})^{j}(1-\alpha
\overline{\zeta})^{j}(1-z\overline{\zeta})^{n-j}}
\end{align*}
and a related conclusion is that a function $f$
belongs to $\cM$ if and only if at admits a representation
\begin{equation}
f(z)=b_\alpha(z)\cdot \sum_{j=0}^{n-1} \left(\frac{\sqrt{1-|\alpha|^2}}{1-z\overline{\alpha}}
\right)^{j}\cdot g_{n-j}(z)\quad\mbox{for some}\quad g_k\in \cA_k.
\label{ad5}
\end{equation}

\smallskip

\textbf{The second approach:}
Since the function $z\mapsto (1-z\overline{\alpha})^{-j}$ belong to $H^\infty$ for all $j$, it follows that 
the factor ${\displaystyle \sum_{j=0}^{n-1} \left(\frac{\sqrt{1-|\alpha|^2}}{1-z\overline{\alpha}}
\right)^{j}\cdot g_{n-j}(z)}$ in \eqref{ad5} belongs to $\cA_n$. Therefore, the subspace $\cM$ can be represented 
as 
 $\cM=b_\alpha \cdot\cA_n$; this observation illustrates the second approach above. In this case, $\cM$ with lifted 
norm 
\eqref{4.1} is contractively included in $\cA_n$.

\smallskip

\textbf{The third approach:} We first find $B_k$ and $D_k$ satisfying \eqref{wghtcoisom} for 
$A$ and $C$ as in \eqref{ad2} and hence, with $\Gr_{n,k,C,A}$ and $\Gr_{n,k+1,C,A}$ defined as in 
\eqref{ad4}. We are seeking $B_k$ and $D_k$ subject to
\begin{align}
|B_k|^2&=\frac{1}{\Gr_{n,k+1,C,A}}-\frac{|\alpha|^2}{\Gr_{n,k,C,A}}
=\frac{(1-|\alpha|^2)^n}{\mu_{n,k}\Gr_{n,k,C,A}\Gr_{n,k+1,C,A}},\label{ad6}\\
|D_k|^2&=\mu_{n,k}-\frac{(1-|\alpha|^2)^n}{\Gr_{n,k,C,A}}=\mu_{n,k}\cdot
\frac{|\alpha|^2\Gr_{n,k+1,C,A}}{\Gr_{n,k,C,A}},\label{ad7}\\
B_k\overline{D}_k&=-\frac{\overline{\alpha}(1-|\alpha|^2)^{\frac{n}{2}}}{\Gr_{n,k,C,A}}.\label{ad8}
\end{align}
Note that the second equalities in \eqref{ad7} and \eqref{ad8} follow 
from \eqref{ad4}; alternatively one can use \eqref{ad4a} combined 
with \eqref{difquotRnk}. The essentially 
unique choice of $B_k$ and $D_k$ satisfying \eqref{ad6}-\eqref{ad8} is the following:
\begin{equation}
B_k=-\frac{\overline{\alpha}}{|\alpha|}\ 
\cdot\frac{(1-|\alpha|^2)^\frac{n}{2}}{\sqrt{\mu_{n,k}\Gr_{n,k}\Gr_{n,k+1}}},
\quad D_k=|\alpha|\cdot\sqrt{\frac{\mu_{n,k}\Gr_{n,k+1}}{\Gr_{n,k}}}.
\label{ad9}
\end{equation}
Now we arrive at the formula for $\Theta_{n,k}$:
\begin{align}
\Theta_{n,k}(z)&=\mu_{n,k}^{-1}D_k+zCR_{n,k+1}(zA)B_k\\
&=|\alpha|\cdot\sqrt{\frac{\Gr_{n,k+1}}{\mu_{n,k}\Gr_{n,k}}}
-\frac{z\overline{\alpha}(1-|\alpha|^2)^n 
R_{n,k+1}(z\overline{\alpha})}{|\alpha|\sqrt{\mu_{n,k}\Gr_{n,k}\Gr_{n,k+1}}}\notag\\
&=\frac{1}{|\alpha|\sqrt{\mu_{n,k}\Gr_{n,k}\Gr_{n,k+1}}}\cdot\left(
|\alpha|^2\Gr_{n,k+1}-(1-|\alpha|^2)^nz\overline{\alpha}R_{n,k+1}(z\overline{\alpha})\right)\notag\\
&=\frac{(1-|\alpha|^2)^n}{|\alpha|\sqrt{\mu_{n,k}\Gr_{n,k}\Gr_{n,k+1}}}
\cdot\left(|\alpha|^2R_{n,k+1}(|\alpha|^2)-
z\overline{\alpha}R_{n,k+1}(z\overline{\alpha})\right)\notag\\
&=\frac{(1-|\alpha|^2)^n}{|\alpha|\sqrt{\mu_{n,k}\Gr_{n,k}\Gr_{n,k+1}}}
\cdot\left(R_{n,k}(|\alpha|^2)-R_{n,k}(z\overline{\alpha})\right),\label{ad10}
\end{align}
where we used \eqref{ad4a} for the fourth equality and \eqref{difquotRnk} for the fifth.
For $k=0$, the latter formula gives
\begin{align}
\Theta_{n,0}(z)&=\frac{(1-|\alpha|^2)^n}{|\alpha|\sqrt{\mu_{n,0}\Gr_{n,0}\Gr_{n,1}}}
\cdot\left((1-|\alpha|^2)^{-n}-(1-z\overline{\alpha})^{-n}\right)\notag\\
&=\frac{1}{|\alpha|\sqrt{\Gr_{n,1}}}\left(1-\left(\frac{1-|\alpha|^2}{1-z\overline{\alpha}}\right)^n\right)\notag\\
&=-\frac{\overline{\alpha}b_\alpha(z)}{\sqrt{1-(1-|\alpha|^2)^n}}\cdot \sum_{j=0}^{n-1}
\left(\frac{1-|\alpha|^2}{1-z\overline{\alpha}}\right)^j. 
\label{ad11}
\end{align}
For $k\ge 1$, we use the identity
$$
\frac{1}{(1-|\alpha|^2)^{j}}-\frac{1}{(1-z\overline{\alpha})^{j}}=-\overline{\alpha}b_\alpha(z)\cdot
\sum_{r=1}^{j}\frac{1}{(1-|\alpha|^2)^{r}(1-z\overline{\alpha})^{j-r}}
$$
and formula \eqref{07-1} to get
\begin{align*} 
& R_{n,k}(|\alpha|^2)-R_{n,k}(z\overline{\alpha})=\sum_{\ell=1}^n\bcs{\ell+k-2\\ \ell-1}\left[
\frac{1}{(1-|\alpha|^2)^{n+1-\ell}}-\frac{1}{(1-z\overline{\alpha})^{n+1-\ell}}\right]\\
& \quad =-\overline{\alpha}b_\alpha(z)\sum_{\ell=1}^n\bcs{\ell+k-2\\ \ell-1}\sum_{r=1}^{n+1-\ell}
\frac{1}{(1-|\alpha|^2)^{r}(1-z\overline{\alpha})^{n+1-\ell-r}}\\
& \quad =-\overline{\alpha}b_\alpha(z)\sum_{\ell=1}^n\bcs{\ell+k-2\\ \ell-1}
\sum_{j=0}^{n-\ell} \frac{1}{(1 - |\alpha|^{2})^{n+1-j\ell} 
(1 - z \overline{\alpha})^{j}} \\
& \quad =-\overline{\alpha}b_\alpha(z)\sum_{j=0}^{n-1}\left(\sum_{\ell=1}^{n-j}\bcs{\ell+k-2 \\ 
\ell-1}\frac{1}{(1-|\alpha|^2)^{n+1-j-\ell}}\right)\cdot \frac{1}{(1-z\overline{\alpha})^{j}}\\
& \quad =-\frac{\overline{\alpha}b_\alpha(z)}{(1-|\alpha|^2)^n}\cdot 
\sum_{j=0}^{n-1}\left(\sum_{\ell=0}^{n-j-1}\bcs{\ell+k-1 \\
\ell}(1-|\alpha|^2)^{\ell}\right)\cdot\left(\frac{ 1-|\alpha|^2}{1-z\overline{\alpha}}\right)^j,
\end{align*} 
which being substituted into \eqref{ad10} gives 
$$
\Theta_{n,k}(z)=-\frac{\overline{\alpha}b_\alpha(z)}{|\alpha|\sqrt{\mu_{n,k}\Gr_{n,k}\Gr_{n,k+1}}}
\cdot\sum_{j=0}^{n-1}\left(\sum_{\ell=0}^{n-j-1}\bcs{\ell+k-1 \\
\ell}(1-|\alpha|^2)^{\ell}\right)\cdot\left(\frac{ 1-|\alpha|^2}{1-z\overline{\alpha}}\right)^j.
$$

\smallskip

\textbf{The fourth approach:}  Construction of the Bergman inner 
function associated with the wandering subspace for $\cM$ amounts to 
the construction of $\Theta_{0}$ already derived above in \eqref{ad11}.
It is easily checked that the formula \eqref{ad11} agrees with  
 \cite[formula (6) page 125]{DS} (with $n=2$ in \eqref{ad11}), and with the formula for $G_{a}(z)$ in 
\cite[page 58]{HKZ} (with $\alpha = n-2$), up to a constant 
multiplicative factor.   The derivation in both \cite{DS} and 
\cite{HKZ} is based on the work of Hedenmalm \cite{hend1} whereby 
Bergman inner functions are produced as solutions of an appropriate 
extremal problem.  We plan to discuss how our state-space methods can be used 
to solve such extremal problems directly in a future publication.

\section{Connections with time-varying linear systems theory}  
\label{S:syscon}

Here we make more explicit the connections of a Bergman-inner family 
$\{\Theta_{n,k}\}_{k \ge 0}$ and the associated realization formulas 
(see Step 3 in the algorithm in Theorem \ref{T:BL3}) with the theory 
of conservative/dissipative time-varying linear systems as presented 
e.g.~ in \cite{ABP} and \cite{OT100}.

We suppose that we are given an exactly $n$-observable pair $(C,A)$ ($A \in 
\cL(\cX)$, $C \in \cL(\cX, \cY)$).
Then Proposition \ref{P:exactobs} assures us that all the gramians 
$\Gr_{n,k,C,A}$ ($k = 0,1,2,\dots$) are strictly positive definite.  Let us introduce Hilbert spaces
$\cX_{k}$  ($k = 0,1,2, \dots$) with $\cX_{k} = \cX$ for all $k$ but 
with $\cX_{k}$ given the inner product induced by the $k$-th shifted 
gramian:
$$
  \langle x, x' \rangle_{\cX_{k}}:= \langle \Gr_{n,k,C,A} x, x' 
  \rangle_{\cX}.
$$
Similarly we set $\cY_{k} = \cY$ with $\cY_{k}$-inner product given by
$$
  \langle y, y' \rangle_{\cY_{k}}:= \left\langle \binom{n+k-1}{k} y, 
  y' \right\rangle_{\cY}.
$$
In addition suppose that operators $B_{k} \in \cL(\cU_{k}, \cX)$ and 
$D_{k} \in \cL(\cU_{k}, \cY)$ are constructed as in Step 2 of the 
algorithm in Theorem \ref{T:BL3}.  We then let $\bU_{k}$ be the 
colligation matrix
\begin{align}
  \bU_{k} & =  \begin{bmatrix} \bA_{k} & \bB_{k} \\ \bC_{k} & \bD_{k} 
\end{bmatrix}  := \begin{bmatrix} \frac{k+n}{k+1} A & \binom{k+n}{k+1} 
B_{k} \\ C & \binom{k+n-1}{k} D_{k} \end{bmatrix}  \notag \\
& = \begin{bmatrix} \frac{k+n}{k+1} I & 0 \\ 0 & I \end{bmatrix}
\begin{bmatrix} A & B_{k} \\ C & D_{k} \end{bmatrix} \begin{bmatrix} 
    I & 0 \\ 0 & \binom{k+n-1}{k} I_{\cU_{k}} \end{bmatrix}
\colon \begin{bmatrix} \cX_{k} \\ \cU_{k} \end{bmatrix} \to 
\begin{bmatrix} \cX_{k+1} \\ \cY \end{bmatrix}.
    \label{bUk}
    \end{align}
Here $\bA_{k} = \frac{k+n}{k+1} A$ with $A$ considered as acting from $\cX_{k}$ into 
$\cX_{k+1}$ and similarly $C_{k} = C $ but considered as acting from 
$\cX_{k}$ into $\cY$ while $B_{k}$ is considered as acting from 
$\cU_{k}$ into $\cX_{k}$.  The result of the construction in Theorem 
\ref{T:BL3} is that each $U_{k}: = \left[ \begin{smallmatrix} A & 
B_{k} \\ C & D_{k} \end{smallmatrix} \right]$ is unitary from $\cX_{k} \oplus 
\cU_{k}$ to $\cX_{k+1} \oplus \cY_{k}$.  One can then study the 
associated conservative discrete-time time-varying linear system
\begin{equation}   \label{tvsys}
\Sigma_{\{U_{k}\}_{k \ge 0}}\colon \left\{
\begin{array}{rcl} x(k+1) & = & A_{k} x(k) + B_{k} u(k) \\
       y(k) & = & C_{k} x(k) + D_{k} u(k).
\end{array}  \right.
\end{equation}
If we specify an 
initial condition $x(0) = 0$, then the input string $\{u(k)\}_{k \ge 
0}$ recursively uniquely determined via the system equations 
\eqref{tvsys} an output string $\{y(k)\}_{k \ge 0}$, thereby defining 
a linear operator $T_{\{\bU_{k}\}} \colon \{u(k)\}_{k \ge 0} \mapsto 
\{y(k) \}_{k \ge 0}$, the {\em input-output operator} for the linear 
system $\Sigma_{\{\bU_{k}\}_{k \ge 0}}$.
Such systems (with $k$ running over all of ${\mathbb Z}$ rather than 
just ${\mathbb Z}_{+}$) are studied in \cite{ABP} from the point of 
view of model theory for a sequence of contraction operators $T_{k} = 
A_{k}^{*} \colon \cX_{k+1} \to \cX_{k}$ and associated time-varying 
unitary dilation theory and Lax-Phillips scattering theory.   However 
the time-varying system associated with the realization of our 
Bergman-inner family of transfer functions is not given by 
$\{U_{k}\}$ but rather by $\{\bU_{k}\}_{k \ge 0}$, namely:
\begin{equation}   \label{bUk-sys}
\Sigma_{\{ \bU_{k}\}_{k \ge 0}} \colon 
\left\{ \begin{array}{rcl}
x(k+1) & =  & \bA_{k} x(k)  +  \bB_{k}  u(k) \\
  y(k) & = & \bC_{k} x(k) + \bD_{k} u(k).
  \end{array}   \right.
\end{equation}
where $\bU_{k} = \left[ \begin{smallmatrix} \bA_{k} & \bB_{k} \\ 
\bC_{k} & \bD_{k} \end{smallmatrix} \right]$ is the tweak of $U_{k}$ 
given by \eqref{bUk}. The input-output operator $T_{\{\bU_{k}\}_{k 
\ge 0}}$ associated with the system \eqref{bUk-sys} gives rise 
to an input-output operator $T_{\{\bU_{k}\}_{k \ge 0}} \colon 
\ell^{2}(\{\cU_{k}\}_{k \ge 0}) \to \ell^{2}(\cY)$ with image space
$T_{\{\bU_{k}\}_{k \ge 0}} \left( \ell^{2}(\{\cU_{k}\}_{k \ge 0}) 
\right)$ exactly  equal the set of all Taylor-coefficient strings 
$\{y(k) \}_{k \ge 0}$ for which the associated $Z$-transform $y(z) = 
\sum_{k=0}^{\infty} y(k)  z^{k}$ is in the prescribed closed $S_{n}$-invariant 
subspace $\cM$ of $\cA_{n}(\cY)$ (assuming that $\cM$ is chosen so 
that it is generated by its wandering subspace).  Moreover 
$T_{\{\bU\}_{k \ge 0}}$ is isometric if the output string 
$\{y(k)\}_{k \ge 0}$ is given the Bergman-space norm of its 
$Z$-transform $\| \sum_{k=0}^{\infty} y(k) z^{k} \|_{\cA_{n}(\cY)}$.
It is amusing that this tweak $U_{k} \mapsto \bU_{k}$ as in \eqref{bUk} of a 
standard conservative time-varying linear system 
$\Sigma_{\{U_{k}\}_{k \ge 0}}$  is the precise object required to get 
a realization theory for a Bergman-inner family $\{ \Theta_{k}\}_{k 
\ge 0}$.


\begin{thebibliography}{99}
    
\bibitem{aglerhyper}
J.~Agler, {\em Hypercontractions and subnormality}, J. Operator Theory {\bf 13} (1985), no. 2, 
203--217.    
    
 
\bibitem{ars}
A.~Aleman, S.~Richter and C.~ Sundberg, {\em Beurling's theorem for the Bergman space}, Acta Math. {\bf 177} 
(1996), no. 2, 275310.

\bibitem{ABP}  D.~Alpay, J.A.~Ball, and Y.~Peretz, {\em System 
theory, operator models and scattering:  the time-varying case}, 
J.~Operator Theory \textbf{47} (2002)l 245--286.

\bibitem{AEM} C.-G. Ambrozie, M.~Engli\v{s}, and V.~M\"uller, {\em 
Operator tuples and analytic models over general domains in ${\mathbb 
C}^{n}$}, J.~Operator Theory \textbf{47} (2002), 287--302.

\bibitem{ABFP} C.~Apostol, H.~Bercovici, C.~Foias, and C.~Pearcy, 
{\em Invariant subspaces, dilation theory, and the structure of the 
predual of a dual algebra, I}, J.~Functl.~Anal. \textbf{63} (1985), 
369--404.

   
    
\bibitem{Arv1998} W.~Arveson, {\em Subalgebras of $C^{*}$-algebras, III. 
Multivariable operator theory}, Acta Math.~\textbf{181} (1998), 
159--228.

\bibitem{BBF1} J.A.~Ball, V.~Bolotnikov, and Q.~Fang, {\em Multivariable 
backward-shift-invariant subspaces and observability operators},  
Multidimens.~Syst.~Signal Process.~\textbf{18} (2007) no.~4, 191--248.

\bibitem{BBF2a} J.A.~Ball,V.~Bolotnikov, and Q.~Fang, {\em 
Transfer-function realization for multipliers of the ARveson space}, 
J.~Math.~Anal.~Appl.~\textbf{333} (2007) no.~1, 68--92.

\bibitem{NFRKHS} J.A.~Ball and V.~Vinnikov, \emph{Formal reproducing kernel
Hilbert spaces: the commutative and noncommutative settings}, in {\em
Reproducing Kernel Spaces and Applications} (Ed.~D.~Alpay), pp.
77--134, \textbf{OT 143}, Birkh\"auser, Basel, 2003.

\bibitem{BC} J.A.~Ball and N.~Cohen, {\em de Branges-Rovnyak operator 
models and systems theory:  a survey}, in Topics in Matrix and 
Operator Theory, Rotterdam, 1989,   (Ed.~H.~Bart, I.~Gohberg, and 
M.A.~Kaashoek), pp.~93--136, \textbf{OT 50} 
Birkh\"auser-Verlag, 1991.

\bibitem{Beurling}  A.~Beurling, {\em On two problems concerning 
linear transformations in Hilbert space}, Acta Math.~\textbf{81} 
(1949),239--255.  

\bibitem{dBR1}
{L. de} Branges and J.~Rovnyak,
Canonical models in quantum scattering theory,
in: {\em Perturbation Theory and its
Applications in Quantum Mechanics} (C.~Wilcox, ed.) pp. 295--392,
{Holt, Rinehart and Winston, New York}, 1966.

\bibitem{dBR2}
L.~de Branges and J.~Rovnyak,
{\em Square summable power series},
Holt, Rinehart and Winston, New York, 1966.

\bibitem{CC} S.~Chavan and R.E.~Curto, {\em Operators Cauchy dual to 
2-hyperexpansive operators: the multivariable case}, Integral 
Equations and Operator Theory \textbf{73} (2012), 481--516.





\bibitem{DS} P.~Duren and A.~Shuster, {\em Bergman Spaces}, 
Mathematical Surveys and Monographs \textbf{100}, American 
Mathematical Society, 2004.

\bibitem{OT100} C.~Foias, A.E.~Frazho, I.~Gohberg, and M.A.~Kaashoek, 
{\em Metric Constrained Interpolation, Commutant Lifting and 
Systems}, \textbf{OT 100} Birkh\'auser-Verlag, Basel, 1998.

\bibitem{Potapov} I.~Gohberg and L.A.~Sakhnovich (Ed.), {\em Matrix 
and operator valued functions: the Vladimir Petrovich Potapov 
memorial volume}, \textbf{OT 72} Birkh\"auser-Verlag, Basel, 1994.

\bibitem{Halmos}  P.R.~Halmos, {\em Shifts on Hilbert spaces}, 
J.~Reine Angew.~Math.~\textbf{208} (1961), 102--112.

\bibitem{HKZ}  H.~Hedenmalm, B.~Korenblum, and K.~Zhu, {\em Theory of 
Bergman Spaces}, Graduate Texts in Mathematics \textbf{199}, 
Springer, 2000.

\bibitem{hend1}
H.~Hedenmalm, A factorization theorem for square area-integrable analytic functions, J. Reine Angew. 
Math. 422 (1991), 45--68.



\bibitem{Hor} C.~Horowitz, {\em Factorization theorems for functions 
in the Bergman spaces}, Duke Math.~J.~\textbf{44} (1977), 201--213.



\bibitem{Lax} P.D.~Lax, {\em Translation invariant subspaces}, Acta 
Math.~\textbf{101} (1959), 163--178.

\bibitem{Loehr} N.A.~Loehr, {\em Bijective Combinatorics}, Discrete 
Mathematics and its Applications (Boca Raton), Chapman \& Hall/CRC Press, Boca Raton, 
FL, 2011.



\bibitem{McCT2000}  S.~McCullough and T.T.~Trent, {\em Invariant 
subspaces and Nevanlinna-Pick kernels}, J.~Funct.~Anal.~\textbf{178} 
(2000) no.~1, 226--249.

\bibitem{muller}
V.~M\"uller, {\em Models for operators using weighted shifts}, J. Operator Theory {\bf 20}
(1988), no. 1, 3--20.

\bibitem{MV} V.~M\"uller and F.H.~Vasilescu, Standard models for some 
commuting multioperators, {\em Proc.~Amer.~Math.~Soc.} \textbf{117} 
No.~4 (1993), 979--989.



\bibitem{oljfa} A.~Olofsson, {\em A characteristic operator function
for the class of $n$-hypercontractions}, J.~Funct.~Anal.~\textbf{236}
(2006),  517--545. 

\bibitem{olieot}
A.~Olofsson, {\em An operator-valued Berezin transform and the class of
$n$-hypercontractions}, Integral Equations Operator Theory  {\bf 58} (2007), no. 4,
503--549.

\bibitem{olaa}
A.~Olofsson {\em Operator-valued Bergman inner functions as transfer functions},
Algebra i Analiz  {\bf 19} (2007),  no. 4, 146--173.




\bibitem{RRbook}  M.~Rosenblum and J.~Rovnyak, {\em Hardy classes and 
operator theory}, Oxford Univ.~Press, New York, 1985; corrected 
reprint: Dover Publications, Meneola, NY, 1997. 

\bibitem{sh1}
S.~Shimorin, {\em Wold-type decompositions and wandering subspaces for operators close to isometries}, 
J. Reine Angew. Math. {\bf 531} (2001), 147189.

\bibitem{sh2}
S.~Shimorin, {\em On Beurling-type theorems in weighted $\ell^{2}$ and Bergman spaces},
Proc. Amer. Math. Soc. {\bf 131} (2003), no. 6, 17771787.

\bibitem{Sutton} D.J.~Sutton, {\em Structure of Invariant Subspaces 
for Left-Invertible Operators on Hilbert Space}, Virginia Tech PhD 
dissertation, August, 2010.




\end{thebibliography}
\end{document}